\documentclass[12pt]{article}
\usepackage{amsfonts}
\input{epsf.sty}
\usepackage[dvips]{graphicx}
\usepackage{latexsym,amsmath,amsfonts,amscd, amsthm}
\usepackage{epsfig}
\usepackage{changebar}
\usepackage{pstricks}
\usepackage{pst-plot}
\usepackage{multirow,bigstrut}
\usepackage{subfigure}
\usepackage{placeins}
\usepackage{amssymb}

\usepackage{mathrsfs}
\usepackage{amssymb,amsthm,hyperref,amsmath,graphicx,bm,color,booktabs}
\usepackage{bm}
\usepackage{comment}
\usepackage{enumerate}

\numberwithin{equation}{section} %

\newtheorem{thm}{Theorem}[section]
\newtheorem{defn}{Definition}[section]
\newtheorem{prop}{Proposition}[section]
\newtheorem{lem}{Lemma}[section]

\newtheorem{rem}{Remark}[section]
\theoremstyle{definition}

\newtheorem{assump}{Assumption}[section]

\newcommand{\brac}[1]{\left(#1\right)}
\newcommand{\abs}[1]{\left\vert#1\right\vert}
\newcommand{\norm}[1]{\left\Vert#1\right\Vert}

\newcommand{\ie}{{\it{i.e.}}}

\newcommand{\diag}{\mbox{diag}}
\newcommand{\diff}{\mbox{d}}

\allowdisplaybreaks

\begin{document}

\baselineskip=2pc

\begin{center}
{\bf \large 
Uniform accuracy of implicit-explicit Runge-Kutta methods for linear hyperbolic relaxation systems}
\end{center}

\vspace{.2in}

\centerline{
Zhiting Ma \footnote{Beijing Institute of Mathematical Sciences and Applications, Beijing 101408, China. E-mail: mazt@bimsa.cn}
\qquad
Juntao Huang \footnote{Department of Mathematics and Statistics, Texas Tech University, Lubbock, TX, 79409, USA. E-mail: juntao.huang@ttu.edu. Research is partially supported by NSF DMS-2309655 and DOE DE-SC0023164.}
}

\vspace{.2in}

\vspace{.2in}

\centerline{\bf Abstract}
In this paper, we study the uniform accuracy of implicit-explicit (IMEX) Runge-Kutta (RK) schemes for general linear hyperbolic relaxation systems satisfying the structural stability condition proposed in \cite{yong_singular_1999}. We establish the uniform stability and accuracy of a class of IMEX-RK schemes with spatial discretization using a Fourier spectral method. Our results demonstrate that the accuracy of the fully discretized schemes is independent of the relaxation time across all regimes. Numerical experiments on applications in traffic flows and kinetic theory verify our theoretical analysis.

\vfill

\newpage

\section{Introduction}
\setcounter{equation}{0}
\setcounter{figure}{0}
\setcounter{table}{0}

In this paper, we consider the uniform stability and uniform accuracy of a class of implicit-explicit (IMEX) Runge-Kutta (RK) schemes for one-dimensional linear hyperbolic relaxation systems
\begin{equation}\label{eq:PDE-general}
	\bm{U}_t + \bm{A} \bm{U}_{x} = \frac{1}{\varepsilon}\bm{Q} \bm{U}.
\end{equation}
Here $\bm{U}=\bm{U}(x,t)\in\mathbb{R}^m$, $x \in\mathbb{R}$, $t\geq 0$, $\bm{A} $ and $\bm{Q}$ are two $m\times m$ constant matrices, the subscripts $t$ and $x$ refer to the partial derivatives with respect to $t$ and $x$. The parameter $\epsilon > 0$ is a small value that denotes the relaxation time. 

Stiff systems of the linearized version of such differential equations describe several physical phenomena of great importance in applications.
Examples include kinetic theories (moment closure systems \cite{levermore1996moment,Di2017nm}, discrete-velocity kinetic models \cite{broadwell1964shock,platkowski1988discrete}), nonlinear optics \cite{hanouzet2000approximation}, radiation hydrodynamics \cite{pomraning2005equations,mihalas2013foundations}, traffic flows \cite{aw2000siam}, dissipative relativistic fluid flows \cite{geroch1990dissipative}, chemically reactive flows \cite{giovangigli2012multicomponent}, invisicid gas dynamics with relaxation \cite{zeng1999gas}, etc.

Developing efficient numerical schemes for such systems is challenging, since in many applications the relaxation time varies widely --- from values of order one to values much smaller than the time scale determined by the characteristic speeds of the system. In the latter case, the hyperbolic system with relaxation is said to be stiff. It is typically very difficult to split the problem into separate regimes and to use different solvers in the stiff and non-stiff regions. Recently developed implicit-explicit schemes, including IMEX-RK method (e.g.\cite{ascher1997,kennedy2003additive,dimarco2013siam,pareschi2005jsc}) and IMEX backward differentiation formulas (IMEX-BDF, e.g \cite{ascher1995siam,hundsdorfe2007imex,dimarco2017siam,albi2020implicit-explicit}) method overcome this difficulty, providing basically the same advantages of the splitting schemes, without the limitation of the time step is of order $O(\varepsilon)$.
Additionally, the formalism of the IMEX method can guarantee the order of accuracy when $\varepsilon$ is of order 1 and also when $\varepsilon\ll 1$ (the so-called asymptotic preserving property \cite{jin1999efficient-AP,hu2017asymptotic}). However, in the intermediate regime, many numerical experiments indicate that IMEX-RK schemes often suffer from order reduction  \cite{boscarino2007error-analysis,hu_uniform_2019,hundsdorfe2007imex,dimarco2017siam,albi2020implicit-explicit}. 

For the Jin-Xin model \cite{jin1995cpam} as a specific relaxation system, the uniform stability and accuracy have been studied in \cite{hu_uniform_2019} for the IMEX-BDF schemes and in \cite{hu_uniform_2023} for the IMEX-RK schemes.
In our previous work \cite{MHY2023}, we prove the uniform stability and accuracy of a class of IMEX-BDF schemes discretized spatially by a Fourier spectral method, which illustrates the accuracy of the fully discretized IMEX-BDF schemes is independent of the relaxation time in all regimes.

In this work, we investigate the uniform accuracy of the IMEX-RK schemes for linear hyperbolic relaxation systems \eqref{eq:PDE-general}. Our main contribution is to extend the analysis in \cite{hu_uniform_2019}, which is limited to the Jin–Xin model of size $2 \times 2$, to more general hyperbolic relaxation systems of arbitrary size that satisfy the structural stability condition proposed in \cite{yong_singular_1999}. Since the structural stability condition is tacitly respected by many well-established physical theories \cite{yong_singular_1999,liu2001basic,Yong2008}, our analysis is expected to have broader applicability.

Furthermore, we emphasize that this extension is highly non-trivial. First, the energy estimates must be carried out in a weighted norm induced by the system's symmetrizer, rather than in a standard norm. Second, the error estimate in \cite{hu_uniform_2019} relies on the introduction of the auxiliary error vector to address the low stage order of intermediate stages. Generalizing this technique from Jin-Xin model to hyperbolic relaxation system of arbitrary size \eqref{eq:PDE-general} presents significant challenges, as the numerical solution involves vectors of varying dimensions across stages. Consequently, the analysis must be conducted at the matrix level. To deal with this difficulty, we employ the Kronecker product and the vectorization operator, which requires numerous intricate calculations, see Section \ref{subsec:4-2} for details.

We also mention the difference between this work and our previous work \cite{MHY2023} which focused on IMEX-BDF schemes, a class of multistep method. The key technique in \cite{MHY2023} is the energy estimate with a new multiplier technique. In contrast, our current work deals with IMEX-RK schemes, where the main difficulty lies in handling the multiple intermediate stages of the solution within each time step.

The rest of the paper is organized as follows. In Section \ref{sec:preliminaries}, we recall the structural stability condition, the regularity of the solution, and a class of IMEX-RK schemes for system \eqref{eq:PDE-general}. In Section \ref{sec:uniform-stability}, we establish the uniform stability of a class of IMEX-RK schemes. We prove rigorously uniform second-order accuracy and uniform third-order accuracy of the IMEX-RK schemes in Section \ref{sec:2nd-uniform-accuracy} and Section \ref{sec:3rd-uniform-accuracy}. Numerical experiments are presented in Section \ref{sec:numerical-test} to validate our theoretical findings.

\section{Preliminaries}\label{sec:preliminaries}
In this section, we introduce the structural stability condition \cite{yong_singular_1999}, the regularity of the solution which has been proved in \cite{MHY2023}, and a class of IMEX-RK schemes for system \eqref{eq:PDE-general}.

\subsection{Structural Stability Condition}\label{subsec:ssc}

For system \eqref{eq:PDE-general}, the structural stability condition \cite{yong_singular_1999,Yong2008} reads as
\begin{enumerate}[(i)]
	\item There is an invertible $m\times m$ matrix $\bm{P}$ and an invertible $r\times r$ $(0<r\le m)$ matrix $\hat{\bm{S}}$ such that
	\begin{equation}\nonumber
		\bm{P} \bm{Q} = 
		\left(
		\begin{array}{cc}
			0 & 0 \\
			0 & \hat{\bm{S}}
		\end{array}
		\right)
		\bm{P}.
	\end{equation}

    \item \label{item:stability-2} There exists a symmetric positive-definite (SPD) matrix $\bm{A_0}$ such that
	\begin{equation}\nonumber
		\bm{A_0} \bm{A}  = \bm{A}^T  \bm{A_0}.
	\end{equation}

	\item The hyperbolic part and the source term are coupled in the sense:
	\begin{equation}\nonumber
		\bm{A_0}  \bm{Q}  + \bm{Q}^T  \bm{A_0}  \le -\bm{P}^T 
		\left(
		\begin{array}{cc}
			0 & 0 \\
			0 & \bm{I_r}
		\end{array}
		\right)
		\bm{P} .
	\end{equation}
\end{enumerate}
Here the superscript $T$ denotes the transpose and $\bm{I}_r$ is the unit matrix of order $r$.

About this set of conditions, we remark as follows. 
Condition (i) is classical for initial-value problems of systems of ordinary differential equations (ODE, spatially homogeneous systems). Condition (ii) means the symmetrizable hyperbolicity of the system of first-order partial differential equations (PDE) in \eqref{eq:PDE-general}. Condition (iii) characterizes a kind of coupling between the ODE and PDE parts. 
As shown in \cite{yong_singular_1999,liu2001basic,Yong2008}, the structural stability condition has been tacitly respected by many well-developed physical theories. 
Recently, it has been shown in \cite{Di2017nm,zhao2017stability,ma2023nonrelativisti} to be proper for certain moment closure systems. 
Under the structural stability condition, the existence and stability of the zero relaxation limit of the corresponding initial-value problems have been established in \cite{yong_singular_1999}.

Assuming the structural stability condition, we introduce $\tilde{\bm{U}} := \bm{P} \bm{U}$
and transform system \eqref{eq:PDE-general} into its equivalent version
\begin{equation}\nonumber
	\tilde{\bm{U}}_t + \tilde{\bm{A}} \tilde{\bm{U}}_{x} = \frac{1}{\varepsilon}
	\left(
	\begin{array}{cc}
		0 & 0 \\
		0 & \hat{ \bm{S} }
	\end{array}
	\right)
	\tilde{\bm{U}},
\end{equation}
where $\tilde{\bm{A}} := \bm{P} \bm{A} \bm{P}^{-1}$.
It is easy to see that the above equivalent version satisfies the structural stability condition with $\tilde{\bm{P}}=\bm{I}_m$ and $\tilde{\bm{A}}_0 = \bm{P}^{-T}\bm{A_0}\bm{P}^{-1}$. 
Thus, throughout this paper we only consider the transformed version (drop the tilde)
\begin{equation}\label{eq:PDE}
	\bm{U}_t + \bm{A} \bm{U}_{x} = \frac{1}{\varepsilon}\left(
\begin{array}{cc}
	0 & 0 \\
	0 & \hat{\bm{S}}
\end{array}
\right) \bm{U} \equiv \frac{1}{\varepsilon} \bm{Q} \bm{U} .
\end{equation}

It was proved in \cite{yong_singular_1999} (Theorem 2.2) that $\bm{P}^{-T}\bm{A_0} \bm{P}^{-1}$ is a block-diagonal matrix (with the same partition as in (i) and (iii)).
Thus, the symmetrizer for \eqref{eq:PDE} has the following block-diagonal form
\begin{equation}\nonumber
\bm{A}_0 =
\left(
\begin{array}{cc}
	\bm{A}_{01} & 0 \\
	0 & \bm{A}_{02}
\end{array}
\right).
\end{equation}
We further assume that $\bm{A}_{02}\hat{\bm{S}}$ is symmetric (negative-definite), which holds true for many physical models \cite{Yong2008}.

\subsection{Regularity result}\label{seubsec:regularity-result}
We state the following results which have been proved in our previous work in Theorem 3.2 in \cite{MHY2023}. 
\begin{lem}[\cite{MHY2023}]\label{thm:regularity-const}
	For any integer $s\ge0$, the solution to \eqref{eq:PDE} satisfies
	\begin{enumerate}
		\item 
		for all $t\ge 0$,
		\begin{equation}\label{eq:regularity-Hs}
			\norm{\bm{U}(\cdot,t)}_{H^s}^2 \le C \norm{\bm{U}(\cdot,0)}_{H^s}^2,
		\end{equation}
		
		\item
		for all $t\ge  2\delta_0^{-1} s\varepsilon\log(1/\varepsilon)$,
		\begin{equation}\label{eq:regularity-U}
			\norm{\partial_t^{r_1}\partial_x^{r_2}\bm{U}(\cdot,t)}^2 \le C \norm{\bm{U}(\cdot,0)}_{H^s}^2, \quad r_1+{r_2}\le s
		\end{equation}
		and
		\begin{equation}\label{eq:regularity-W}
			\norm{\partial_t^{r_1}\partial_x^{r_2}\bm{W}(\cdot,t)}^2 \le C \varepsilon^2 \norm{\bm{U}(\cdot,0)}_{H^s}^2, \quad r_1+{r_2}\le s-1.
		\end{equation}	
  Here $\delta_0>0$ is a constant determined by the SPD matrices $\bm{A_{02}}$ and $\bm{A_{02}}\hat{\bm{S}}$, $C$ is a generic constant independent of $\varepsilon$, $\bm{U}=\begin{pmatrix}
      \bm{V}\\\bm{W}
  \end{pmatrix}$ with $\bm{V} \in\mathbb{R}^{m-r}$ and $\bm{W} \in\mathbb{R}^{r}$, and $r_1, r_2$ are non-negative integers.
	\end{enumerate}
\end{lem}
Here and below, $\norm{\cdot}$ denotes the $L^2$ norm, and $\norm{\cdot}_{H^s}$ denotes the standard Sobolev norm in the space $H^s$.

\subsection{IMEX-RK method}\label{esc:IMEX-RK-method}
 
An IMEX-RK scheme applied to the system \eqref{eq:PDE} employs an explicit approach to the non-stiff convection term and an implicit one for addressing the stiff relaxation term \cite{pareschi2005jsc}
\begin{equation}\label{scheme:IMEX-RK}
	\begin{aligned}
		&\bm{U}^{(i)} = \bm{U}^{n} - \Delta t \sum_{j=1}^{i-1} \tilde{h}_{ij} \bm{A} \partial_x \bm{U}^{(j)} + \frac{\Delta t}{\varepsilon}\sum_{j=1}^{i}h_{ij} \bm{Q} \bm{U}^{(j)}, \qquad i = 1, \ldots, s, \\
		&\bm{U}^{n+1} = \bm{U}^{n} - \Delta t \sum_{j=1}^{s} \tilde{b}_{j} \bm{A} \partial_x \bm{U}^{(j)} + \frac{\Delta t}{\varepsilon}\sum_{j=1}^{s}b_{j} \bm{Q} \bm{U}^{(j)}.
	\end{aligned}
\end{equation}
Here $\bm{U}^{n}$ denotes the numerical solution at the time $t^n=n\Delta t$ where $\Delta t$ is the time step size. $s$ is the number of stages. The matrix $\tilde{\bm{H}} = (\tilde{h}_{ij}) \in \mathbb{R}^{s\times s}$ is strictly low-triangular (i.e., $\tilde{h}_{ij}=0$ and $j\geq i$), $\bm{H} = (h_{ij}) \in \mathbb{R}^{s\times s}$ is low-triangular (i.e., $h_{ij}=0$, $j > i$). 
With the vectors $\tilde{\bm{b}} = (\tilde{b}_1, \ldots, \tilde{b}_s)^T$ and $\bm{b} = (b_1, \ldots, b_s)^T$, they can be represented by a double Butcher tableau: 
\begin{equation}\label{tableau:IMEX_RK}
	\begin{tabular}{l|c}
		$\bm{\tilde{c}}$ & $\tilde{\bm{H}}$\\ \hline 
		 & $\tilde{\bm{b}}^T$
	\end{tabular},\qquad 
	\begin{tabular}{l|c}
		$\bm{c}$ & $\bm{H}$\\ \hline
		 & $\bm{b}^T$
	\end{tabular}.
\end{equation}
Here, the vectors $\tilde{\bm{c}} = (\tilde{c}_1, \ldots, \tilde{c}_s)^T$ and $\bm{c} = (c_1, \ldots, c_s)^T$ are define by 
\begin{equation}\nonumber
	\tilde{c}_i = \sum_{j=1}^{i-1}\tilde{h}_{ij}, \qquad c_i = \sum_{j=1}^{i} h_{ij}.
\end{equation}

We assume that $\bm{H}$ has non-negative diagonal elements (i.e., $h_{ii}\ge0$ for $i=1,\ldots,s$)  since this guarantees the solvability of the numerical solution for any $\Delta t>0$ when applied to the linear test equation $y'=\lambda y$ with $\operatorname{Re}(\lambda) < 0$ \cite{prothero1974stability}. It is easy to show that this condition also implies that $\bm{U}^{(i)}$ for $i=1,\ldots,s$ given by \eqref{scheme:IMEX-RK} is well-defined under the structural stability condition in Section \ref{subsec:ssc}.

The tableau \eqref{tableau:IMEX_RK} must satisfy the following standard order conditions \cite{pareschi2005jsc}:
\begin{description}
    \item \textbf{First Order}
\begin{equation}\label{equ:cond-1st-rk}
    \sum_{i=1}^s \tilde{b}_i= \sum_{i=1}^s b_i=1 .
\end{equation}
\item \textbf{Second Order}
\begin{equation}\label{equ:cond-2nd-rk}
    \sum_{i=1}^s \tilde{b}_i \tilde{c}_i= \sum_{i=1}^s b_i c_i=\sum_{i=1}^s \tilde{b}_i c_i=  \sum_{i=1}^s b_i \tilde{c}_i=\frac{1}{2}.
\end{equation}
\item \textbf{Third Order}
\begin{equation}\label{equ:cond-3rd-rk}
    \begin{aligned}
        &\sum_{i,j=1}^s \tilde{b}_i \tilde{h}_{i j} \tilde{c}_j=\sum_{i,j=1}^s b_i h_{i j} c_j=\frac{1}{6}, \quad \sum_{i=1}^s \tilde{b}_i \tilde{c}_i \tilde{c}_i=\sum_{i=1}^s b_i c_i c_i=\frac{1}{3},\\
        &\sum_{i,j=1}^s \tilde{b}_i \tilde{h}_{i j} c_j= \sum_{i,j=1}^s \tilde{b}_i h_{i j} \tilde{c}_j=\sum_{i,j=1}^s \tilde{b}_i h_{i j} c_j=\frac{1}{6}, \\
        &\sum_{i,j=1}^s b_i \tilde{h}_{i j} c_j=\sum_{i,j=1}^s b_i h_{i j} \tilde{c}_j=\sum_{i,j=1}^s b_i \tilde{h}_{i j} \tilde{c}_j=\frac{1}{6}, \\
        &\sum_{i=1}^s \tilde{b}_i c_i c_i=\sum_{i=1}^s \tilde{b}_i \tilde{c}_i c_i=\sum_{i=1}^s b_i \tilde{c}_i \tilde{c}_i=\sum_{i=1}^s b_i \tilde{c}_i c_i=\frac{1}{3} .
    \end{aligned}
\end{equation}
\end{description}

In this paper, we restrict our study to the IMEX-RK schemes of type CK \cite{kennedy2003additive} and with implicitly-stiffly-accurate (ISA) property where the definitions are given as follows:
\begin{defn}[Type CK and Type ARS]\label{defn:CK-ARS}
    The IMEX-RK method given by \eqref{tableau:IMEX_RK} is of type CK if the matrix $\bm{H}$ can be written as
$$
\left(\begin{array}{cc}
0 & 0 \\
\bm{h} & \hat{\bm{H}}
\end{array}\right),
$$
where the vector $\bm{h} \in \mathbb{R}^{s-1}$ and the submatrix $\hat{\bm{H}} \in \mathbb{R}^{(s-1) \times(s-1)}$ is invertible. In particular,   if $\bm{h}=\bm{0}$, $b_1=0$, the scheme is of type ARS.
\end{defn}
\begin{defn}[ISA and GSA]\label{defn:ISA-GSA}
    The IMEX-RK method given by \eqref{tableau:IMEX_RK} is  implicitly-stiffly-accurate (ISA) if $h_{s i}=b_i$, $i=1, \ldots, s$. 
    If further $\tilde{h}_{s i}=\tilde{b}_i$, $i=1, \ldots, s$,  the scheme is said to be globally stiffly accurate (GSA).
\end{defn}

\subsection{Spatial discretization}

We consider the system \eqref{eq:PDE} with periodic boundary conditions. The Fourier-Galerkin spectral method is applied to the semi-discretized IMEX-RK scheme \eqref{scheme:IMEX-RK-1} in the spatial domain:
\begin{equation}\label{scheme:IMEX-RK-1}
	\begin{aligned}
		&(\bm{U}^{(i)})_N = (\bm{U}^{n})_N - \Delta t \sum_{j=1}^{i-1} \tilde{h}_{ij} \bm{A} \partial_x (\bm{U}^{(j)})_N + \frac{\Delta t}{\varepsilon}\sum_{j=1}^{i}h_{ij} \bm{Q} (\bm{U}^{(j)})_N, \quad i = 1, \ldots, s, \\
		&(\bm{U}^{n+1})_N = (\bm{U}^{n})_N - \Delta t \sum_{j=1}^{s} \tilde{b}_{j} \bm{A} \partial_x (\bm{U}^{(j)})_N + \frac{\Delta t}{\varepsilon}\sum_{j=1}^{s}b_{j} \bm{Q} (\bm{U}^{(j)})_N.
	\end{aligned}
\end{equation}
Here $(\bm{U}^n)_N = ((U_1^n)_N, (U_2^n)_N, \ldots, (U_m^n)_N)^T$ where $(U_k^n)_N\in P_N := \textrm{span}\{ e^{ikx}| -N\le k\le N \}$ with $N$ being an integer. We denote $(\bm{U}^0)_N$ as the orthogonal projection of the initial condition $\bm{U}_{in}$ of the system \eqref{eq:PDE} in space $P_N$.
For any function $f\in P_N$, the following inequality holds \cite{hesthaven2007spectral}:  
\begin{equation}\label{equ:spectral}
\norm{\partial_x f}  \le N  \norm{f}.
\end{equation}

\section{Uniform stability}\label{sec:uniform-stability}

In this section, we will prove the uniform stability of the scheme \eqref{scheme:IMEX-RK-1} using the energy method.

For convenience, we rewrite \eqref{scheme:IMEX-RK-1} in the component-wise form
\begin{equation}\label{scheme:IMEX-RK-2}
	\begin{aligned}
		&U^{(i)}_k = U^{n}_k - \Delta t \sum_{j=1}^{i-1} \tilde{h}_{ij} \sum_{l=1}^m a_{kl} \partial_x U_{l}^{(j)} + \frac{\Delta t}{\varepsilon}\sum_{j=1}^{i}h_{ij} \sum_{l=1}^m q_{kl}U^{(j)}_l, \quad k = 1, \ldots, m, \quad i = 1, \ldots, s, \\
		&U^{n+1}_k = U^{n}_k - \Delta t \sum_{j=1}^{s} \tilde{b}_{j} \sum_{l=1}^m a_{kl} \partial_x U^{(j)}_l + \frac{\Delta t}{\varepsilon}\sum_{j=1}^{s}b_{j} \sum_{l=1}^m q_{kl} U^{(j)}_l, \quad k = 1, \ldots, m,
	\end{aligned}
\end{equation}
where $a_{kl}$ and $q_{kl}$ for $k,l=1,2,\ldots,m$ are the entries in the matrices $\bm{A}$ and $\bm{Q}$, respectively. Here and below, we omit the subscript $N$ in the scheme for simplicity.

Define the vectors $\bm{U}_k$ as a collection of the $k$-th component of $\bm{U}$ in $s$ stages:
\begin{equation}\label{equ:vector-u-stages}
    \bm{U}_k = (U^{(1)}_k, \ldots, U^{(s)}_k)^T \in \mathbb{R}^{s \times 1}, \quad k = 1,\ldots,m.
\end{equation}
Then the first $s$ stages in \eqref{scheme:IMEX-RK-2} can be rewritten as
\begin{equation}\label{equ:U-equ}
	\begin{aligned}
		\bm{U}_k =  U^{n}_k \bm{e} - \Delta t \tilde{\bm{H}}\sum_{l=1}^m a_{k l}\partial_x \bm{U}_l + \frac{\Delta t}{\varepsilon} \bm{H} \sum_{l=1}^m q_{kl} \bm{U}_l, \quad k = 1, \ldots, m,
	\end{aligned}
\end{equation}
with $\bm{e} = (1, 1, \ldots, 1)^T \in \mathbb{R}^{s\times 1}$.

We use the energy method to prove the uniform stability of \eqref{scheme:IMEX-RK-1}. Following \cite{hu_uniform_2023}, taking a constant matrix $\bm{M} \in\mathbb{R}^{s\times s}$ to be determined and left multiplying by $\bm{U}_j^T \bm{M}$ on both sides of \eqref{equ:U-equ}, we get a scalar equation
\begin{equation}\nonumber
	\begin{aligned}
		\bm{U}^T_j \bm{M} \bm{U}_k = \bm{U}^T_j \bm{M} U^{n}_k \bm{e}  - \Delta t \bm{U}^T_j \bm{M} \tilde{\bm{H}}\sum_{l=1}^m a_{kl}\partial_x \bm{U}_l + \frac{\Delta t}{\varepsilon} \bm{U}^T_j \bm{M} \bm{H} \sum_{l=1}^m q_{kl} \bm{U}_l,
	\end{aligned}
\end{equation}
which can be further simplified as
\begin{equation}\label{equ:rk-key-equ}
	\begin{aligned}
		U^{(s)}_j U^{(s)}_k = U^{n}_j U^{n}_k - \bm{U}^T_j \bm{M}_\star \bm{U}_k - \Delta t \bm{U}^T_j \bm{M} \tilde{\bm{H}}\sum_{l=1}^m a_{kl}\partial_x \bm{U}_l + \frac{\Delta t}{\varepsilon} \bm{U}^T_j \bm{M} \bm{H} \sum_{l=1}^m q_{kl} \bm{U}_l.
	\end{aligned}
\end{equation}
Here $\bm{M}_\star \in\mathbb{R}^{s\times s}$ is determined by $\bm{M}$:
\begin{equation}\nonumber
	\begin{aligned}
		 \bm{M}_\star = \bm{M}\begin{pmatrix}
			0 & 0 & 0 & \cdots & 0\\
			-1 & 1 & 0 & \cdots & 0\\
			-1 & 0 & 1 & \cdots & 0\\
			\vdots & \vdots & \vdots & \vdots & \vdots \\
			-1 & 0 & 0 & \cdots & 1
		\end{pmatrix} + 
		\begin{pmatrix}
			1 & 0 & 0 & \cdots & 0\\
			0 & 0 & 0 & \cdots & 0\\
			0 & 0 & 0 & \cdots & 0\\
			\vdots & \vdots & \vdots & \vdots & \vdots \\
			0 & 0 & 0 & \cdots & -1
		\end{pmatrix}.
	\end{aligned}
\end{equation}
Note that we use the fact that $U_k^{(1)} = U_k^{n}$ since the scheme is of type CK.

Multiplying \eqref{equ:rk-key-equ} by the symmetrizer $\bm{A}_0 = (a_{0jk})$ in the structural stability condition \eqref{item:stability-2} in section \ref{subsec:ssc}, summing over $j, k$, and integrating over $x$, we can obtain
\begin{equation}\label{equ:err-uniform-stability}
	\begin{aligned}
		\int \sum_{j, k=1}^m a_{0jk} U^{(s)}_j U^{(s)}_k \diff x 
		={}& \int \sum_{j, k=1}^m a_{0jk} U^{n}_j U^{n}_k \diff x \\
        &- \int \sum_{j, k=1}^m a_{0jk} \bm{U}^T_j \bm{M}_\star \bm{U}_k \diff x \\
		&- \Delta t \int \sum_{j, k=1}^m a_{0jk} \bm{U}^T_j \bm{M} \tilde{\bm{H}} \sum_{l=1}^m a_{kl}\partial_x \bm{U}_l \diff x\\
        &+ \frac{\Delta t}{\varepsilon} \int \sum_{j, k=1}^m a_{0jk} \bm{U}^T_j \bm{M} \bm{H} \sum_{l=1}^m q_{kl} \bm{U}_l \diff x\\
        {}& := T_1 + T_2 + T_3 + T_4.
	\end{aligned}
\end{equation}
Since $\bm{A}_0$ is an SPD matrix and $\bm{A}_0 \bm{Q}$ is a symmetric semi-negative-definite matrix from the structural stability condition in section \ref{subsec:ssc}, if further assuming $ \bm{M}_\star $ and $\bm{M} \bm{H} $ are semi-positive-definite matrices, the second term $T_2$ and the fourth term $T_4$ on the RHS of \eqref{equ:err-uniform-stability} may be good terms in the energy estimate.
Therefore, following \cite{hu_uniform_2023}, we assume that, for the matrix $\bm{H}$ in a type CK IMEX-RK scheme, there exists a matrix $\bm{M}$ such that
\begin{assump}\label{assump:M1}

\end{assump}
\begin{itemize}
	\item [($\bm{M1}$)] $\bm{M} \bm{H} + (\bm{M} \bm{H})^T$ is semi-positive-definite and has rank $(s-1)$.
	\item [($\bm{M2}$)] $\bm{M}_\star + \bm{M}_\star^T$ is semi-positive-definite and has rank $(s-1)$.
\end{itemize}
For some widely used IMEX-RK schemes such as ARS(2,2,2) \cite{ascher1997}, ARS(4,4,3) \cite{ascher1997} and BHR(5,5,3)$^\star$ \cite{boscarino2009siam}, \cite{hu_uniform_2023} gives their corresponding $\bm{M}$ matrices and discuss necessary conditions for the existence of $\bm{M}$. For another IMEX-RK scheme ARS(2,3,2) \cite{ascher1997}, it is easy to see the corresponding $\bm{M}$ matrix is the same as for ARS(2,2,2).  

The following lemma from \cite{hu_uniform_2023} will be used in the proof:
\begin{lem}[\cite{hu_uniform_2023}]\label{lemma-M2}
    Under ($\bm{M2}$) in Assumption \ref{assump:M1}, there exists a constant $C_{\bm{M}_\star}>0$ such that
    \begin{equation}\nonumber
        \bm{\xi}^T \bm{M}_\star \bm{\xi} \geq C_{\bm{M}_\star} \sum_{1\leq i < j \leq s} |\xi_i -\xi_j|^2,
    \end{equation}
	for any vector $\bm{\xi} = (\xi_1, \xi_2, \ldots, \xi_s)\in \mathbb{R}^s$. 
\end{lem}

Here we give our main result of this section:
\begin{thm}[Uniform stability]\label{theorem:imex-rk-us}
    Consider the fully discrete scheme \eqref{scheme:IMEX-RK-1}. Assume that the IMEX-RK time discretization is of type CK and ISA, and there exists a matrix $\bm{M}$ satisfying ($\bm{M1}$) and ($\bm{M2}$) in Assumption \ref{assump:M1}. Let $C_{CFL}>0$ be any fixed positive number. Then for any time $T>0$ and $n\in \mathbb{N}^+$ with $n\Delta t \leq T$, there exists a constant $C$ such that
	\begin{equation}\nonumber
		\norm{(\bm{U}^n)_N}^2_{\bm{A}_0} \leq C \norm{(\bm{U}^0)_N}^2_{\bm{A}_0},
	\end{equation}
	with the condition $\Delta t \leq \min(C_{CFL}/N^2, C)$. Here $C$ is a positive constant independent of $\varepsilon$, $N$ and $\Delta t$.
\end{thm}

\begin{proof}	
In the proof, we first estimate the RHS of \eqref{equ:err-uniform-stability}.

We start with the second term on the RHS of \eqref{equ:err-uniform-stability}.
Since $\bm{A}_0\in\mathbb{R}^{m\times m}$ is an SPD matrix, there exists another SPD matrix $\bm{D} = (d_{ij})\in\mathbb{R}^{m\times m}$ such that $\bm{A}_0=\bm{D}^T \bm{D}$.  
Then the second term on the RHS of \eqref{equ:err-uniform-stability} can be estimated as
\begin{equation}
\begin{aligned}\nonumber
    T_2 :=& -\int \sum_{j, k=1}^m a_{0jk} \bm{U}^T_j \bm{M}_\star \bm{U}_k \diff x \\
    =& -\int \sum_{j, k=1}^m \sum_{l=1}^m d_{jl}d_{lk} \bm{U}^T_j \bm{M}_\star \bm{U}_k \diff x \\
    =& -\int \sum_{l=1}^m \brac{\sum_{j=1}^m d_{jl}\bm{U}^T_j} \bm{M}_\star \brac{\sum_{k=1}^m d_{lk} \bm{U}_k} \diff x. 
\end{aligned}
\end{equation}
Setting $\tilde{\bm{U}}_j := \sum_{k=1}^m d_{jk}\bm{U}_k$ for $j=1,2,\ldots,m$, we obtain
\begin{equation}\label{eq:err-stability-second-term-RHS}
\begin{aligned}
    T_2 ={}& -\int \sum_{l=1}^m \tilde{\bm{U}}_l^T \bm{M}_\star \tilde{\bm{U}}_l \diff x  \\
    \leq & -C_{\bm{M}_\star} \sum_{l=1}^m \sum_{1\leq i < j \leq s} \int |\tilde{U}_l^{(i)} - \tilde{U}_l^{(j)}|^2 \diff x \\
    \leq & -C \sum_{l=1}^m \sum_{1\leq i < j \leq s} \int |U_l^{(i)} - U_l^{(j)}|^2 \diff x\\
    \leq & -C_\star \norm{\delta \bm{U}}^2
\end{aligned}
\end{equation}
with $C_\star$ a positive constant.
Here we denote 
\begin{equation}\nonumber
	\norm{\delta \bm{U}}^2 := \sum_{l=1}^m\norm{\delta U_l}^2, \quad  \norm{\delta U_l}^2:= \sum_{1\leq i < j \leq s}\norm{U^{(i)}_l - U_l^{(j)}}^2.
\end{equation} 
In the first inequality, we use Lemma \ref{lemma-M2}; and in the second inequality, we use the equivalence of norms $\norm{\cdot}$ and $\norm{\cdot}_{\bm{A}_0}$ since $\bm{A}_0$ is an SPD matrix.

The third term on the RHS of \eqref{equ:err-uniform-stability} can be estimated as 
\begin{equation}\label{err:rk-u-deri-u}
	\begin{aligned}
		T_3 := &- \Delta t \int \sum_{j, k=1}^m a_{0jk} \bm{U}^T_j \bm{M} \tilde{\bm{H}}\sum_{l=1}^m a_{kl}\partial_x \bm{U}_l \diff x \\
        ={}& - \Delta t \sum_{i, r=1}^s \sum_{j, l=1}^m  \int  U^{(i)}_j (\bm{M}\tilde{\bm{H}})_{i r} (\bm{A}_0\bm{A})_{jl}\partial_x U^{(r)}_l \diff x \\
        ={}& - \Delta t \sum_{i, r=1}^s \sum_{j, l=1}^m  \int  U^{(i)}_l (\bm{M} \tilde{\bm{H}})_{i r} (\bm{A}_0\bm{A})_{lj}\partial_x U^{(r)}_j \diff x \\
        ={}& - \Delta t \sum_{i, r=1}^s \sum_{j, l=1}^m  \int \frac{1}{2} \brac{(\bm{M} \tilde{\bm{H}})_{i r} (\bm{A}_0\bm{A})_{jl}}\brac{U^{(i)}_j\partial_x U^{(r)}_l + U^{(i)}_l\partial_x U^{(r)}_j} \diff x \\
        \leq {}& C \Delta t \sum_{i, r=1}^s \sum_{j, l=1}^m \abs{\int  \left( U^{(i)}_j \partial_x U^{(r)}_l + U^{(i)}_l \partial_x U^{(r)}_j \right ) \diff x}.
	\end{aligned}
\end{equation}
Here, in the second equality, we utilize the invariance of the summation when interchanging the indices $j$ and $l$; in the third equality, we use the fact that $\bm{A}_0 \bm{A}$ is a symmetric matrix.
Then, each term in the result of \eqref{err:rk-u-deri-u} can be rewritten as 
\begin{equation}\label{equ:err-u-periodic}
	\begin{aligned}
		&\bigg| \Delta t \int \left( U^{(i)}_j \partial_x U^{(r)}_l + U^{(i)}_l \partial_x U^{(r)}_j \right ) \diff x \bigg | \\
        ={}& \bigg|  \Delta t \int (U_j^{(i)} - U_j^{(r)}) \partial_x U_l^{(r)} \diff x + \Delta t \int (U_l^{(i)} - U_l^{(r)}) \partial_x U_j^{(r)} \diff x + \Delta t \int (U_j^{(r)} \partial_x U_l^{(r)} + U_l^{(r)} \partial_x U_j^{(r)}) \diff x \bigg| \\
        ={}& \bigg|  \Delta t \int (U_j^{(i)} - U_j^{(r)}) \partial_x U_l^{(r)} \diff x + \Delta t \int (U_l^{(i)} - U_l^{(r)}) \partial_x U_j^{(r)} \diff x + \Delta t \int \partial_x(U_j^{(r)}  U_l^{(r)}) \diff x \bigg| \\
        ={}& \bigg|  \Delta t \int (U_j^{(i)} - U_j^{(r)}) \partial_x U_l^{(r)} \diff x + \Delta t \int (U_l^{(i)} - U_l^{(r)}) \partial_x U_j^{(r)} \diff x \bigg| \\
        \le{}& \bigg|  \Delta t \int (U_j^{(i)} - U_j^{(r)}) \partial_x U_l^{(r)} \diff x \bigg| + \bigg| \Delta t \int (U_l^{(i)} - U_l^{(r)}) \partial_x U_j^{(r)} \diff x \bigg|,
	\end{aligned}
\end{equation}
where the periodic boundary condition is applied in the second to last equality.
Thus, $T_3$ can be estimated by
\begin{equation}\label{equ:err-u-s}
\begin{aligned}
    T_3 \leq {}& \sum_{1\leq i < r \leq s} \sum_{j, l=1}^m \brac{ C_\star \norm{U_j^{(i)} - U_j^{(r)}}^2 + C \Delta t^2 \norm{\partial_x U_l^{(r)}}^2 }\\
    \leq {}& \frac{C_\star}{4} \norm{\delta \bm{U}}^2 + C\sum_{r=1}^s \sum_{l=1}^m \Delta t \norm{U_l^{(r)}}^2
\end{aligned}
\end{equation}
with $C_\star$ in \eqref{eq:err-stability-second-term-RHS}.
Here, in the first inequality, we use Young’s inequality; and in the second inequality, we use the property \eqref{equ:spectral} of Fourier-Galerkin spectral method with the CFL condition:
\begin{equation}\label{eq:fourer-cfl}
    \Delta t \norm{\partial_x U_l^{(r)}}^2 \leq C_{CFL} \norm{U_l^{(r)}}^2.
\end{equation}
The term $\Delta t \norm{U_l^{(r)}}^2$ in \eqref{equ:err-u-s} can be further estimated by
\begin{equation}\label{equ:err-u-s-2}
	\begin{aligned}
		C \Delta t \norm{U_l^{(r)}}^2 \leq 2C \Delta t \brac{\norm{U_l^{(1)}}^2 + \norm{U_l^{(r)} - U_l^{(1)}}^2} = 2C \Delta t \brac{\norm{U_l^{n}}^2 + \norm{U_l^{(r)} - U_l^{(1)}}^2}.
	\end{aligned}
\end{equation}
Take $\Delta t$ sufficiently small such that 
\begin{equation}\nonumber
\begin{aligned}
    C\sum_{r=1}^s \sum_{l=1}^m \Delta t \norm{U_l^{(r)}}^2 \leq {}C \Delta t \norm{\bm{U}^n}^2 + \frac{C_\star}{4} \norm{\delta \bm{U}}^2.
\end{aligned}
\end{equation}
Therefore, we have 
\begin{equation}\nonumber
    T_3 \leq {} C \Delta t \norm{\bm{U}^n}^2 +  \frac{C_\star}{2} \norm{\delta \bm{U}}^2.
\end{equation}

Next, we provide the estimate of the fourth term on the RHS of \eqref{equ:err-uniform-stability}.
According to the structural stability condition in section \ref{subsec:ssc}, $\bm{A}_0 \bm{Q}= \diag(0, - \bar{\bm{S}})$ with $\bar{\bm{S}}=-\bm{A}_{02}\hat{\bm{S}}\in \mathbb{R}^{r\times r}$ is an SPD matrix. Thus there exists a symmetric positive-semidefinite matrix $\bm{K}=(k_{ij})\in \mathbb{R}^{m\times m}$ such that $-\bm{A}_0 \bm{Q} = \bm{K}^T \bm{K}$. Set $\bar{\bm{U}}_i = \sum_{j= 1}^m k_{ij} \bm{U}_j$ with $i=1, \ldots, m$. 
Then using ($\bm{M1}$), the fourth term on the RHS of \eqref{equ:err-uniform-stability} can be estimated by
\begin{equation}\nonumber
	\begin{aligned}
		T_4 := \frac{\Delta t}{\varepsilon} \int \sum_{j, k, l=1}^m a_{0jk} \bm{U}^T_j \bm{M} \bm{H} q_{kl} \bm{U}_l \diff x 
        ={}& \frac{\Delta t}{\varepsilon} \int \sum_{j, l=1}^m \bm{U}^T_j \bm{M} \bm{H} (\bm{A}_0 \bm{Q})_{jl} \bm{U}_l \\
        ={}& -\frac{\Delta t}{\varepsilon} \int \sum_{j, l=1}^m \sum_{i=1}^m \bm{U}^T_j k_{ji}  \bm{M} \bm{H} k_{il}\bm{U}_l\\
        ={}& - \frac{\Delta t}{\varepsilon} \int \sum_{i=1}^m \bar{\bm{U}}_i^T \bm{M} \bm{H} \bar{\bm{U}}_i \leq 0.
	\end{aligned}
\end{equation}

Combining the above estimates for each term in \eqref{equ:err-uniform-stability}, we obtain the following inequality 
\begin{equation}\label{equ:err-us}
	\begin{aligned}
        &\int \sum_{j, k=1}^m a_{0jk} U^{(s)}_j U^{(s)}_k \diff x \leq  \int \sum_{j, k=1}^m a_{0jk} U^{n}_j U^{n}_k \diff x - \frac{C_\star}{2} \norm{\delta \bm{U}}^2 + C \Delta t \norm{\bm{U}^{n}}^2.
	\end{aligned}
\end{equation}
If the IMEX-RK scheme is GSA, then $\bm{U}^{(s)}=\bm{U}^{n+1}$ and we can move to \eqref{3equ:estimate-stability} and the proof is completed. For a general ISA scheme which does not necessarily satisfy the GSA property (such as BHR(5,5,3)*), it will take some efforts to estimate the difference between  $\bm{U}^{(s)}$ and $\bm{U}^{n+1}$.

Using the ISA property, we rewrite the last equation of \eqref{scheme:IMEX-RK-2} as
\begin{equation}\nonumber
	\begin{aligned}
		U^{n+1}_k = U^{(s)}_k - \Delta t \sum_{j=1}^{s} (\tilde{b}_{j}-\tilde{h}_{sj}) \sum_{l=1}^m a_{kl} \partial_x U^{(j)}_l.
	\end{aligned}
\end{equation}
Multiplying the above equation by $2 U^{n+1}_i a_{0ik}$, summing over $i, k$ and integrating over $x$ give
\begin{equation}\nonumber
	\begin{aligned}
 		\int \sum_{i, k=1}^m a_{0ik} U^{n+1}_i U^{n+1}_k \diff x 
		={}& \int \sum_{i, k=1}^m U^{(s)}_i a_{0ik} U^{(s)}_k \diff x  - \int \sum_{i, k=1}^m (U^{n+1}_i - U^{(s)}_i)  a_{0ik} (U^{n+1}_k - U^{(s)}_k) \diff x  \\
		&- 2 \Delta t \int \sum_{i, k=1}^m  U^{n+1}_i a_{0ik} \sum_{j=1}^{s} (\tilde{b}_{j}-\tilde{h}_{sj}) \sum_{l=1}^m a_{kl} \partial_x U^{(j)}_l \diff x.\\
	\end{aligned}
\end{equation}
Combining with \eqref{equ:err-us}, we obtain
\begin{equation}\nonumber
	\begin{aligned}
		&\int \sum_{i, k=1}^m a_{0ik} U^{n+1}_i U^{n+1}_k \diff x \\
		\leq{} &\int \sum_{j, k=1}^m a_{0jk} U^{n}_j U^{n}_k \diff x + C \Delta t \norm{\bm{U}^{n}}^2 - \frac{C_\star}{2} \norm{\delta \bm{U}}^2 - C \norm{\bm{U}^{n+1} - \bm{U}^{(s)}}^2   \\
		&- 2 \Delta t \int \sum_{i, k=1}^m  U^{n+1}_i a_{0ik} \sum_{j=1}^{s} (\tilde{b}_{j}-\tilde{h}_{sj}) \sum_{l=1}^m a_{kl} \partial_x U^{(j)}_l \diff x.
	\end{aligned}
\end{equation}
Applying a similar treatment on the last term in the above equation as \eqref{err:rk-u-deri-u}, we have
\begin{equation}\nonumber
	\begin{aligned}
		&- 2 \Delta t \int \sum_{i, k=1}^m  U^{n+1}_i a_{0ik} \sum_{j=1}^{s} (\tilde{b}_{j}-\tilde{h}_{sj}) \sum_{l=1}^m a_{kl} \partial_x U^{(j)}_l \diff x\\
		\leq{}& C \Delta t \sum_{j=1}^{s} \sum_{i, l=1}^m \abs{ \int U^{n+1}_i \partial_x U^{(j)}_l + U^{n+1}_l \partial_x U^{(j)}_i \diff x }\\
        \leq{}& C \Delta t \sum_{j=1}^{s} \sum_{i, l=1}^m \abs{ \int \left( (U^{n+1}_i - U^{(j)}_i)  \partial_x U^{(j)}_l + (U^{n+1}_l - U^{(j))}_l) \partial_x U^{(j)}_i \right) \diff x }
	\end{aligned}
\end{equation}
and each term can be estimated by 
\begin{equation}\nonumber
	\begin{aligned}
		\Delta t \bigg| \int (U^{n+1}_i - U^{(j)}_i)\partial_x U^{(j)}_l \bigg|\leq C \norm{U^{n+1}_i - U^{(j)}_i}^2 + C \Delta t \norm{U^{(j)}_l}^2,
	\end{aligned}
\end{equation}
where we use Young's inequality and \eqref{eq:fourer-cfl}.
Using the following estimate 
\begin{equation}\nonumber
	\begin{aligned}
		\norm{U^{n+1}_l - U^{(j)}_l}^2 \leq 2(\norm{U^{n+1}_l - U^{(s)}_l}^2 + \norm{U^{(s)}_l - U^{(j)}_l}^2)
	\end{aligned}
\end{equation}
and \eqref{equ:err-u-s-2}, we obtain 
\begin{equation}\nonumber
	\begin{aligned}
		&- 2 \Delta t \int \sum_{i, k=1}^m  U^{n+1}_i a_{0ik} \sum_{j=1}^{s} (\tilde{b}_{j}-\tilde{h}_{sj}) \sum_{l=1}^m a_{kl} \partial_x U^{(j)}_l \diff x\\
		\leq{}& C \Delta t \norm{\bm{U}^{n}}^2 + \frac{C_\star}{4} \norm{\delta \bm{U}}^2 + C \norm{\bm{U}^{n+1} - \bm{U}^{(s)}}^2.
	\end{aligned}
\end{equation}
Therefore, we obtain 
\begin{equation}\label{3equ:estimate-stability}
	\begin{aligned}
		\int  \bm{U}^{n+1} \bm{A}_{0} \bm{U}^{n+1} \diff x \leq \int \bm{U}^{n} \bm{A}_0 \bm{U}^{n} \diff x + C \Delta t \norm{\bm{U}^{n}}^2.
	\end{aligned}
\end{equation}
which immediately implies
\begin{equation}\nonumber
	\norm{(\bm{U}^{n})_{N}}_{\bm{A}_0} \leq \exp(CT) \norm{(\bm{U}^{0})_N}_{\bm{A}_0},
\end{equation}
where we use the fact that $\norm{\bm{U}^{n}}_{\bm{A}_0}$ is equivalent to $\norm{\bm{U}^{n}}$ since $\bm{A}_0$ is an SPD matrix and Gronwall’s inequality. 

\end{proof}

Following \cite{hu_uniform_2023}, the IMEX-RK schemes ARS(2,2,2), ARS(2,3,2), ARS(4,4,3), and BHR(5,5,3)* all satisfy the assumptions in Theorem \ref{theorem:imex-rk-us}. Therefore, they are uniformly stable.
    
\section{Second-order uniform accuracy}\label{sec:2nd-uniform-accuracy}

In this section, we will prove the second-order uniform accuracy of the IMEX-RK scheme.

Following \cite{hu_uniform_2023}, we make the assumption for the matrix $\bm{H}$ in the IMEX-RK scheme of type CK \eqref{tableau:IMEX_RK}:
\begin{assump}\label{assumpt-H}
    The last component of $\bm{v}$ is zero, where $\bm{v}$ is a generator of the one-dimensional null space of $\bm{H}$. 
\end{assump}
This assumption is satisfied for the IMEX schemes of type ARS \cite{hu_uniform_2023} and BHR(5,5,3)* \cite{boscarino2009siam}.
Additionally, this leads to the following lemma:
\begin{lem}(\cite{hu_uniform_2023})\label{lemma:M1A-prop}
Under the assumption ($\bm{M1}$) and Assumption \ref{assumpt-H}, there exists a constant $C_{\bm{M}\bm{H}}>0$ such that
	\begin{equation*}
		\bm{\xi}^T \bm{M} \bm{H} \bm{\xi} \geq C_{\bm{M}\bm{H}} |\xi_s|^2
	\end{equation*}
	for any vector $\bm{\xi}=(\xi_1, \xi_2, \ldots, \xi_s)\in \mathbb{R}^s$. 
\end{lem}

We denote the numerical error at the ${n}$-th time step as
\begin{equation}\nonumber
	\bm{U}_{e}^{n} = (U_{1e}^{n}, ~\ldots, ~U_{ne}^{n})^T := \brac{ (U_{1}^{n})_N - U_1(t_{n}), ~\ldots, ~(U_{m}^{n})_N - U_m(t_{n})}^T ,
\end{equation}
where $(U_{k})_N^{n}$ is the $k$-th component of the numerical solution and $U_k(t_{n})$ is the $k$-th component of the exact solution at $t_{n}$. We say that the initial data is consistent up to order $q$ if $\norm{\bm{U}_{in}}_{H^q}^2\leq C$ 
and the scheme is applied after an initial layer of length $T_0\geq 2\delta_0^{-1}q\varepsilon \log(\frac{1}{\varepsilon})$ with $\delta_0$ referring to Lemma \ref{thm:regularity-const}.

Our main result in this section is stated as follows.
\begin{thm}[Second order uniform accuracy of IMEX-RK schemes]\label{theorem:imex-rk-2order-u-a}
    Consider the fully discrete scheme \eqref{scheme:IMEX-RK-1}. Assume that the IMEX-RK time discretization is of type CK and ISA, and there exists a matrix $\bm{M}$ satisfying ($\bm{M1}$) and ($\bm{M2}$). In addition, further assume
	\begin{itemize}
		\item The IMEX-RK scheme satisfies the standard second-order conditions \eqref{equ:cond-1st-rk}--\eqref{equ:cond-2nd-rk}.
		\item $c_i = \tilde{c}_i, \qquad i =1, \ldots, s$.
		\item Assumption \ref{assumpt-H}.
		\item The initial data is consistent up to order 6.
	\end{itemize}
	Then for any $T>0$ and ${n} \in \mathbb{N}^+$ with ${n} \Delta t \leq T$, we have
	\begin{equation}\label{err:thm-err-4-1}
		\norm{\bm{U}_e^{n}}^2_{\bm{A}_0} \leq C(\Delta t^4 + \frac{1}{N^8}),
	\end{equation}
	with $C$ independent of $\varepsilon$, $N$ and $\Delta t$.
\end{thm}

We note that the IMEX-RK schemes ARS(2,2,2), ARS(2,3,2), ARS(4,4,3), and BHR(5,5,3)* all satisfy the assumptions in Theorem \ref{theorem:imex-rk-2order-u-a}, hence they will exhibit at least second-order uniform accuracy in time.

We will prove Theorem \ref{theorem:imex-rk-2order-u-a} in the rest of this section. 
To begin with, notice that  the error of the initial value is bounded by 
\begin{equation}\nonumber
	\norm{(\bm{U}^0)_{N}-\bm{U}_{in}}^2 \leq \frac{1}{N^8}\norm{\bm{U}_{in}}_{H^4}^2 \leq \frac{C}{N^8}.
\end{equation}
Here we use the property of Fourier projection \cite{hesthaven2007spectral} and the fact that the initial data is consistent up to 4, i.e. $\norm{\bm{U}_{in}}^2_{H^4}\leq C$.

The proof of Theorem \ref{theorem:imex-rk-2order-u-a} is organized as follows:
We first analyze the local truncation error in section \ref{subsec:second-order-local-truncation}. Then, in section \ref{subsec:4-2}, we conduct energy estimates for the error $U_{ke}^{n}$, $k=1,\ldots, m$, in which we introduce the auxiliary error vector $\bm{U}_{ke\star}$ to handle the low stage order of intermediate stages. These energy estimates directly imply the first-order uniform accuracy. Based on this, we first reconsider the energy estimates of $U_{j}$, $j=m-r+1, \ldots, m$. Using Lemma \ref{lemma:M1A-prop} and mathematical induction, we can improve to second-order uniform accuracy. Furthermore, based on the estimate of $U_{j}$, $j=m-r+1, \ldots, m$, the second-order uniform accuracy can be obtained by Gronwall's inequality.

We also comment on the role of the consistency of the initial data in the analysis. To establish first-order uniform accuracy, we assume the initial data is consistent up to order 4. For the second-order accuracy, the proof additionally relies on error estimates for the spatial derivatives of the relevant quantities up to second order, which requires the initial data to be consistent up to order 6.

\subsection{Local truncation error}\label{subsec:second-order-local-truncation}
Using the scheme \eqref{scheme:IMEX-RK}, we define the local truncation errors $\bm{E}^{(i)} = (E_1^{(i)}, E_2^{(i)}, \ldots, E_m^{(i)})^T$ for $i=2, \ldots, s$ and $\bm{E}^{n+1} = (E_1^{n+1}, E_2^{n+1}, \ldots, E_m^{n+1})^T$ as follows:
\begin{equation}\label{equ:truncation-equ}
	\begin{aligned}
		&\bm{U}(t_{n} + c_i \Delta t)= \bm{U}(t_{n}) - \Delta t \sum_{j=1}^{i-1} \tilde{h}_{ij} \bm{A} \partial_x \bm{U}(t_{n} + c_j \Delta t) + \frac{\Delta t}{\varepsilon}\sum_{j=1}^{i}h_{ij} \bm{Q} \bm{U}(t_{n} + c_j \Delta t) - \bm{E}^{(i)}, \\
		&\bm{U}(t_{n} + \Delta t) = \bm{U}(t_{n}) - \Delta t \sum_{j=1}^{s} \tilde{b}_{j} \bm{A} \partial_x \bm{U}(t_{n} + c_j \Delta t) + \frac{\Delta t}{\varepsilon}\sum_{j=1}^{s}b_{j} \bm{Q} \bm{U}(t_{n} + c_j \Delta t) - \bm{E}^{n+1}.
	\end{aligned}
\end{equation}
Here $\bm{U}$ denotes the exact solution and we define $\bm{E}^{(1)}=0$.

We have the following estimates for the local truncation error:
\begin{lem}[Estimates for local truncation error]\label{lem:truncatiom-rk-2nd}
    For a second-order IMEX-RK scheme of type CK with $c_i = \tilde{c}_i$ and assume the initial data is consistent up to order $q\geq 3$. Then, we have the following estimate in the $L^2$ norm:
	\begin{equation}\nonumber
		\bm{E}^{(i)} = O(\Delta t^2),\quad i = 2, \ldots, s, \qquad \bm{E}^{n+1} = O(\Delta t^3)
	\end{equation}
    and the results hold for their $x$-derivatives up to order $q-3$.
\end{lem}
\begin{proof}
	Lemma \ref{thm:regularity-const} shows  
	\begin{equation}\label{equ:truncatiom-regularity}
		\begin{aligned}
			\norm{\partial_t \bm{U}}_{H^2} + \norm{\partial_{tt} \bm{U}}_{H^1} + \norm{\partial_{ttt} \bm{U}} \leq C,\qquad \norm{\partial_t \bm{W}} + \norm{\partial_{tt} \bm{W}} \leq C \varepsilon,
		\end{aligned}
	\end{equation}
	if the initial data is consistent up to order 3. 
	
	Taylor expansion of both sides of the first equation in \eqref{equ:truncation-equ} gives
	\begin{equation}\nonumber
		\begin{aligned}
			&\bm{U}(t_{n}) + c_i \Delta t \partial_t \bm{U}(t_{n}) +O(\Delta t^2 \partial_{tt} \bm{U}) \\
			={}& \bm{U}(t_{n}) - \Delta t \sum_{j=1}^{i-1} \tilde{h}_{ij} \bm{A} \partial_x \bm{U}(t_{n}) + O(\Delta t^2 \partial_t\partial_x \bm{U})\\
            &+ \frac{\Delta t}{\varepsilon}\sum_{j=1}^{i}h_{ij} \hat{\bm{S}} \bm{W}(t_{n}) +  O\brac{\frac{\Delta t^2}{\varepsilon} \partial_t \bm{W}}- \bm{E}^{(i)}.\\
		\end{aligned}
	\end{equation}
	Since $\bm{U}$ satisfies the equation \eqref{eq:PDE}, we see that the $O(\Delta t)$ terms are cancelled due to $c_i = \sum_{j=1}h_{ij} = \sum_{j=1}\tilde{h}_{ij}$. Therefore, we get $\bm{E}^{(i)} = O(\Delta t^2)$ with the aid of estimates \eqref{equ:truncatiom-regularity}.
	
	Taylor expansion of both sides of the second equation in \eqref{equ:truncation-equ} gives
	\begin{equation}\nonumber
		\begin{aligned}
			&\bm{U}(t_{n}) + \Delta t \partial_t \bm{U}(t_{n}) + \frac{1}{2}\Delta t^2\partial_{tt}\bm{U}(t_{n}) +  O(\Delta t^3 \partial_{ttt} \bm{U}) \\
            ={}& \bm{U}(t_{n}) - \Delta t \sum_{j=1}^{s} \tilde{b}_{j} \bm{A} \brac{\partial_x \bm{U}(t_{n}) + c_j \Delta t \partial_t\partial_x \bm{U}(t_{n})} + O(\Delta t^3 \partial_{tt}\partial_x \bm{U})\\
            &+ \frac{\Delta t}{\varepsilon}\sum_{j=1}^{s}b_{j}  \hat{\bm{S}}\brac{ \bm{W}(t_{n}) + c_j \Delta t \partial_t \bm{W}(t_{n})} + O \brac{\frac{\Delta t^3}{\varepsilon} \partial_{tt} \bm{W}} - \bm{E}^{n+1}.
		\end{aligned}
	\end{equation}
	Similar as before, the order conditions \eqref{equ:cond-1st-rk} and \eqref{equ:cond-2nd-rk} show that all the $O(\Delta t)$ and $O(\Delta t^2)$ terms are cancelled. Therefore, we get $\bm{E}^{n+1} = O(\Delta t^3)$ with the aid of estimates \eqref{equ:truncatiom-regularity}.

 Since the equation \eqref{eq:PDE} is linear, the results for their $x-$derivatives can be obtained similarly.
\end{proof}

\subsection{Energy estimates for the error}\label{subsec:4-2}

Before deriving the energy estimates, we state the definitions of two operators: vectorization and Kronecker product, and list some properties that will be used later.
\begin{defn}[Vectorization]
    The vectorization of a matrix is a linear transformation which converts the matrix into a vector. Specifically, the vectorization of a $m\times n$ matrix $\bm{A}=(a_{ij})$, denoted by $vec(\bm{A})$, is a $mn \times 1$ column vector obtained by stacking the columns of the matrix $\bm{A}$ on top of one another: 
    \begin{equation}\nonumber
        vec(\bm{A}) = (a_{11}, \cdots, a_{m1}, a_{12}, \cdots, a_{m2}, \cdots, a_{1n}, \cdots, a_{mn})^T.
    \end{equation}
\end{defn}
\begin{defn}[Kronecker product]
    The Kronecker product, denoted by $\otimes$, is an operation on two matrices of arbitrary size resulting in a block matrix.  If $\bm{A}=(a_{ij})$ is $m \times n$ matrix and $\bm{B}$ is $p \times q$ matrix, then the Kronecker product $\bm{A} \otimes \bm{B}$ is $pm \times qn$ block matrix: 
    \begin{equation*}
        \bm{A} \otimes \bm{B} = \begin{pmatrix}
            a_{11}\bm{B} & \cdots & a_{1n}\bm{B}\\
            \ddots & \vdots & \ddots\\
            a_{m1}\bm{B} & \cdots & a_{mn}\bm{B}
        \end{pmatrix}.
    \end{equation*}
\end{defn}

\begin{prop}[\cite{Graham1981Kronecker}]
\begin{enumerate}
\item For $\bm{A}=(a_{ij})\in \mathbb{R}^{p\times q}$, $\bm{B}=(b_{kl})\in \mathbb{R}^{r\times s}$, 
\begin{equation}\label{prop: Kronecker-product-1}
    (\bm{A}\otimes \bm{B})_{(i-1)p+k,(j-1)q+l} = a_{i j}b_{k l}.
\end{equation}
\item For $\bm{A}\in \mathbb{R}^{p\times q}$, $\bm{B}\in \mathbb{R}^{q\times r}$, $\bm{C}\in \mathbb{R}^{r\times s}$
\begin{equation}\label{prop: vec}
    vec(\bm{A}\bm{B}\bm{C})=(\bm{C}^T\otimes \bm{A})vec(\bm{B}).
\end{equation} 

\item For $\bm{A}\in \mathbb{R}^{p\times q}, \bm{C}\in \mathbb{R}^{q\times k}, \bm{B}\in \mathbb{R}^{r\times s}, \bm{D}\in \mathbb{R}^{s\times l}$
\begin{equation}\label{prop: Kronecker-product}
    (\bm{A}\otimes \bm{B})(\bm{C}\otimes \bm{D}) = \bm{A} \bm{C} \otimes \bm{B}\bm{D}.
\end{equation}

\item
$\bm{A}\otimes \bm{B}$ is invertible if and only if both $\bm{A}$ and $\bm{B}$ are invertible, and
\begin{equation}\label{prop: invertible-Kronecker-product}
    (\bm{A}\otimes \bm{B})^{-1} =\bm{A}^{-1}\otimes \bm{B}^{-1}.
\end{equation}

\end{enumerate}
\end{prop}

Similar as the proof of uniform stability in Section \ref{sec:uniform-stability}, we introduce the vectors $\bm{U}_k(t_{n})$ and $\bm{E}_k$ as a collection of the $k$-th component of the exact solution and the truncation error in all the stages:
\begin{equation}\label{equ:vec-U-nt}
    \bm{U}_k(t_{n}) := (U_k(t_{n} + c_1 \Delta t), \ldots, U_k(t_{n} + c_s \Delta t))^T, 
    \quad
    \bm{E}_k := (E_k^{(1)}, \ldots, E_k^{(s)})^T, \quad k=1,\ldots, m.
\end{equation}
where $U_k$ is the exact solution and $E_k^{(i)}$ is the truncation error defined in \eqref{equ:truncation-equ}.
Then, we rewrite \eqref{equ:truncation-equ} in the component-wise form
\begin{equation}\nonumber
	\begin{aligned}
		&U_k(t_n + c_i \Delta t) ={} U_k(t_n) - \Delta t \sum_{j=1}^{i-1} \tilde{h}_{ij} \sum_{l=1}^m A_{kl} \partial_x U_l(t_n + c_j \Delta t) \\
        &\qquad\qquad\qquad + \frac{\Delta t}{\varepsilon}\sum_{j=1}^{i}h_{ij} \sum_{l=1}^m Q_{kl}U_l(t_n + c_j \Delta t)-E_k^{(i)}, \quad k = 1, \ldots, m, \quad i = 1, \ldots, s,\\[2mm]
		&U_k(t_n + \Delta t) ={} U_k(t_n) - \Delta t \sum_{j=1}^{s} \tilde{b}_{j} \sum_{l=1}^m A_{kl} \partial_x U_l(t_n + c_j \Delta t) 
        + \frac{\Delta t}{\varepsilon}\sum_{j=1}^{s}b_{j} \sum_{l=1}^m Q_{kl}U_l(t_n + c_j \Delta t)-E_k^{n+1}.
	\end{aligned}
\end{equation}
Since $c_1=0$, we can write the first set of equations as 
\begin{equation}\nonumber
	\begin{aligned}
		\bm{U}_k(t_{n}) = U_k(t_{n})\bm{e} - \Delta t\tilde{\bm{H}}\sum_{l=1}^m a_{kl}\partial_x \bm{U}_l + \mu \bm{H} \sum_{l=1}^m q_{kl}\bm{U}_l - \bm{E}_k, \quad k = 1, \cdots, m.
	\end{aligned}
\end{equation}
Here and below, we use the notation
\begin{equation}\nonumber
    \mu := \frac{\Delta t}{\varepsilon}.
\end{equation}
Denote
\begin{equation}\nonumber
	\bm{U}_{ke} := \bm{U}_k - \bm{U}_k(t_{n}), \quad k = 1, \cdots, m,
\end{equation}
as the vector of numerical error in the $n$-th time step (with $\bm{U}_k$ given in \eqref{equ:vector-u-stages} and $\bm{U}_k(t_n)$ given in \eqref{equ:vec-U-nt}. Thus using equations \eqref{equ:U-equ} and \eqref{equ:truncation-equ}, we can obtain
\begin{equation}\label{equ:vec-u-ke}
	\begin{aligned}
		\bm{U}_{ke} = U_{ke}^{(1)}\bm{e} - \Delta t\tilde{\bm{H}}\sum_{l=1}^m a_{kl}\partial_x \bm{U}_{le} + \mu \bm{H} \sum_{l=1}^m q_{kl}\bm{U}_{le} + \bm{E}_k.
	\end{aligned}
\end{equation}

We now absorb the error vectors by introducing the auxiliary error vector $\bm{U}_{ke_\star}$ which satisfies
\begin{equation}\label{equ:u-star-def-1}
	\bm{E}_k = (\bm{U}_{ke} - \bm{U}_{ke_\star}) - \mu \bm{H} \sum_{l=1}^m q_{kl}(\bm{U}_{le} - \bm{U}_{le_\star}).
\end{equation}
\begin{rem}
The purpose of introducing the auxiliary error vector $\bm{U}_{ke_\star}$ is to address the low stage order of intermediate stages. It will become clear later when we rewrite the equation for $\bm{U}_{ke}$ in \eqref{equ:vec-u-ke} as an equation for $\bm{U}_{ke\star}$ in \eqref{equ:energy-estimate-Vstar}. By comparing \eqref{equ:vec-u-ke} and \eqref{equ:energy-estimate-Vstar}, we observe that an extra factor of $\Delta t$ appears before the residual terms in \eqref{equ:energy-estimate-Vstar}. This additional factor facilitates the error estimate in Section \ref{subsubsec:1st-order} and Section \ref{subsubsec:2nd-order}. 
\end{rem}

Next we analyze the relationship between $\bm{E}_{k}$ and $(\bm{U}_{le} - \bm{U}_{le_\star})$ to illustrate that $\bm{U}_{ke_\star}$ is well-defined from \eqref{equ:u-star-def-1}.
We introduce two matrices
\begin{equation}\label{def:e-v-matrices}
    \bm{E} := (\bm{E}_1, \bm{E}_2, \ldots, \bm{E}_m)\in \mathbb{R}^{s\times m},\quad \bm{V} := (\bm{U}_{1e} - \bm{U}_{1e\star}, \bm{U}_{2e} - \bm{U}_{2e\star}, \ldots, \bm{U}_{me} - \bm{U}_{me\star})\in \mathbb{R}^{s\times m}.
\end{equation}
Then \eqref{equ:u-star-def-1} can be rewritten in component-wise form: 
\begin{equation}
	\begin{aligned}\nonumber
		&e_{ik} = v_{ik} - \mu \sum_{j=1}^s\sum_{l=1}^m h_{ij}q_{kl} v_{jl}, \quad i = 1,\ldots,s, \quad k = 1,\ldots,m,
	\end{aligned}
\end{equation}
where $e_{ij}$ and $v_{ij}$ for $i=1,\ldots,s$ and $j=1,\ldots,m$ are the entries in the matrices $\bm{E}$ and $\bm{V}$, respectively.
Multiplying the above equation by $a_{0kr}$, i.e., the entry in the matrix $\bm{A}_0$, we obtain 
\begin{equation}\nonumber
	\begin{aligned}
		e_{ik}a_{0kr} = v_{ik}a_{0kr} - \mu \sum_{j=1}^s\sum_{l=1}^m h_{ij}v_{jl} q_{kl}a_{0kr}.
	\end{aligned}
\end{equation}
Summing over $k$ gives
\begin{equation}\label{equ:ea0-defi}
	\begin{aligned}
		\bm{E} \bm{A}_0 = \bm{V}\bm{A}_0 - \mu \bm{H} \bm{V} (\bm{A}_0 \bm{Q})^T,
	\end{aligned}
\end{equation}
where we use the fact that $\bm{A}_0$ is a symmetric matrix.

Applying vectorization operation $vec(\cdot)$ for equation \eqref{equ:ea0-defi} and using the property \eqref{prop: vec}, we can obtain
\begin{equation}\nonumber
	\begin{aligned}
		vec(\bm{E}\bm{A}_0) ={}& vec(\bm{V}\bm{A}_0) - \mu vec(\bm{H} \bm{V} (\bm{A}_0 \bm{Q})^T)\\
		={}& \brac{\bm{A}_0 \otimes \bm{I}_{s}} vec(\bm{V}) - \mu ((\bm{A}_0 \bm{Q}) \otimes \bm{H}) vec(\bm{V})\\
		={}& \brac{\bm{A}_0 \otimes \bm{I}_{s} - \mu (\bm{A}_0 \bm{Q}) \otimes \bm{H} }vec(\bm{V}).
	\end{aligned}
\end{equation}
Here $\bm{I}_{s}$ is a unit matrix with order $s\times s$. 
Since $\bm{A}_0$ is an SPD matrix and $\bm{A}_0 \bm{Q}$ is a symmetric matrix, there exists an invertible matrix $\bm{G}\in \mathbb{R}^{m\times m}$ and a diagonal matrix $\bm{\Lambda}\in \mathbb{R}^{m\times m}$ such that
\begin{equation}\nonumber
    \bm{A}_0 = \bm{G}^T \bm{G}, \qquad \bm{A}_0 \bm{Q} = \bm{G}^T \bm{\Lambda} \bm{G}.
\end{equation}
Here we briefly explain the constructions of $\bm{G}$ and $\bm{\Lambda}$. In fact, since $\bm{A}_0$ is an SPD matrix, there exists an invertible matrix $\bm{\tilde{G}}\in \mathbb{R}^{m\times m}$
such that
\begin{equation}\nonumber
    \bm{A}_0 = \bm{\tilde{G}}^T \bm{\tilde{G}}.
\end{equation}
For the symmetric matrix $\bm{\tilde{G}}^{-T} \bm{A}_0 \bm{Q}\bm{\tilde{G}}^{-1}$, there exist an orthogonal matrix $\bm{\bar{G}}\in \mathbb{R}^{m\times m}$ and a diagonal matrix $\bm{\Lambda}\in \mathbb{R}^{m\times m}$ such that
\begin{equation}\nonumber
    \bm{\tilde{G}}^{-T} \bm{A}_0 \bm{Q}\bm{\tilde{G}}^{-1} =\bm{\bar{G}}^T \bm{\Lambda} \bm{\bar{G}}.
\end{equation}
Set $\bm{G} := \bm{\bar{G}}\bm{\tilde{G}}$. Since $\bm{\bar{G}}$ is an orthogonal matrix, it is easy to check 
\begin{equation}\nonumber
    \bm{A}_0 = (\bm{\bar{G}}^{-1}\bm{G})^{T}\bm{\bar{G}}^{-1}\bm{G}=\bm{G}^{T}\bm{\bar{G}}^{-T}\bm{\bar{G}}^{-1}\bm{G}=\bm{G}^T \bm{G}, \qquad \bm{A}_0 \bm{Q} = \bm{G}^T \bm{\Lambda} \bm{G}.
\end{equation} 

Since congruent transformation does not change the negative index of inertia of $\bm{A}_0 \bm{Q}$, $\bm{\Lambda}$ is semi-negative definite.
Using the Kronecker product property \eqref{prop: Kronecker-product}, we have 
\begin{equation}\nonumber
	\begin{aligned}
	    &\bm{A}_0 \otimes \bm{I}_{s} - \mu (\bm{A}_0 \bm{Q}) \otimes \bm{H} \\
        ={}& \bm{G}^T \bm{G} \otimes \bm{I}_{s} - \mu (\bm{G}^T \bm{\Lambda} \bm{G}) \otimes H\\
        ={}& (\bm{G}^T \otimes \bm{I}_{s})(\bm{G} \otimes \bm{I}_{s}) - \mu (\bm{G}^T \otimes \bm{I}_{s})(\bm{\Lambda} \otimes \bm{H})(\bm{G} \otimes \bm{I}_{s}) \\
        ={}& (\bm{G}^T \otimes \bm{I}_{s})\brac{\bm{I}_{ms} - \mu (\bm{\Lambda} \otimes \bm{H})}(\bm{G} \otimes \bm{I}_{s}).
	\end{aligned}
\end{equation}
From the property \eqref{prop: invertible-Kronecker-product}, $(\bm{G} \otimes \bm{I}_{s})$ is an invertible matrix. 
Thus, we have 
\begin{equation}\nonumber
	\begin{aligned}
	    vec(\bm{V}) &= \left(\bm{A_0} \otimes \bm{I}_{s} - \mu (\bm{A_0} \bm{Q}) \otimes \bm{H}\right)^{-1} vec(\bm{E}\bm{A_0}) \\
        &= \left(\bm{A_0} \otimes \bm{I}_{s} - \mu (\bm{A_0} \bm{Q}) \otimes \bm{H}\right)^{-1} \left(\bm{A_0} \otimes \bm{I}_{s}\right)vec(\bm{E})\\
        &=(\bm{G} \otimes \bm{I}_{s})^{-1}\left(\bm{I}_{ms} - \mu (\bm{\Lambda} \otimes \bm{H})\right)^{-1}(\bm{G}^T \otimes \bm{I}_{s})^{-1}(\bm{G}^T \otimes \bm{I}_{s})(\bm{G} \otimes \bm{I}_{s})vec(\bm{E})\\
        &=(\bm{G} \otimes \bm{I}_{s})^{-1}\left(\bm{I}_{ms} - \mu (\bm{\Lambda} \otimes \bm{H})\right)^{-1}(\bm{G} \otimes \bm{I}_{s})vec(\bm{E}).
	\end{aligned}
\end{equation}
Define the matrix $\bm{\Phi}$:
\begin{equation}\label{eq:def-matrix-Phi}
    \bm{\Phi} = (\bm{G} \otimes \bm{I}_{s})^{-1}\left(\bm{I}_{ms} - \mu (\bm{\Lambda} \otimes \bm{H})\right)^{-1}(\bm{G} \otimes \bm{I}_{s})
\end{equation}    
we get
\begin{equation}\label{relation-eu}
    vec(\bm{V}) = \bm{\Phi} vec(\bm{E}).
\end{equation}
Combining the above relations with the definitions of $\bm{E}$ and $\bm{V}$ in \eqref{def:e-v-matrices}, we can absorb the error vectors by introducing $\bm{U}_{ke\star}$ for $k=1,\ldots, m$. The use of the auxiliary error vector $\bm{U}_{ke\star}$ is the main technique for addressing the low stage order of intermediate stages.

Here we use the following lemma to state the validity of the definition of $\bm{U}_{\star}$ in \eqref{equ:u-star-def-1}. 
\begin{lem}\label{lem:ustar-well-definied}
    The vector $\bm{U}_{\star}=\brac{\bm{U}_{1\star}, \bm{U}_{2\star}, \ldots, \bm{U}_{m\star}}^T$ is well-defined, as given by \eqref{equ:u-star-def-1}.
\end{lem}
\begin{proof}
    By \eqref{relation-eu}, $\bm{U}_{\star}$ is well-defined if and only if the matrix $\bm{\Phi}$ is invertible. Since $(\bm{G} \otimes \bm{I}_{s})$ is an invertible matrix, we only need to prove $\bm{I}_{ms} - \mu (\bm{\Lambda} \otimes \bm{H})$ is invertible. Because $\bm{H}$ is a lower triangular with non--negative diagonal elements and $\bm{\Lambda}$ is a diagonal semi--negative definite matrix,  $\bm{\Lambda} \otimes \bm{H}$ is also a lower triangular matrix with non--positive diagonal elements. Thus $\bm{I}_{ms} - \mu (\bm{\Lambda} \otimes \bm{H})$ is a lower triangular matrix and the diagonal elements are non-negative for $\mu >0$. Hence, matrix  $\bm{\Phi}$ is invertible, $\ie$ $\bm{U}_{\star}$ is well-defined.
\end{proof}

For the matrix $\bm{\Phi}$ defined in \eqref{eq:def-matrix-Phi}, we have the following results:   
\begin{lem} \label{lemma:muH}
    Let $\mu>0$. Then in component-wise, we have
	\begin{equation}
        \begin{aligned}
		  \bm{\Phi} &= O(1),\\
          \mu \left((\bm{A}_0 \bm{Q}) \otimes \bm{H}\right)\bm{\Phi} &= O(\frac{\mu}{1+\mu}).
        \end{aligned}
	\end{equation}
\end{lem}
\begin{proof}
    Since $\bm{\Phi} = (\bm{G} \otimes \bm{I}_{s})^{-1}\left(\bm{I}_{ms} - \mu (\bm{\Lambda} \otimes \bm{H})\right)^{-1}(\bm{G} \otimes \bm{I}_{s})$ and $\bm{G} \otimes \bm{I}_{s}$ is independent on $\mu$, we only need to analyse matrix $\left(\bm{I}_{ms} - \mu (\bm{\Lambda} \otimes \bm{H})\right)^{-1}$. Set 
    \begin{equation*}
        \tilde{\bm{D}} := \diag(0, h_{22}, \ldots, h_{ss})    
    \end{equation*}
    as the diagonal part of $\bm{H}$, and $\bm{K}$ as the remaining strictly lower triangular part of $\bm{H}$:
    \begin{equation*}
        \bm{K} := \bm{H} - \tilde{\bm{D}}.
    \end{equation*}
    Denote
    \begin{equation}\nonumber
        \begin{aligned}
            \bm{D} := \bm{I}_{ms}  - \mu \bm{\Lambda} \otimes \tilde{\bm{D}}, \qquad \bm{L} := \bm{I}_{ms}  -\mu \bm{\Lambda} \otimes \bm{H} - \bm{D} = -\mu \bm{\Lambda} \otimes \bm{K}.
        \end{aligned}
    \end{equation}
    Note that $\bm{D}$ is also a diagonal $ms\times ms$ matrix and $\bm{L}$ is a strictly lower triangular $ms\times ms$ matrix since $\bm{\Lambda}$ is a diagonal matrix.
    Then we compute
    \begin{equation}\label{equ:phi-ana-form}
        \begin{aligned}
            (\bm{I}_{ms}  - \mu \bm{\Lambda}\otimes \bm{H})^{-1}
            ={}& (\bm{D}+\bm{L})^{-1} \\
            ={}& (\bm{I}_{ms}  + \bm{D}^{-1}\bm{L})^{-1}\bm{D}^{-1} \\
            ={}& \big(\bm{I}_{ms} -\mu \bm{D}^{-1}(\bm{\Lambda}\otimes \bm{K}) \big)^{-1}\bm{D}^{-1} \\
            ={}& \brac{\bm{I}_{ms}  + \sum_{i=1}^\infty\brac{\mu \bm{D}^{-1}(\bm{\Lambda}\otimes \bm{K})}^i } \bm{D}^{-1}\\
            ={}& \brac{\bm{I}_{ms}  + \sum_{i=1}^{ms-1}\brac{\mu\bm{D}^{-1}(\bm{\Lambda}\otimes \bm{K})}^i} \bm{D}^{-1}.
        \end{aligned}
    \end{equation}

Next, we proceed to analyze the above result. For any $1 \leq  \alpha, \beta \leq ms\in \mathbb{N}^+$, there exist $1\leq p,q \leq m\in \mathbb{N}^+$ and $1\leq i, j \leq s\in \mathbb{N}^+$ such that $\alpha = (p-1)s+i$ and $\beta = (q-1)s+j$.
Thus the element of $\mu \bm{D}^{-1}(\bm{\Lambda}\otimes \bm{K})$ can be rewritten as 
\begin{equation}\nonumber
    \begin{aligned}
           \mu \brac{\bm{D}^{-1}(\bm{\Lambda}\otimes \bm{K})}_{\alpha,\beta} ={}& \mu \brac{\bm{D}^{-1}(\bm{\Lambda}\otimes \bm{K})}_{(p-1)s+i,(q-1)s+j} \\
           ={}& \mu \sum_{r=1}^m\sum_{t=1}^s (\bm{D}^{-1})_{(p-1)s+i,(r-1)s+t}(\bm{\Lambda}\otimes \bm{K})_{(r-1)s+t,(q-1)s+j}.
    \end{aligned}
\end{equation}
Using the proposition \eqref{prop: Kronecker-product-1}, we can obtain
\begin{equation}\nonumber
    \begin{aligned}
            (\bm{D})_{(p-1)s+i,(r-1)s+t}={}& (\bm{I}_{ms}  - \mu \bm{\Lambda} \otimes \tilde{\bm{D}})_{(p-1)s+i,(r-1)s+t}\\
            ={}& \delta^{(p-1)s+i}_{(r-1)s+t} - \mu (\bm{\Lambda} \otimes \tilde{\bm{D}})_{(p-1)s+i,(r-1)s+t} \\
            ={}& \delta^{(p-1)s+i}_{(r-1)s+t} - \mu \Lambda_{p,r} \tilde{d}_{i,t}\\
    \end{aligned}
\end{equation}
and 
\begin{equation}\nonumber
    \begin{aligned}
            (\bm{\Lambda}\otimes \bm{K})_{(r-1)s+t,(q-1)s+j}=\Lambda_{r,q} k_{t, j}.
    \end{aligned}
\end{equation}
Here $\delta_i^j$ is Kronecker delta function.
Since $1\leq p,r \leq m\in \mathbb{N}^+$ and $1\leq i, t \leq s\in \mathbb{N}^+$, simple calculation shows that $(p-1)s+i = (r-1)s+t$ if and only if $p=r$ and $i=t$. This means that $\delta^{(p-1)s+i}_{(r-1)s+t}=\delta^{p}_{r}\delta^{i}_{t}$.

Since $\tilde{\bm{D}}$ and $\bm{\Lambda}$ are diagonal matrices, this means that $\Lambda_{p,r}=\Lambda_{p,p}\delta_r^p$ and $ \tilde{d}_{i,t}=\tilde{d}_{i,i}\delta_t^i$. Thus we know the elements of the diagonal matrix $\bm{D}^{-1}$ have the following form
\begin{equation*}
    (\bm{D}^{-1})_{(p-1)s+i,(r-1)s+t} = \frac{\delta^{p}_{r}\delta^{i}_{t}}{1-\mu\Lambda_{p,p}\tilde{d}_{i,i}}.
\end{equation*}
Therefore,  when $\alpha = (p-1)s+i$, $\beta = (q-1)s+j$, we get
\begin{equation}\nonumber
    \begin{aligned}
           \mu \brac{\bm{D}^{-1}(\bm{\Lambda}\otimes \bm{K})}_{\alpha,\beta}= \sum_{r=1}^m\sum_{t=1}^s \frac{\delta^{p}_{r}\delta^{i}_{t}} {1 - \mu \Lambda_{p,p} \tilde{d}_{i,i} } \mu \Lambda_{q,q}\delta_q^r k_{t, j}= \frac{\mu  \Lambda_{q,q}\delta_q^p k_{i, j}} {1 - \mu \Lambda_{p,p}\tilde{d}_{i,i}}= O(\frac{\mu}{1+\mu})\leq O(1)
    \end{aligned}
\end{equation}
and $\bm{D}^{-1} = O(1)$. We notice that the notation $O(\frac{\mu}{1+\mu})\leq O(1)$ is not standard, but is used here just for brevity. It means that $\frac{\mu}{1+\mu}$ grows no faster than 1 up to a constant factor as $\mu\rightarrow a$, where $0\le a\le +\infty$ (a finite non-negative number or $+\infty$). This is evidently true.
Therefore, we have $\bm{\Phi} = O(1)$.

Furthermore, we have
\begin{equation}\nonumber
    \begin{aligned}
        &(\mu (\bm{A}_0 \bm{Q}) \otimes \bm{H})\bm{\Phi} \\
        ={}& \mu (\bm{G}^T \otimes \bm{I}_{s})(\bm{\Lambda} \otimes \bm{H})(\bm{G} \otimes \bm{I}_{s})\bm{\Phi} \\
        ={}& \mu (\bm{G}^T \otimes \bm{I}_{s})(\bm{\Lambda} \otimes \bm{H})(\bm{G} \otimes \bm{I}_{s}) (\bm{G} \otimes \bm{I}_{s})^{-1}\left(\bm{I}_{ms} - \mu (\bm{\Lambda} \otimes \bm{H})\right)^{-1}(\bm{G} \otimes \bm{I}_{s})\\
        ={}& (\bm{G}^T \otimes \bm{I}_{s})(\mu\bm{\Lambda} \otimes \bm{H})\left(\bm{I}_{ms} - \mu (\bm{\Lambda} \otimes \bm{H})\right)^{-1}(\bm{G} \otimes \bm{I}_{s}).
    \end{aligned}
\end{equation}
For matrix $-(\mu\bm{\Lambda} \otimes \bm{H})\left(\bm{I}_{ms} - \mu (\bm{\Lambda} \otimes \bm{H})\right)^{-1}$, we have
\begin{equation}\nonumber
    \begin{aligned}
       & -(\mu\bm{\Lambda} \otimes \bm{H})\left(\bm{I}_{ms} - \mu (\bm{\Lambda} \otimes \bm{H})\right)^{-1} \\
       ={}& \left(\bm{I}_{ms} - \mu\bm{\Lambda} \otimes \bm{H} - \bm{I}_{ms}\right )\left(\bm{I}_{ms} - \mu (\Lambda \otimes \bm{H})\right)^{-1}\\
        ={}& \bm{I}_{ms} - \left(\bm{I}_{ms} - \mu (\bm{\Lambda} \otimes \bm{H})\right)^{-1} \\
        ={}& \bm{I}_{ms} - \bm{D}^{-1} - \sum_{i=1}^{ms-1}(\mu \bm{D}^{-1}(\bm{\Lambda}\otimes \bm{K}))^i\bm{D}^{-1}.
    \end{aligned}
\end{equation}
The previous analysis of $\mu \bm{D}^{-1}(\bm{\Lambda}\otimes \bm{K})$ shows that the last summation is $ O(\frac{\mu}{1+\mu})$. We also know that the element of the diagonal matrix $\bm{I}_{ms} - \bm{D}^{-1}$ has the following form
\begin{equation}\nonumber
    \begin{aligned}
        \brac{\bm{I}_{ms} - \bm{D}^{-1}}_{(p-1)s+i, (q-1)s+j}=\delta^p_q \delta^i_j - \frac{\delta^p_q \delta^i_j}{1  - \mu \Lambda_{p,p} \tilde{d}_{i,i} } = O(\frac{\mu}{1+\mu}).
    \end{aligned}
\end{equation}
Therefore $\mu\left((\bm{A}_0 \bm{Q}) \otimes \bm{H}\right)\bm{\Phi} =O(\frac{\mu}{1+\mu})$.
\end{proof}

Back to the energy estimates for the error, \eqref{equ:vec-u-ke} can be written as the equation of $\bm{U}_{ke\star}$ as follows
\begin{equation}\label{equ:energy-estimate-Vstar}
	\begin{aligned}
		\bm{U}_{ke\star} = U_{ke}^{(1)}\bm{e} - \Delta t\tilde{\bm{H}}\sum_{l=1}^m a_{kl}\partial_x \bm{U}_{le\star} + \mu \bm{H} \sum_{l=1}^m q_{kl}\bm{U}_{le\star} - \Delta t \tilde{\bm{H}}\bm{F}_{k}
	\end{aligned}
\end{equation}
where
\begin{equation}\nonumber
	\bm{F}_{k} = \sum_{l=1}^m a_{kl}\partial_x (\bm{U}_{le}-\bm{U}_{le\star})=\sum_{l=1}^m a_{kl}\partial_x \bm{V}_{l}.
\end{equation}
This implies that $\bm{F}_{k}$  consists of linear combinations of $\partial_x \bm{V}_{l}$. By \eqref{relation-eu} and Lemma \ref{lemma:muH}, each $\bm{V}_{l}$ is a linear combination of $\bm{E}$ with $O(1)$ coefficients. Therefore, $\bm{F}_{k}$ can be expressed as a linear combination of $\partial_x \bm{E}$.

\subsubsection{Proof of the first order uniform accuracy}\label{subsubsec:1st-order}
With the help of the auxiliary error vector $\bm{U}_{ke\star}$ and Lemma \ref{lemma:muH}, we will first prove the first order uniform accuracy of the scheme in this subsection, and then improve it to second-order uniform accuracy in Section \ref{subsubsec:2nd-order}.

Notice that the first component of $\bm{U}_{ke\star}$ is exactly $U_{ke}^{(1)} = U_{ke}^{n}$. Multiplying the $\bm{U}_{ke\star}$ equation in \eqref{equ:energy-estimate-Vstar} by $\bm{U}^T_{je\star} \bm{M}$, we get a scalar equation
\begin{equation}\nonumber
	\begin{aligned}
		\bm{U}^T_{je\star} \bm{M} \bm{U}_{ke\star} ={}& \bm{U}^T_{je\star} \bm{M} U^{n}_{ke\star} e  - \Delta t \bm{U}^T_{je\star} \bm{M} \tilde{\bm{H}}\sum_{l=1}^m a_{kl}\partial_x \bm{U}_{le\star}\\
        {}&+ \mu \bm{U}^T_{je\star} \bm{M} \bm{H} \sum_{l=1}^m q_{kl} \bm{U}_{le\star} - \Delta t \bm{U}^T_{je\star}\bm{M} \tilde{\bm{H}}\bm{F}_{k},
	\end{aligned}
\end{equation}
which can be further simplified as
\begin{equation}\label{equ:rk-key-equ-2}
	\begin{aligned}
		U^{(s)}_{je\star} U^{(s)}_{ke\star} ={}& U^{n}_{je} U^{n}_{ke} - \bm{U}^T_{je\star} \bm{M}_\star \bm{U}_{ke\star} - \Delta t \bm{U}^T_{je\star} \bm{M} \tilde{\bm{H}}\sum_{l=1}^m a_{kl}\partial_x \bm{U}_{le\star}\\
		{}& + \mu \bm{U}^T_{je\star} \bm{M} \bm{H} \sum_{l=1}^m q_{kl} \bm{U}_{le\star} - \Delta t \bm{U}^T_{je\star} \bm{M} \tilde{\bm{H}}\bm{F}_{k}.
	\end{aligned}
\end{equation}
Multiplying the last equation by $\bm{A}_0 = (a_{0jk})$, summing over $j, k$ and integrating in $x$, we obtain
\begin{equation}\label{4equ:err-u-s_jkestar}
	\begin{aligned}
		&\int \sum_{j, k=1}^m U^{(s)}_{je\star} a_{0jk}  U^{(s)}_{ke\star} \diff x \\
		={}& \int \sum_{j, k=1}^m  U^{n}_{je} a_{0jk} U^{n}_{ke} \diff x - \int \sum_{j, k=1}^m  \bm{U}^T_{je\star} a_{0jk} \bm{M}_\star \bm{U}_{ke\star} \diff x \\
		&- \Delta t \int \sum_{j, k=1}^m  \bm{U}^T_{je\star} a_{0jk} \bm{M}\tilde{\bm{H}}\sum_{l=1}^m a_{kl}\partial_x \bm{U}_{le\star} \diff x\\
		&+ \mu \int \sum_{j, k=1}^m  \bm{U}^T_{je\star}a_{0jk} \bm{M} \bm{H} \sum_{l=1}^m q_{kl} \bm{U}_{le\star} \diff x - \Delta t \int \sum_{j, k=1}^m \bm{U}^T_{je\star} a_{0jk}  \bm{M} \tilde{\bm{H}}\bm{F}_{k} \diff x.
	\end{aligned}
\end{equation}
Similar to the proof of Theorem \ref{theorem:imex-rk-us}, we can obtain
\begin{equation}\label{4equ:err-u-s_jestar}
	\begin{aligned}
		&\int \sum_{j, k=1}^m U^{(s)}_{je\star} a_{0jk}  U^{(s)}_{ke\star} \diff x \\
		\leq & \int \sum_{j, k=1}^m  U^{n}_{je}a_{0jk} U^{n}_{ke} \diff x + C\Delta t \norm{\bm{U}^{n}_{e\star}}^2 - \frac{C_\star}{2} \norm{\delta \bm{U}_{e\star}}^2 - \Delta t \int \sum_{j, k=1}^m \bm{U}^T_{je\star} a_{0jk}  \bm{M} \tilde{\bm{H}}\bm{F}_{k} \diff x,
	\end{aligned}
\end{equation}
where we use the fact $\bm{U}^{n}_{e}=\bm{U}^{n}_{e\star}$ and $C_\star$ in \eqref{eq:err-stability-second-term-RHS}. In the above inequality, $\bm{F}_{k}$ consists of linear combinations of $\partial_x \bm{E}$ with $O(1)$ coefficients (due to the Lemma \ref{lemma:muH}).
To treat the terms with $\bm{F}_{k}$, we have the estimate
\begin{equation}\nonumber
	\begin{aligned}
		\Delta t \bigg|\int  U^{(l)}_{je\star} \partial_x E_k^{(i)} \diff x \bigg| 
        \leq {} \Delta t(\norm{U^{(l)}_{je\star}}^2 + \norm{\partial_x E_k^{(i)}}^2 ) 
        \leq {} C \Delta t \norm{U^{n}_{le}}^2 + C\Delta t\norm{U^{(l)}_{je\star}-U^{n}_{le}}^2 + C (\Delta t)^5,
	\end{aligned}
\end{equation}
since $\partial_x E_k^{(i)} = O(\Delta t^2)$ for $1\leq i \leq s$ by Lemma \ref{lem:truncatiom-rk-2nd}. Therefore,  absorbing $C\Delta t\norm{U^{(l)}_{je\star}-U^{n}_{le}}^2$ by the good term $ - \frac{C_\star}{4} \norm{\delta \bm{U}_{e\star}}^2$ since $C \Delta t \leq C_\star/4$ and $\bm{U}^{n}_{e} = \bm{U}^{n}_{e\star}$, we get
\begin{equation}\label{err:ue-star}
	\begin{aligned}
		\int \sum_{j, k=1}^m U^{(s)}_{je\star} a_{0jk}  U^{(s)}_{ke\star} \diff x 
		\leq  \int \sum_{j, k=1}^m a_{0jk} U^{n}_{je} U^{n}_{ke} \diff x + C\Delta t \norm{\bm{U}^{n}_{e\star}}^2 - \frac{C_\star}{4} \norm{\delta \bm{U}_{e\star}}^2 + O(\Delta t^5).
	\end{aligned}
\end{equation}

Using \eqref{scheme:IMEX-RK-2} and \eqref{equ:truncation-equ}, the error $U_{ke}^{n+1}$ satisfies
\begin{equation}\label{equ:vec-u-nplus1e}
	\begin{aligned}
		U^{n+1}_{ke} = U_{ke}^{n} - \Delta t\sum_{l=1}^m\sum_{j=1}^s \tilde{b}_{j} a_{kl}\partial_x U^{{(j)}}_{le} + \mu \sum_{j=1}^s  \sum_{l=1}^m h_{sj} q_{kl}U^{{(j)}}_{le} + E^{n+1}_k.
	\end{aligned}
\end{equation}
Here we use the ISA property which is $b_j = h_{sj}$.
We rewrite \eqref{equ:vec-u-nplus1e} with $\bm{U}_{ke\star}$ as 
\begin{equation}\label{equ:umplus1-star}
\begin{aligned}
    U^{n+1}_{ke} = U_{ke}^{n} - \Delta t \tilde{\bm{b}} \sum_{l=1}^m  a_{kl}\partial_x \bm{U}_{le\star} + \mu  \bm{H}_{s} \sum_{l=1}^m q_{kl}\bm{U}_{le\star} + E^{n+1}_k 
    - \Delta t\tilde{\bm{b}} \bm{F}_{k} + \mu \bm{H}_s \sum_{l=1}^m q_{kl}(\bm{U}_{le}-\bm{U}_{le\star}),
\end{aligned}
\end{equation}
where $\bm{H}_{s}$ denotes the last row of the matrix $\bm{H}$ (and similar notation is used for the last row of other matrices).
Subtracting with the last row of the vector equations \eqref{equ:energy-estimate-Vstar}, we get
\begin{equation}\nonumber
	\begin{aligned}
		U^{n+1}_{ke} = U_{ke\star}^{(s)} - \Delta t (\tilde{\bm{b}}- \tilde{\bm{H}}_s) \sum_{l=1}^m a_{kl}\partial_x \bm{U}_{le\star} + E^{n+1}_k 
		- \Delta t(\tilde{\bm{b}}- \tilde{\bm{H}}_s) \bm{F}_{k} + \mu H_s \sum_{l=1}^mq_{kl}(\bm{U}_{le}-\bm{U}_{le\star}).
	\end{aligned}
\end{equation}
Multiplying above equation by $2U^{n+1}_{ie} a_{0ik} $ respectively, summing over $i, k$ and integrating in $x$ gives the energy estimate
\begin{equation}\label{err:rk-um1}
	\begin{aligned}
		&\int \sum_{i,k=1}^m U^{n+1}_{ie} a_{0ik} U^{n+1}_{ke} \diff x \\
		={}& \int \sum_{i,k=1}^m U^{(s)}_{ie\star} a_{0ik} U_{ke\star}^{(s)} \diff x -  \int \sum_{i,k=1}^m (U^{n+1}_{ie}-U^{(s)}_{ie\star}) a_{0ik} (U^{n+1}_{ke} - U_{ke\star}^{(s)}) \diff x \\
		& - 2\Delta t \int \sum_{i,k=1}^m U^{n+1}_{ie} a_{0ik}  (\tilde{\bm{b}}- \tilde{\bm{H}}_s) \sum_{l=1}^m a_{kl}\partial_x \bm{U}_{le\star} \diff x\\
		&+ 2 \int \sum_{i,k=1}^m U^{n+1}_{ie} a_{0ik} E^{n+1}_k \diff x - 2\Delta t \int \sum_{i,k=1}^m U^{n+1}_{ie} a_{0ik} (\tilde{\bm{b}}- \tilde{\bm{H}}_s) \bm{F}_{k} \diff x\\
		&+ 2 \mu \int \sum_{i,k=1}^m U^{n+1}_{ie} a_{0ik} \sum_{l=1}^m \bm{H}_s q_{kl}(\bm{U}_{le}-\bm{U}_{le\star}) \diff x.
	\end{aligned}
\end{equation}
Adding with \eqref{err:ue-star} gives 
\begin{equation}
	\begin{aligned}\label{4equ:err-umplus1}
		&\int \sum_{i,k=1}^m U^{n+1}_{ie} a_{0ik} U^{n+1}_{ke} \diff x \\
		\leq {}&\int \sum_{j, k=1}^m a_{0jk} U^{n}_{je} U^{n}_{ke} \diff x + C\Delta t \norm{\bm{U}^{n}_{e\star}}^2 -  \frac{C_\star}{4} \norm{\delta \bm{U}_{e\star}}^2 - C \norm{\bm{U}^{n+1}_{e}-\bm{U}^{(s)}_{e\star}}^2 \\
        & - 2\Delta t \int \sum_{i,k=1}^m U^{n+1}_{ie} a_{0ik}  (\tilde{\bm{b}}- \tilde{\bm{H}}_s) \sum_{l=1}^m a_{kl}\partial_x \bm{U}_{le\star} \diff x\\
		&+ 2 \int \sum_{i,k=1}^m U^{n+1}_{ie} a_{0ik} E^{n+1}_k \diff x - 2\Delta t \int \sum_{i,k=1}^m U^{n+1}_{ie} a_{0ik} (\tilde{\bm{b}}- \tilde{\bm{H}}_s) \bm{F}_{k} \diff x\\
		&+ 2 \mu \int \sum_{i,k=1}^m U^{n+1}_{ie} a_{0ik} \sum_{l=1}^m \bm{H}_s q_{kl}(\bm{U}_{le}-\bm{U}_{le\star}) \diff x + O(\Delta t^5).
	\end{aligned}
\end{equation}

The second line on the right side can be estimated by 
\begin{equation}\nonumber
    \begin{aligned}
        & - 2\Delta t \int \sum_{i,k=1}^m U^{n+1}_{ie} a_{0ik}  (\tilde{\bm{b}}- \tilde{\bm{H}}_s) \sum_{l=1}^m a_{kl}\partial_x \bm{U}_{le\star} \diff x \\
        \leq {}& C \sum_{i, l=1}^m \sum_{j=1}^s \Delta t \abs{\int  \brac{U_{ie}^{n+1} \partial_x U_{le\star}^{(j)} + U_{le}^{n+1} \partial_x U_{ie\star}^{(j)}} \diff x}\\
        \leq {}& C \sum_{i, l=1}^m \sum_{j=1}^s \Delta t \abs{\int \brac{ (U_{ie}^{n+1} - U_{ie\star}^{(j)}) \partial_x U_{le\star}^{(j)} + (U_{le}^{n+1} - U_{le\star}^{(j)}) \partial_x U_{ie\star}^{(j)}} \diff x}. 
    \end{aligned}
\end{equation}
Each term can be controlled by 
\begin{equation}\nonumber
    \begin{aligned}
        &\Delta t \abs{\int (U_{ie}^{n+1} - U_{ie\star}^{(j)}) \partial_x U_{le\star}^{(j)} \diff x} \leq C \norm{U_{ie}^{n+1} - U_{ie\star}^{(j)}}^2 + C \Delta t^2 \norm{\partial_x U_{le\star}^{(j)}}^2 \\[2mm]
        \leq {}& C \norm{U_{ie}^{n+1} - U_{ie\star}^{(s)}}^2 + C \norm{U_{ie\star}^{(s)} - U_{ie\star}^{(j)}}^2 + C \Delta t \norm{U_{le\star}^{(j)}}^2 \\[2mm]
        \leq {}& C \norm{U_{ie}^{n+1} - U_{ie\star}^{(s)}}^2 + C \norm{U_{ie\star}^{(s)} - U_{ie\star}^{(j)}}^2 + C \Delta t \norm{U_{le\star}^{n}}^2 + C \Delta t \norm{U_{le\star}^{(j)} - U_{le\star}^{n}}^2.
    \end{aligned}
\end{equation}
The first term can be absorbed by $- C\norm{\bm{U}^{n+1}_{e}-\bm{U}^{(s)}_{e\star}}^2$ and the second and fourth term can be absorbed by $- \frac{C_\star}{4} \norm{\delta \bm{U}_{e\star}}^2$. The only remaining items is $C \Delta t \norm{U_{le\star}^{n}}^2$ which can be estimated by the first term of RHS of in \eqref{4equ:err-umplus1} since $U_{le\star}^{n} = U_{le}^{n}$.

The each term in $\int \sum_{i,k=1}^m U^{n+1}_{ie} a_{0ik} E^{n+1}_k \diff x$ can be estimated by
\begin{equation}\nonumber
	\begin{aligned}
		\int U^{n+1}_{ie} a_{0ik} E^{n+1}_k \diff x \leq C (\Delta t\norm{U^{n+1}_{ie\star}}^2 + \frac{1}{\Delta t}\norm{E^{n+1}_k}^2 ) \leq C \Delta t \norm{\bm{U}^{n+1}_{e}}^2 + O(\Delta t^5),
	\end{aligned}
\end{equation}
using $E^{n+1}_k = O(\Delta t^3)$. 
Notice that if we do not introduce the new variable $\bm{U}_{\star}$, we just rewrite the equation of $\bm{U}^{n+1}_{e}$ by subtracting with the last row of the vector equation \eqref{equ:vec-u-ke} as we did before in proof of Theorem \ref{theorem:imex-rk-us} and estimate it. There will be an additional item for $\int \sum_{i,k=1}^m U^{n+1}_{ie} a_{0ik} E^{(s)}_k \diff x$ in the RHS of new equation  \eqref{4equ:err-umplus1} of $\bm{U}^{n+1}_{e}$. This term can only bounded by $\Delta t \norm{\bm{U}^{n+1}_{e}}^2$ and $O(\Delta t^3)$ since $E^{(s)}_k = O(\Delta t^2)$. This can not derive the second-order accuracy.
Similarly, for the terms involving $\bm{F}_{k}$, one can estimate as 
\begin{equation}\nonumber
	\begin{aligned}
		\Delta t \abs{\int U^{n+1}_{je\star} \partial_x E_k^{(i)} \diff x } \leq \Delta t(\norm{U^{n+1}_{je\star}}^2 + \norm{\partial_x E_k^{(i)}}^2 ) \leq C \Delta t \norm{\bm{U}^{n+1}_{e}}^2 + O(\Delta t^5),
	\end{aligned}
\end{equation}
and the term $C \Delta t \norm{\bm{U}^{n+1}_{e}}^2$ can be absorbed by LHS. Therefore all these non-stiff terms give a contribution of at most $O(\Delta t^5)$.

Note that 
\begin{equation}\nonumber
    \begin{aligned}
	   vec(\bm{U}_{1e} - \bm{U}_{1e\star}, \cdots, \bm{U}_{me} - \bm{U}_{me\star}) = \bm{\Phi} vec(\bm{E}_1,  ~\cdots, ~\bm{E}_m).
  \end{aligned}
\end{equation}
For the matrix $\mu ((\bm{A}_0 \bm{Q}) \otimes \bm{H}_s)\bm{\Phi}$, the best thing one can say is that it has elements $\leq C \frac{\mu}{1+\mu}$ from Lemma \ref{lemma:muH}. Therefore, the worst term is the stiff term, which can be estimated by
\begin{equation}\nonumber
	\begin{aligned}
		&2 \mu \int  U^{n+1}_{ie} a_{0ik} \bm{H}_s q_{kl}(\bm{U}_{le}-\bm{U}_{le\star}) \diff x \\[2mm]
		\leq {}& C\Delta t \norm{U^{n+1}_{ie}}^2 + \frac{C}{\Delta t}\max_{1\leq l\leq n} \norm{\mu ((\bm{A}_0 \bm{Q}) \otimes \bm{H}_s)\bm{\Phi} \bm{E}_{l}}^2\\[2mm]
		\leq {}& C\Delta t \norm{U^{n+1}_{ie}}^2 + \frac{C}{\Delta t}\frac{\mu^2}{(1+\mu)^2}\max_{1\leq l\leq n} \norm{\bm{E}_l}^2\\[2mm]
		\leq {}& C\Delta t \norm{U^{n+1}_{ie}}^2 + C(\Delta t)^3\frac{\mu^2}{(1+\mu)^2},
	\end{aligned}
\end{equation}
using $\norm{\bm{E}_l} = O(\Delta t^2)$. The term $C \Delta t \norm{U^{n+1}_{ie}}^2$ can be absorbed by LHS. Therefore we finally get 
\begin{equation}\label{err-inequ-UAU-1}
	\begin{aligned}
		&\int \sum_{i,k=1}^m U^{n+1}_{ie} a_{0ik} U^{n+1}_{ke} \diff x \\
		\leq {}& (1+ C\Delta t)\int \sum_{j, k=1}^m U^{n}_{je} a_{0jk} U^{n}_{ke} \diff x + O( \Delta t^5 + \Delta t^3\frac{\mu^2}{(1+\mu)^2}).
	\end{aligned}
\end{equation}
Using Gronwall's inequality, we get
\begin{equation}\label{err:uim1-1rd}
	\begin{aligned}		\norm{\bm{U}^{n}_{e}}^2_{\bm{A}_0} \leq{}& C(T)(\Delta t^4 + \Delta t^2\frac{\mu^2}{(1+\mu)^2} + \norm{\bm{U}^{0}_{e}}^2)\\
        \leq{}& C(T)(\Delta t^4 + \Delta t^2\frac{\mu^2}{(1+\mu)^2} + \frac{1}{N^8}).
	\end{aligned}
\end{equation}
Notice that the same error estimate works for any intermediate stages $U_{ke\star}^{(j)}$. Also, the above estimate only utilizes the consistency of initial data up to order 4. Since we assumed the consistency of initial data up to order 6, the same error estimate works for the $x$-derivatives of these quantities up to order 2. 

The inequality \eqref{err:uim1-1rd} implies the first-order uniform accuracy. To be more precise, if $\varepsilon = O(1)$, $\ie$  $\mu=O(\Delta t)$, then it gives the second-order accuracy, but it degenerates to the first-order accuracy if $\varepsilon\ll \Delta t$, $\ie$, $\mu\gg 1$. This motivates us to utilize the terms in the equation \eqref{4equ:err-u-s_jestar}, a coercive term proportional to $\mu$, to obtain the second-order uniform accuracy.

\subsubsection{Proof of the second order uniform accuracy}\label{subsubsec:2nd-order}
In this section, we improve the error estimate \eqref{err:uim1-1rd} in the previous section to the second-order uniform accuracy.

In the rest of proof, we assume $\varepsilon \leq 1$. Thus, $\Delta t^2\frac{\mu^2}{(1+\mu)^2}$ is always the worst term in \eqref{err:uim1-1rd}, since $\Delta t\leq C$ is assumed.

Multiplying \eqref{equ:rk-key-equ-2} by $a_{0jk}$, summing over $j, k=m-r+1, \ldots, m$ and integrating in $x$, we obtain
\begin{equation}\nonumber
	\begin{aligned}
		&\int \sum_{j, k=m-r+1}^m U^{(s)}_{je\star} a_{0jk}  U^{(s)}_{ke\star} \diff x \\
		={}& \int \sum_{j, k=m-r+1}^m  U^{n}_{je} a_{0jk} U^{n}_{ke} \diff x - \int\sum_{j, k=m-r+1}^m  \bm{U}^T_{je\star} a_{0jk} \bm{M}_\star \bm{U}_{ke\star} \diff x \\
		&- \Delta t \int \sum_{j, k=m-r+1}^m  \bm{U}^T_{je\star} a_{0jk} \bm{M}\tilde{\bm{H}}\sum_{l=1}^m a_{kl}\partial_x \bm{U}_{le\star} \diff x\\
		&+ \mu \int \sum_{j, k=m-r+1}^m  \bm{U}^T_{je\star}a_{0jk} \bm{M} \bm{H} \sum_{l=1}^m q_{kl} \bm{U}_{le\star} \diff x - \Delta t \int \sum_{j, k=m-r+1}^m \bm{U}^T_{je\star} a_{0jk}  \bm{M} \tilde{\bm{H}}\bm{F}_{k} \diff x.
	\end{aligned}
\end{equation}
Similar to the estimate in \eqref{eq:err-stability-second-term-RHS}, we have
\begin{equation}\nonumber
	\begin{aligned}
		&\int \sum_{j, k=m-r+1}^m U^{(s)}_{je\star} a_{0jk}  U^{(s)}_{ke\star} \diff x \\
		\leq {}& \int \sum_{j, k=m-r+1}^m a_{0jk} U^{n}_{je} U^{n}_{ke} \diff x - C_\star \sum_{i=m-r+1}^m \norm{\delta U_{ie\star}}^2 \\
        &- \Delta t \int \sum_{j, k=m-r+1}^m a_{0jk} \bm{U}^T_{je\star} \bm{M} \tilde{\bm{H}}\sum_{i=1}^m a_{ki}\partial_x \bm{U}_{ie\star} \diff x\\
		&+ \mu \int \sum_{j, k=m-r+1}^m a_{0jk} \bm{U}^T_{je\star} \bm{M} \bm{H} \sum_{i=1}^m q_{ki} \bm{U}_{ie\star} \diff x - \Delta t \int \sum_{j, k=m-r+1}^m a_{0jk} \bm{U}^T_{je\star} \bm{M} \tilde{\bm{H}}\bm{F}_{k} \diff x
	\end{aligned}
\end{equation}
with $C_\star$ in \eqref{eq:err-stability-second-term-RHS}.
For the fourth term of RHS in the above equation, according to the structural stability condition, we have $\bm{A}_0 \bm{Q}= \diag(0, -\bar{\bm{S}})$ with $\bar{\bm{S}} = -\bm{A}_{02}\hat{\bm{S}}$ being an SPD matrix. Thus there exists an SPD matrix $\bar{\bm{K}}=(\bar{k}_{ij})$ such that $\bar{\bm{S}} = \bar{\bm{K}}^T \bar{\bm{K}}$. Set $\bar{\bm{U}}_i = \sum_{j=m-r+1}^m \bar{k}_{ji} \bm{U}_j, i=m-r+1, \ldots, m$. Since $\bm{(M1)}$ and $\bm{(H)}$ are assumed, Lemma \ref{lemma:M1A-prop} gives the coercive estimate
\begin{equation}\nonumber
	\begin{aligned}
		&\mu \int \sum_{j, k=m-r+1}^m \bm{U}^T_{je\star} \bm{M} \bm{H} \sum_{i=1}^m a_{0jk} q_{ki} \bm{U}_{ie\star} \diff x \\
        ={}& - \mu \int \sum_{i=m-r+1}^m \bar{\bm{U}}^T_{ie\star} \bm{M} \bm{H}  \bar{\bm{U}}_{ie\star} \diff x \\
        \leq{}& - \mu C_{\bm{M}\bm{H}}  \sum_{i=m-r+1}^m \norm{U_{ie\star}^{(s)}}^2.
	\end{aligned}
\end{equation}
Here we use the equivalence of the norms $\norm{\cdot}$ and $\norm{\cdot}_{\bar{\bm{S}}}$.

For the third term, we have
\begin{equation}\label{equ:r-2or-2-1}
	\begin{aligned}
		{}&\bigg|- \Delta t \int \sum_{j, k=m-r+1}^m a_{0jk} \bm{U}^T_{je\star} \bm{M}\tilde{\bm{H}}\sum_{i=1}^m a_{ki}\partial_x \bm{U}_{ie\star} \diff x\bigg| \\
		\leq {}& C \Delta t \sum_{j=m-r+1}^m  \sum_{i=1}^m\sum_{p,l=1}^s \bigg\vert \int U^{(p)}_{je\star} \partial_x U^{(l)}_{ie\star} \diff x \bigg\vert \\
		\leq {}& \sum_{j=m-r+1}^m  \sum_{i=1}^m\sum_{p,l=1}^s\bigg(C_0 \frac{\mu}{1+\mu} \norm{U^{(p)}_{je\star}}^2 + C\Delta t^2\frac{1+\mu}{\mu} \norm{\partial_x U^{(l)}_{ie\star}}^2 \bigg) \\
		\leq {}& \sum_{j=m-r+1}^m \sum_{p=1}^s \frac{C_0\mu}{1+\mu} \norm{U^{(p)}_{je\star}}^2 + C\sum_{i=1}^m\sum_{l=1}^s\Delta t^2\frac{1+\mu}{\mu} \norm{\partial_x U^{(l)}_{ie\star}}^2  \\
		\leq {}& \sum_{j=m-r+1}^m \sum_{p=1}^s \frac{C_0 \mu}{1+\mu}\brac{\norm{U_{je\star}^{(p)} - U_{je\star}^{(s)}}^2 +  \norm{U_{j\star}^{(s)}}^2} + C\Delta t^2\frac{1+\mu}{\mu} \brac{\Delta t^2\frac{\mu^2}{(1+\mu)^2} + \frac{1}{N^8}}\\
		\leq {}& \sum_{j=m-r+1}^m \brac{\frac{C_\star}{4} \norm{\delta U_{je\star}}^2 + \mu \frac{C_{\bm{M}\bm{H}} }{2} \norm{U_{j\star}^{(s)}}^2} + C (\Delta t^4 + \frac{1}{N^8}) \frac{\mu}{1+\mu},
	\end{aligned}
\end{equation}
with $2C_0= \min \{C_\star/4, C_{\bm{M}\bm{H}}/2 \}$. Here we use the estimate \eqref{err:uim1-1rd} for $\norm{\partial_x U^{(l)}_{ie\star}}$ and the fact that $\Delta t^2\frac{1+\mu}{\mu} \leq C \frac{\mu}{1+\mu}$.

The remaining term $- \Delta t \int \sum_{j, k=m-r+1}^m a_{0jk} \bm{U}^T_{je\star} \bm{M} \tilde{\bm{H}}\bm{F}_{k} \diff x$ can be controlled by $ O(\Delta t^5) + C\Delta t \sum_{j=m-r+1}^m\norm{U_{je}^{n}}^2 - \frac{C_\star}{4}\sum_{j=m-r+1}^m \norm{\delta U_{je\star}}^2$ as before.

Therefore, we obtain the following estimate
\begin{equation}\nonumber
	\begin{aligned}
		&\int \sum_{j, k=m-r+1}^m U^{(s)}_{je\star} a_{0jk}  U^{(s)}_{ke\star} \diff x \\
        \leq {}& (1+ C\Delta t)\int \sum_{j, k=m-r+1}^m a_{0jk} U^{n}_{je} U^{n}_{ke} \diff x - \frac{C_\star}{2} \sum_{i=m-r+1}^m \norm{\delta U_{ie\star}}^2 \\
        &- \mu \frac{C_{\bm{M}\bm{H}} }{2}   \sum_{i=m-r+1}^m \norm{U_{ie\star}^{(s)}}^2 + C(\Delta t^4 + \frac{1}{N^8})\frac{\mu}{1+\mu}.
	\end{aligned}
\end{equation}
Since we assume $\Delta t\leq C$, $\varepsilon \leq 1$, it is easy to show that $O(\Delta t^5) \leq O(\Delta t^4 \frac{\mu}{1+\mu}) $ by noticing that $\Delta t^4 \frac{\mu}{1+\mu} = \Delta t^4 \frac{\frac{\Delta t}{\varepsilon}}{1+\frac{\Delta t}{\varepsilon}} =  \frac{\Delta t^5}{\varepsilon + \Delta t}$.

Combining the last estimate, we can obtain the following estimate similar to \eqref{err:rk-um1} 
\begin{equation}\nonumber
	\begin{aligned}
		&\int \sum_{j, k=m-r+1}^m U^{n+1}_{ie} a_{0ik} U^{n+1}_{ke} \diff x \\
        \leq {}& (1+ C\Delta t)\int \sum_{j, k=m-r+1}^m a_{0jk} U^{n}_{je} U^{n}_{ke} \diff x -  \frac{C_\star}{2} \sum_{i=m-r+1}^m \norm{\delta U_{ie\star}}^2 - \mu\frac{C_{\bm{M}\bm{H}}}{2} \sum_{i=m-r+1}^m \norm{U_{ie\star}^{(s)}}^2 \\
        & - C \sum_{i=m-r+1}^m \norm{U^{n+1}_{ie}-U^{(s)}_{ie\star}}^2 - 2\Delta t \int \sum_{i,k=m-r+1}^m U^{n+1}_{ie} a_{0ik}  (\tilde{\bm{b}}- \tilde{\bm{H}}_s) \sum_{l=1}^m a_{kl}\partial_x \bm{U}_{le\star} \diff x\\
		&+ 2 \int \sum_{i,k=m-r+1}^m U^{n+1}_{ie} a_{0ik} E^{n+1}_k \diff x - 2\Delta t \int \sum_{i,k=m-r+1}^m U^{n+1}_{ie} a_{0ik} (\tilde{\bm{b}}- \tilde{\bm{H}}_s) \bm{F}_{k} \diff x\\
		&+ 2 \mu \int \sum_{i,k=m-r+1}^m U^{n+1}_{ie} a_{0ik} \sum_{l=1}^m H_s q_{kl}(\bm{U}_{le}-\bm{U}_{le\star}) \diff x  + C(\Delta t^4 + \frac{1}{N^8})\frac{\mu}{1+\mu}.
	\end{aligned}
\end{equation}
Since $\bm{A}_0 \bm{Q} = \diag(0, -\bar{\bm{S}})$, the last integral can be rewritten as 
\begin{equation}\nonumber
    \begin{aligned}
         &2 \mu \int \sum_{i,k=m-r+1}^m U^{n+1}_{ie} a_{0ik} \sum_{l=1}^m \bm{H}_s q_{kl}(\bm{U}_{le}-\bm{U}_{le\star}) \diff x\\
         ={}& -2 \mu \int \sum_{i, l=m-r+1}^m U^{n+1}_{ie} \bm{H}_s \bar{S}_{il}(\bm{U}_{le}-\bm{U}_{le\star}) \diff x.
    \end{aligned}
\end{equation}
Note that 
\begin{equation}\nonumber
    \begin{aligned}
	   &vec(\bm{U}_{1e} - \bm{U}_{1e\star}, \cdots, \bm{U}_{me} - \bm{U}_{me\star}) 
       = \bm{\Phi} vec(\bm{E}_1,  ~\cdots, ~\bm{E}_m).
  \end{aligned}
\end{equation}
The elements in the matrix $\mu ((\bm{A}_0 \bm{Q}) \otimes \bm{H}_s)\bm{\Phi}$ are bounded by  $O(\frac{\mu}{1+\mu})$ from Lemma \ref{lemma:muH}, which means that the elements in the matrix $-\mu (\bar{\bm{S}} \otimes \bm{H}_s)\bm{\Phi}$ also are bounded by  $O(\frac{\mu}{1+\mu})$. 
This fact combined with $\norm{\bm{E}} = O(\Delta t^2)$ helps us improve the worst stiff term estimate as 
\begin{equation}\nonumber
    \begin{aligned}
       & \bigg| 2\mu \int \sum_{i,k=m-r+1}^m U^{n+1}_{ie} a_{0ik} \sum_{l=1}^m \bm{H}_s q_{kl}(\bm{U}_{le}-\bm{U}_{le\star}) \diff x \bigg| \\
       \leq {}& \frac{C_1\mu}{1+\mu}\sum_{i=m-r+1}^m\norm{U^{n+1}_{ie}}^2 + C \frac{1+\mu}{\mu} \max_{1\leq l\leq m} \norm{\mu (\bar{\bm{S}} \otimes \bm{H}_s)\bm{\Phi} \bm{E}_{l}}^2\\
        \leq{}& \frac{C_1\mu}{1+\mu}\sum_{i=m-r+1}^m\norm{U^{n+1}_{ie}}^2 + C \frac{1+\mu}{\mu}\Delta t^4 \frac{\mu^2}{(1+\mu)^2}\\
        = {}& \frac{C_1\mu}{1+\mu}\sum_{i=m-r+1}^m\norm{U^{n+1}_{ie}}^2 + C \Delta t^4 \frac{\mu}{1+\mu}
    \end{aligned}
\end{equation}
with $4C_1=\min\{C, C_{\bm{M}\bm{H}}/2\}$. The previous inequality combined with the following equation 
\begin{equation}\nonumber
    \begin{aligned}
        -C\norm{U^{n+1}_{ie}-U^{(s)}_{ie\star}}^2 - C_{\bm{M}\bm{H}} \mu  \norm{U_{ie\star}^{(s)}}^2 \leq \frac{4C_1\mu}{1+\mu}(-\norm{U^{n+1}_{ie}-U^{(s)}_{ie\star}}^2 -  \norm{U_{ie\star}^{(s)}}^2 )\leq -\frac{2C_1 \mu}{1+\mu} \norm{U^{n+1}_{ie}}^2 
    \end{aligned}
\end{equation}
contributes a total of $-\frac{C_1\mu}{1+\mu}\sum_{i=m-r+1}^m\norm{U^{n+1}_{ie}}^2 + C \Delta t^4 \frac{\mu}{1+\mu}$.
By estimating other terms as we did before which gives $O(\Delta t^5) $ in total, we obtain
\begin{equation}\label{r-err-V-2rd-1}
    \begin{aligned}
       &\int \sum_{j,k=m-r+1}^m U^{n+1}_{ie} a_{0jk} U^{n+1}_{ke} \diff x\\ 
       \leq {}& (1+ C\Delta t)\int \sum_{j, k=m-r+1}^m a_{0jk} U^{n}_{je} U^{n}_{ke} \diff x + \frac{\mu}{1+\mu}(C(\Delta t^4 + \frac{1}{N^8}) - C_1 \sum_{i=m-r+1}^m \norm{U^{n+1}_{ie}}^2) \\
       \leq {}& (1+ C\Delta t)\int \sum_{j, k=m-r+1}^m a_{0jk} U^{n}_{je} U^{n}_{ke} \diff x + \frac{\mu}{1+\mu}(C(\Delta t^4 + \frac{1}{N^8}) \\
       &- C_2 \int \sum_{j,k=m-r+1}^m U^{n+1}_{ie} a_{0jk} U^{n+1}_{ke} \diff x).
    \end{aligned}
\end{equation} 

Next, from \eqref{r-err-V-2rd-1}, we prove the following estimate by mathematical induction:
\begin{equation}\label{r-err-V-2rd}
    \begin{aligned}
        \int \sum_{j,k=m-r+1}^m U^{n}_{ie} a_{0jk} U^{n}_{ke} \diff x \leq K (\Delta t^4 + \frac{1}{N^8}), \quad \forall n\ge 0.
    \end{aligned}
\end{equation}
Here, $K>0$ is a constant to be determined later, which is independent of $\varepsilon$ and $\Delta t$.
For $n=0$, the above formula is obviously true due to the initial projection. Assuming that \eqref{r-err-V-2rd} holds for any integer no larger than $n$, we will prove it also holds for $n+1$.
We rewrite \eqref{r-err-V-2rd-1} as follows:
\begin{equation}\nonumber
\begin{aligned}
    \int \sum_{j,k=m-r+1}^m U^{n+1}_{ie} a_{0jk} U^{n+1}_{ke} \diff x \leq \frac{1+C \Delta t}{\alpha} \int \sum_{j,k=m-r+1}^m U^{n}_{ie} a_{0jk} U^{n}_{ke} \diff x + \frac{\beta}{\alpha} ( \Delta t^4 + \frac{1}{N^8})
\end{aligned}    
\end{equation}
with 
\begin{equation}\nonumber
    \alpha := 1 + \frac{\mu}{1+\mu} C_2, \quad \beta := \frac{\mu}{1+\mu} C,
\end{equation}
From the assumption in the induction, we have
\begin{equation}\nonumber
\begin{aligned}
    \int \sum_{j,k=m-r+1}^m U^{n+1}_{ie} a_{0jk} U^{n+1}_{ke} \diff x \leq \frac{1+C \Delta t}{\alpha} K (\Delta t^4 + \frac{1}{N^8}) + \frac{\beta}{\alpha} ( \Delta t^4 + \frac{1}{N^8}).
\end{aligned}    
\end{equation}
Thus \eqref{r-err-V-2rd} holds for $n+1$ if 
\begin{equation*}
    \frac{1+C \Delta t}{\alpha} K (\Delta t^4 + \frac{1}{N^8}) + \frac{\beta}{\alpha} ( \Delta t^4 + \frac{1}{N^8}) \le K (\Delta t^4 + \frac{1}{N^8}),
\end{equation*}
which is equivalent to
\begin{equation*}
\begin{aligned}
    (\alpha - 1 - C\Delta t) K &\ge \beta, \\
    \Leftrightarrow (\frac{\mu}{1+\mu}C_2 - C\Delta t) K &\ge \frac{\mu}{1+\mu}C, \\
    \Leftrightarrow (C_2 - (\varepsilon + \Delta t)C) K &\ge C.
\end{aligned}
\end{equation*}
The above formula holds by taking $\varepsilon + \Delta t \le \frac{C_2}{2C}$ and $K= \frac{2C}{C_2}$. Hence we prove the estimate \eqref{r-err-V-2rd}.

Next, we improve the error estimate \eqref{err:uim1-1rd} to second-order uniform accuracy for $U_{i}$ for $i = 1, \cdots, m-r$.
Multiplying \eqref{equ:rk-key-equ-2} by $a_{0jk}$, summing over $j, k=1, \ldots, m-r$ and integrating in $x$, we have
\begin{equation}\nonumber
	\begin{aligned}
		&\int \sum_{j, k=1}^{m-r} U^{(s)}_{je\star} a_{0jk}  U^{(s)}_{ke\star} \diff x \\
		\leq {}& \int \sum_{j, k=1}^{m-r} a_{0jk} U^{n}_{je} U^{n}_{ke} \diff x - C_\star \sum_{i=1}^{m-r} \norm{\delta U_{ie\star}}^2 - \Delta t \int \sum_{j, k=1}^{m-r} a_{0jk} \bm{U}^T_{je\star} \bm{M} \tilde{\bm{H}}\sum_{i=1}^m a_{ki}\partial_x \bm{U}_{ie\star} \diff x\\
		& - \Delta t \int \sum_{j, k=1}^{m-r} a_{0jk} \bm{U}^T_{je\star} \bm{M} \tilde{\bm{H}}\bm{F}_{k} \diff x.
	\end{aligned}
\end{equation}
We divide the third term into two parts
\begin{equation}\nonumber
	\begin{aligned}
		&- \Delta t \int \sum_{j, k=1}^{m-r} a_{0jk} \bm{U}^T_{je\star} \bm{M} \tilde{\bm{H}}\sum_{i=1}^m a_{ki}\partial_x \bm{U}_{ie\star} \diff x\\
		={}& - \Delta t \int \sum_{j, k=1}^{m-r} a_{0jk} \bm{U}^T_{je\star} \bm{M} \tilde{\bm{H}}\sum_{i=1}^{m-r} a_{ki}\partial_x \bm{U}_{ie\star} \diff x - \Delta t \int \sum_{j, k=1}^{m-r} a_{0jk} \bm{U}^T_{je\star} \bm{M} \tilde{\bm{H}}\sum_{i=m-r+1}^m a_{ki}\partial_x \bm{U}_{ie\star} \diff x
	\end{aligned}
\end{equation}
Using the estimate in \eqref{r-err-V-2rd}, the second part in the above equation can be estimated by
\begin{equation}\nonumber
	\begin{aligned}
		&- \Delta t \int \sum_{j, k=1}^{m-r} a_{0jk} \bm{U}^T_{je\star} \bm{M} \tilde{\bm{H}}\sum_{i=m-r+1}^m a_{ki}\partial_x \bm{U}_{ie\star} \diff x\\
		\leq {}& C \Delta t \sum_{j=1}^{m-r} \norm{\bm{U}_{je\star}}^2 + C \Delta t \sum_{i=m-r+1}^m \norm{\partial_x \bm{U}_{ie\star}}^2\\
        \leq {}& C \Delta t \sum_{j=1}^{m-r}\norm{U^n_{je\star}}^2 + C \Delta t \sum_{j=1}^{m-r}\sum_{l=1}^s \norm{U^{(l)}_{je\star}-U^n_{je\star}}^2 + C \Delta t^5.
	\end{aligned}
\end{equation}
By estimating other terms in the same way as before, we obtain
\begin{equation}\nonumber
	\begin{aligned}
		\int \sum_{j, k=1}^{m-r} U^{(s)}_{je\star} a_{0jk}  U^{(s)}_{ke\star} \diff x 
        \leq {} (1+ C\Delta t)\int \sum_{j, k=1}^{m-r} a_{0jk} U^{n}_{je} U^{n}_{ke} \diff x - \frac{C_\star}{2} \sum_{i=1}^{m-r} \norm{\delta U_{ie\star}}^2 + C\Delta t^5.
	\end{aligned}
\end{equation}
Similar to \eqref{err:rk-um1}, we can obtain the following estimate
\begin{equation}\nonumber
	\begin{aligned}
		&\int \sum_{j, k=1}^{m-r} U^{n+1}_{ie} a_{0ik} U^{n+1}_{ke} \diff x \\
        \leq {}& (1+ C\Delta t)\int \sum_{j, k=1}^{m-r} a_{0jk} U^{n}_{je} U^{n}_{ke} \diff x -  \frac{C_\star}{2} \sum_{i=1}^{m-r} \norm{\delta U_{ie\star}}^2 - C \sum_{i=1}^{m-r} \norm{U^{n+1}_{ie}-U^{(s)}_{ie\star}}^2 \\&
        - 2\Delta t \int \sum_{j,k=1}^{m-r} U^{n+1}_{je} a_{0jk}  (\tilde{\bm{b}}- \tilde{\bm{H}}_s) \sum_{l=1}^m a_{kl}\partial_x \bm{U}_{le\star} \diff x+ 2 \int \sum_{j,k=1}^{m-r} U^{n+1}_{je} a_{0jk} E^{n+1}_k \diff x \\
        &- 2\Delta t \int \sum_{j,k=1}^{m-r} U^{n+1}_{je} a_{0jk} (\tilde{\bm{b}}- \tilde{\bm{H}}_s) \bm{F}_{k} \diff x  + C\Delta t^5.
	\end{aligned}
\end{equation}
By estimating other terms in the same way as before, adding them to the previous estimate, we obtain
\begin{equation}\nonumber
	\begin{aligned}
		\int \sum_{j, k=1}^{m-r} U^{n+1}_{je\star} a_{0jk}  U^{(s)}_{ke\star} \diff x 
        \leq  (1+ C\Delta t)\int \sum_{j, k=1}^{m-r} a_{0jk} U^{n}_{je} U^{n}_{ke} \diff x + C\Delta t^5.
	\end{aligned}
\end{equation}
Applying Gronwall's inequality, we can get for $i = 1, \cdots, m-r$
\begin{equation}\label{r-err-W-2rd}
\begin{aligned}
        \norm{U_{ie}^{n+1}}^2_{\bm{A}_0} \leq C(T)\brac{\Delta t^4 + \norm{U_{ie}^0}^2}\leq C(T)\brac{\Delta t^4 + \frac{1}{N^8}}, \qquad \forall n\Delta t \leq T  .
\end{aligned}
\end{equation}
Therefore, according to \eqref{r-err-V-2rd} and \eqref{r-err-W-2rd}, we prove the uniform second-order accuracy.

\section{Third order uniform accuracy}\label{sec:3rd-uniform-accuracy}
Our main result in this section is stated as follows.
\begin{thm}[Third order uniform accuracy of IMEX-RK schemes]\label{thm:5-1-3rd}
    Under the same assumption as in Theorem \ref{theorem:imex-rk-2order-u-a}, further assume
    \begin{itemize}
        \item The IMEX-RK scheme satisfies the standard third-order condition \eqref{equ:cond-3rd-rk}.
        \item The stage order conditions
        \begin{equation}\label{cond:stage-3rd-order}
            \frac{1}{2}c_i^2 = \sum_{j=1}^{i-1} \tilde{h}_{ij}c_j = \sum_{j=1}^{i}h_{ij}c_j, \quad i = 3, \dots, s 
        \end{equation}
        \item The $\textit{vanishing coefficient condition}$
        \begin{equation}\label{condition-5.3-vanishing}
            \tilde{b}_2=0 \quad \textit{and}\quad h_{i, 2} =0, \quad i = 3, \dots, s.
        \end{equation}
        \item The initial data is consistent up to order 8.
    \end{itemize}
    
        Then for any $T > 0$ and integer $n$ with $n \Delta t\leq T$, we have
        \begin{equation}
            \norm{\bm{U}^{n}_e} \leq C\left( \Delta t^6 + \frac{1}{N^{12}} \right),
        \end{equation}
        with $C$ independent of $\varepsilon, N$ and $\Delta t$.
\end{thm}

Among the IMEX-RK schemes considered in this paper, only the IMEX-RK scheme BHR(5,5,3)* given in the Appendix satisfies the assumptions in Theorem \ref{thm:5-1-3rd}, hence it will exhibit at least third-order uniform accuracy in time. 

We will prove this theorem in the rest of this section. Similar to Theorem \ref{theorem:imex-rk-2order-u-a}, we may handle the initial projection error using the property of the Fourier projection \cite{hesthaven2007spectral} and ignore the dependence of $N$ in the rest of the proof. The proof follows the same structure as that of Theorem \ref{theorem:imex-rk-2order-u-a} but is more technical.

\subsection{Local truncation error}
\begin{lem}\label{lem:truncatiom-rk-3rd}
	For a third-order IMEX-RK scheme of type CK with $c_i = \tilde{c}_i, i =1,\cdots, s$, assume it further satisfies condition \eqref{cond:stage-3rd-order} and the initial data is consistent up to order $q\geq 4$. Then
	\begin{equation}\nonumber
		\bm{E}^{(2)} = O(\Delta t^2),\quad \bm{E}^{(i)} = O(\Delta t^3), \quad i = 3, \cdots, s, \quad \bm{E}^{n+1} = O(\Delta t^4),
	\end{equation}
	and similar results hold for their $x-$derivatives up to order $q-4$.
\end{lem}
The proof of this lemma is similar to Lemma \ref{lem:truncatiom-rk-2nd} and thus omitted.

\subsection{Energy estimates for the error}
Starting from equation \eqref{equ:energy-estimate-Vstar}, we do a further change of variable in order to absorb the last term involving $\bm{F}_{k}$ in equation.
We absorb the error vectors by introducing $\bm{U}_{ke\star\star}$ which satisfies
\begin{equation}\label{equ:u-star2-def-1}
	-\Delta t \tilde{\bm{H}} \bm{F}_{k} = (\bm{U}_{ke\star} - \bm{U}_{ke\star\star}) - \mu \bm{H} \sum_{l=1}^m q_{kl}(\bm{U}_{le\star} - \bm{U}_{le\star\star}), 
\end{equation}
Equation \eqref{equ:energy-estimate-Vstar} can be rewritten as 
\begin{equation}\label{equ:u-star2-equ}
	\begin{aligned}
		\bm{U}_{ke\star\star} = U_{ke}^{(1)}\bm{e} - \Delta t\tilde{\bm{H}}\sum_{l=1}^m a_{kl}\partial_x \bm{U}_{le\star\star} + \mu \bm{H} \sum_{l=1}^m q_{kl}\bm{U}_{le\star\star} - (\Delta t)^2 \tilde{\bm{H}}\bm{J}_{k},
	\end{aligned}
\end{equation}
where
\begin{equation*}
    \begin{aligned}
         \Delta t \bm{J}_{k} = \sum_{l=1}^m a_{kl}\partial_x (\bm{U}_{le\star} - \bm{U}_{le\star\star}).
    \end{aligned}
\end{equation*}
Similarly, denote a matrix
\begin{equation}\nonumber
    \bm{V}_{\star} = (\bm{U}_{1e\star} - \bm{U}_{1e\star\star}, \bm{U}_{2e\star} - \bm{U}_{2e\star\star}, \cdots, \bm{U}_{me\star} - \bm{U}_{me\star\star})\in \mathbb{R}^{s\times m}.
\end{equation}
According to relations \eqref{relation-eu} and \eqref{equ:u-star2-def-1}, we can obtain have 
\begin{equation}\nonumber
    vec(\bm{V}_{\star}) = -\Delta t \bm{\Phi} vec(\tilde{\bm{H}} \bm{\bm{F}_{U}}), \qquad \bm{\bm{F}_{U}} = (\bm{F}_{1},\ldots, \bm{F}_{n}).
\end{equation}

\subsubsection{Proof of the second order uniform accuracy}\label{subsubsec:2nd-order-2}

We will first prove the second-order uniform accuracy of the scheme in this subsection, and then improve it to third-order uniform accuracy in Section \ref{subsubsec:3rd-order}.

Multiplying the $\bm{U}_{ke\star\star}$ equation with $\bm{U}_{je\star\star}^T \bm{M}$, we get
\begin{equation}\nonumber
	\begin{aligned}
		\bm{U}_{je\star\star}^T \bm{M} \bm{U}_{ke\star\star} ={}& \bm{U}_{je\star\star}^T \bm{M} U_{ke\star}^{(1)}\bm{e} - \Delta t \bm{U}_{je\star\star}^T \bm{M} \tilde{\bm{H}}\sum_{l=1}^m a_{kl}\partial_x \bm{U}_{le\star\star}\\
        &+ \mu \bm{U}_{je\star\star}^T \bm{M} \bm{H} \sum_{l=1}^m q_{kl}\bm{U}_{le\star\star} - \Delta t^2 \bm{U}_{je\star\star}^T \bm{M} \tilde{\bm{H}} \bm{J}_{k},
	\end{aligned}
\end{equation}
i.e.,
\begin{equation}\nonumber
	\begin{aligned}
		U^{(s)}_{je\star\star} U^{(s)}_{ke\star\star} ={}& U^{n}_{je} U_{ke}^{n} -  \bm{U}_{je\star\star}^T \bm{M}_\star U_{ke\star\star} - \Delta t \bm{U}_{je\star\star}^T \bm{M} \tilde{\bm{H}}\sum_{l=1}^m a_{kl}\partial_x \bm{U}_{le\star\star}\\
        &+ \mu \bm{U}_{je\star\star}^T \bm{M} \bm{H} \sum_{l=1}^m q_{kl}\bm{U}_{le\star\star} - \Delta t^2 \bm{U}_{je\star\star}^T \bm{M} \tilde{\bm{H}} \bm{J}_{k}.
	\end{aligned}
\end{equation}
Multiplying last equation by $\bm{A}_0 = (a_{0jk})$, summing over $j, k$ and integrated in $x$, we obtain
\begin{equation}\nonumber
	\begin{aligned}
		&\int \sum_{j, k=1}^m a_{0jk}  U^{(s)}_{je\star\star}  U^{(s)}_{ke\star\star} \diff x \\
		={}& \int \sum_{j, k=1}^m a_{0jk} U^{n}_{je} U^{n}_{ke} \diff x - \int \sum_{j, k=1}^m a_{0jk} \bm{U}^T_{je\star\star} \bm{M}_\star \bm{U}_{ke\star\star} \diff x \\
		&- \Delta t \int \sum_{j, k=1}^m a_{0jk} \bm{U}^T_{je\star\star} \bm{M}\tilde{\bm{H}}\sum_{i=1}^m a_{kl}\partial_x \bm{U}_{le\star\star} \diff x\\
		&+ \mu \int \sum_{j, k=1}^m a_{0jk} \bm{U}^T_{je\star\star} \bm{M} \bm{H} \sum_{i=1}^m q_{kl} \bm{U}_{le\star\star} \diff x - \Delta t^2 \int \sum_{j, k=1}^m a_{0jk} \bm{U}^T_{je\star\star} \bm{M} \tilde{\bm{H}} \bm{J}_{k} \diff x.
	\end{aligned}
\end{equation}
Similar to the proof of Theorem \ref{theorem:imex-rk-us}, we get
\begin{equation}\label{equ:r-3or-2-1}
	\begin{aligned}
		&\int \sum_{j, k=1}^m U^{(s)}_{je\star} a_{0jk}  U^{(s)}_{ke\star} \diff x \\
		\leq {}& \int \sum_{j, k=1}^m a_{0jk} U^{n}_{je} U^{n}_{ke} \diff x - C_\star \norm{\delta \bm{U}_{e\star\star}}^2 - \Delta t \int \sum_{j, k=1}^m a_{0jk} \bm{U}^T_{je\star\star} M\tilde{\bm{H}}\sum_{i=1}^m a_{kl}\partial_x \bm{U}_{le\star\star} \diff x\\
		&+ \mu \int \sum_{j, k=1}^m a_{0jk} \bm{U}^T_{je\star\star} M H \sum_{i=1}^m q_{kl} \bm{U}_{le\star\star} \diff x - \Delta t^2 \int \sum_{j, k=1}^m a_{0jk} \bm{U}^T_{je\star\star} M \tilde{\bm{H}}\bm{J}_{k}\diff x\\
		\leq & \int \sum_{j, k=1}^m a_{0jk} U^{n}_{je} U^{n}_{ke} \diff x + C\Delta t \norm{U^{n}_{e\star\star}}^2 -  \frac{3 C_\star}{4} \norm{\delta \bm{U}_{e\star\star}}^2 - \Delta t^2 \int \sum_{j, k=1}^m a_{0jk} \bm{U}^T_{je\star\star} M \tilde{\bm{H}}\bm{J}_{k}\diff x,
	\end{aligned}
\end{equation}
where $C_\star$ in \eqref{eq:err-stability-second-term-RHS} and $\bm{J}_{k}$ consists of linear combinations of $\partial_{xx} \bm{E}^{(i)}$ with $O(1)$ coefficients due to Lemma \ref{lemma:muH}.
To treat the terms with $\bm{J}_{k}$, we have the estimate
\begin{equation}\nonumber
	\begin{aligned}
		&(\Delta t)^2 \bigg|\int a_{0jk} U^{(l)}_{je\star\star} \partial_{xx} E_k^{(i)} \diff x \bigg| \\
        \leq {}& \Delta t^2(\frac{1}{\Delta t} \norm{U^{(l)}_{je\star\star}}^2 + \Delta t \norm{\partial_{xx} E_k^{(i)}}^2 ) \\
        \leq {}& C \Delta t \norm{U^{n}_{le}}^2 + C\Delta t\norm{U^{(l)}_{je\star\star}-U^{n}_{le}}^2 + C \Delta t^7,
	\end{aligned}
\end{equation}
since $\partial_{xx}\bm{E}^{(i)} = O(\Delta t^2)$ for every $j$ by Lemma \ref{lem:truncatiom-rk-2nd}. Therefore, absorbing $C\Delta t\norm{U^{(l)}_{je\star\star}-U^{n}_{le}}^2$ by the good term $ -\frac{C_\star}{4} \norm{\delta \bm{U}_{e\star\star}}^2$ for $C \Delta t \leq \frac{C_\star}{4}$, we get
\begin{equation}\label{err:ue-starstar}
	\begin{aligned}
		\int \sum_{j, k=1}^m U^{(s)}_{je\star\star} a_{0jk}  U^{(s)}_{ke\star\star} \diff x 
		\leq  (1+ C\Delta t)\int \sum_{j, k=1}^m a_{0jk} U^{n}_{je} U^{n}_{ke} \diff x - \frac{C_\star}{2} \norm{\delta \bm{U}_{e\star\star}}^2 + O(\Delta t^7).
	\end{aligned}
\end{equation}

Note the $U_{ke}^{n+1}$ equation \eqref{equ:umplus1-star} is
\begin{equation}\nonumber
	\begin{aligned}
		U^{n+1}_{ke} ={}& U_{ke}^{n} - \Delta t \tilde{\bm{b}} \sum_{l=1}^m  a_{kl}\partial_x \bm{U}_{le\star} + \mu  \bm{H}_{s} \sum_{l=1}^m q_{kl}U_{le\star} + E^{n+1}_k \\
		&- \Delta t\tilde{\bm{b}} \bm{F}_{k} + \mu \bm{H}_s \sum_{l=1}^m q_{kl}\bm{V}_{l},
	\end{aligned}
\end{equation}
We rewrite it with $\bm{U}_{ke\star\star}$ as 
\begin{equation}\nonumber
	\begin{aligned}
		U^{n+1}_{ke} ={}& U_{ke}^{n} - \Delta t\sum_{l=1}^m \tilde{\bm{b}} a_{kl}\partial_x \bm{U}_{le\star\star} + \mu  \sum_{l=1}^m \bm{H}_{s} q_{kl}U_{le\star\star} + E^{n+1}_k \\
		&- \Delta t\tilde{\bm{b}} \bm{F}_{k} + \mu\sum_{l=1}^m \bm{H}_s q_{kl}\bm{V}_{l} - (\Delta t)^2 \tilde{\bm{b}} \bm{J}_{k} + \mu\sum_{l=1}^m \bm{H}_s q_{kl}\bm{V}_{l\star}
	\end{aligned}
\end{equation}
where $\bm{H}_{s}$ denotes the last row of the matrix $\bm{H}$ (and similar notation is used for the last row of other matrices).

Subtracting with last rows of the vector equations \eqref{equ:u-star2-equ} of $\bm{U}_{ke\star\star}$, we get
\begin{equation}\nonumber
	\begin{aligned}
		U^{n+1}_{ke} ={}& U_{ke\star\star}^{(s)} - \Delta t (\tilde{\bm{b}}- \tilde{\bm{H}}_s) \sum_{l=1}^m a_{kl}\partial_x \bm{U}_{le\star\star} + E^{n+1}_k - \Delta t \tilde{\bm{b}} \bm{F}_{k} \\
  &+ \mu\sum_{l=1}^m \bm{H}_s q_{kl}\bm{V}_{l} - \Delta t^2 (\tilde{\bm{b}}- \tilde{\bm{H}}_s) \bm{J}_{k} + \mu\sum_{l=1}^m \bm{H}_s q_{kl}\bm{V}_{l\star}.
	\end{aligned}
\end{equation}
Multiplying above equation by $2U^{n+1}_{ie} a_{0ik} $ respectively, summing over $i, k$ and integrating in $x$ gives the energy estimate
\begin{equation}\label{err:rk-um1star}
	\begin{aligned}
		&\int \sum_{i,k=1}^m U^{n+1}_{ie} a_{0ik} U^{n+1}_{ke} \diff x \\
        ={}& \int \sum_{i,k=1}^m U^{(s)}_{ie\star\star} a_{0ik} U_{ke\star\star}^{(s)} \diff x -  \int \sum_{i,k=1}^m (U^{n+1}_{ie}-U^{(s)}_{ie\star\star}) a_{0ik} (U^{n+1}_{ke} - U_{ke\star\star}^{(s)}) \diff x \\
		& - 2\Delta t \int \sum_{i,k=1}^m U^{n+1}_{ie} a_{0ik}  (\tilde{\bm{b}}- \tilde{\bm{H}}_s) \sum_{l=1}^m a_{kl}\partial_x \bm{U}_{le\star\star} \diff x\\
		&+ 2 \int \sum_{i,k=1}^m U^{n+1}_{ie} a_{0ik} E^{n+1}_k \diff x - 2\Delta t \int \sum_{i,k=1}^m U^{n+1}_{ie} a_{0ik} \tilde{\bm{b}} \bm{F}_{k} \diff x\\
		&+ 2 \mu \int \sum_{i,k=1}^m U^{n+1}_{ie} a_{0ik} \sum_{l=1}^m \bm{H}_s q_{kl}\bm{V}_{l} \diff x + 2 \mu \int \sum_{i,k=1}^m U^{n+1}_{ie} a_{0ik} \sum_{l=1}^m \bm{H}_s q_{kl}\bm{V}_{l\star} \diff x \\
        &- 2\Delta t^2 \int \sum_{i,k=1}^m U^{n+1}_{ie} a_{0ik}(\tilde{\bm{b}}- \tilde{\bm{H}}_s) \bm{J}_{k}  \diff x.
	\end{aligned}
\end{equation}
Conducting a similar energy estimate in Theorem \ref{theorem:imex-rk-us}, and adding with the \eqref{err:ue-starstar} give 
\begin{equation}\nonumber
	\begin{aligned}
		&\int \sum_{i,k=1}^m U^{n+1}_{ie} a_{0ik} U^{n+1}_{ke} \diff x \\
		\leq {}& (1+ C\Delta t)\int \sum_{j, k=1}^m a_{0jk} U^{n}_{je} U^{n}_{ke} \diff x - \frac{C_\star}{2} \norm{\delta \bm{U}_{e\star\star}}^2 -  \int \sum_{i,k=1}^m (U^{n+1}_{ie}-U^{(s)}_{ie\star}) a_{0ik} (U^{n+1}_{ke} - U_{ke\star}^{(s)}) \diff x \\
		&+ 2 \int \sum_{i,k=1}^m U^{n+1}_{ie} a_{0ik} E^{n+1}_k \diff x - 2\Delta t \int \sum_{i,k=1}^m U^{n+1}_{ie} a_{0ik} \tilde{\bm{b}} \bm{F}_{k} \diff x\\
		&+ 2 \mu \int \sum_{i,k=1}^m U^{n+1}_{ie} a_{0ik} \sum_{l=1}^m \bm{H}_s q_{kl}\bm{V}_{l} \diff x + 2 \mu \int \sum_{i,k=1}^m U^{n+1}_{ie} a_{0ik} \sum_{l=1}^m \bm{H}_s q_{kl}\bm{V}_{l\star} \diff x \\
        &- 2\Delta t \int \sum_{i,k=1}^m U^{n+1}_{ie} a_{0ik} (\tilde{\bm{b}}- \tilde{\bm{H}}_s) \bm{J}_{k} \diff x + O(\Delta t^7).
	\end{aligned}
\end{equation}

Here we need to estimate the following three terms:
\begin{equation}\nonumber
    \begin{aligned}
       \mathcal{I}_1 &=  - 2\Delta t \int \sum_{i,k=1}^m U^{n+1}_{ie} a_{0ik} \tilde{\bm{b}} \bm{F}_{k} \diff x,\\
       \mathcal{I}_2 &= 2\mu \int \sum_{i,k=1}^m U^{n+1}_{ie} a_{0ik} \sum_{l=1}^m \bm{H}_s q_{kl}\bm{V}_{l} \diff x,\\
       \mathcal{I}_3 &= 2\mu \int \sum_{i,k=1}^m U^{n+1}_{ie} a_{0ik} \sum_{l=1}^m \bm{H}_s q_{kl}\bm{V}_{l\star} \diff x.
    \end{aligned}
\end{equation}
Other terms only contribute $O(\Delta t^7)$ or terms which can be absorbed. 
In the term $\mathcal{I}_1$, note that $\bm{F}_{k} =\sum_{l=1}^m a_{kl}\partial_x \bm{V}_{l}$. We need to analyse the following component 
\begin{equation}\label{formula-bf-bv}
    \begin{aligned}
         \tilde{\bm{b}} \bm{F}_{k}= \sum_{j=1}^s \tilde{b}^{j}F^{j}_k =\sum_{j=1}^s \sum_{l=1}^m\tilde{b}^{j} a_{kl}\partial_x v_{jl}= \sum_{l=1}^m a_{kl}\partial_x \brac{\sum_{j=1}^s\tilde{b}^{j} v_{jl}}.
    \end{aligned}
\end{equation}
Since $vec(\bm{V}) = \bm{\Phi} vec(\bm{E})$ from \eqref{relation-eu}, using the properties of vectorization and Kronecker product operators, we can  deduce 
\begin{equation}\nonumber
    \begin{aligned}
        v_{jl} = \sum_{r=1}^s\sum_{q=1}^m \Phi_{(l-1)s+j,~(q-1)s+r} e_{rq}.
    \end{aligned}
\end{equation}
Thus, we have
\begin{equation}\label{formula-bv-appendix}
    \begin{aligned}
        \sum_{j=1}^s\tilde{b}^{j} v_{jl}= \sum_{r=1}^s\sum_{q=1}^m \sum_{j=1}^s\tilde{b}^{j} \Phi_{(l-1)s+j,~(q-1)s+r} e_{rq}.
    \end{aligned}
\end{equation}
Since $e_{2q}=O(\Delta t^2)$ and $e_{rq}=O(\Delta t^3)(3\leq r\leq s)$ from Lemma \ref{lem:truncatiom-rk-3rd}, we use Lemma \ref{lem:b-phi-2} in Appendix \ref{sec:appendix} to state $\sum_{j=1}^s\tilde{b}^{j} \Phi_{(l-1)s+j,~(q-1)s+2} = 0$ to eliminate the $O(\Delta t^2)$ terms.

The term $\mathcal{I}_1$ contributes only $O(\Delta t^7)$ due to Lemma \ref{lem:b-phi-2}. In fact, using  \eqref{formula-bf-bv} and \eqref{formula-bv-appendix}, we have
\begin{equation}\nonumber
    \begin{aligned}
         \tilde{\bm{b}} \bm{F}_{k}= \sum_{l=1}^m a_{kl}\partial_x \brac{\sum_{j=1}^s\tilde{b}^{j} v_{jl}}= \sum_{l=1}^m a_{kl} \brac{\sum_{r=1}^s\sum_{q=1}^m \sum_{j=1}^s\tilde{b}^{j} \Phi_{(l-1)s+j,~(q-1)s+r} \partial_x e_{rq}}.
    \end{aligned}
\end{equation}
Each term $\tilde{b}^{j} \Phi_{(l-1)s+j,~(q-1)s+r} \partial_x \bm{E}_q^{r}$ is $O(\Delta t^3)$, because the second component of the coefficient vector vanishes due to Lemma \ref{lem:b-phi-2}, and the remaining components satisfy $\partial_x  \bm{E}_q^{r} = O(\Delta t^3)$ by Lemma \ref{lem:truncatiom-rk-3rd}.

The term $\mathcal{I}_2$ gives a contribution of $O(\Delta t^5 \frac{\mu^2}{(1+\mu)^2})$ for the same reason, combining with the fact $\mu \left((\bm{A}_0 \bm{Q}) \otimes \bm{H}\right)\bm{\Phi} = O(\frac{\mu}{1+\mu})$ from Lemma \ref{lemma:muH}. 

Recall that
\begin{equation}\nonumber
    vec(\bm{V}_{\star}) = -\Delta t \bm{\Phi} vec(\tilde{\bm{H}} \bm{\bm{F}_{U}}), \qquad \bm{\bm{F}_{U}} = (\bm{F}_{1},\ldots, \bm{F}_{n}).
\end{equation}
The term $\mathcal{I}_3$ gives a contribution of $O(\Delta t^5 \frac{\mu^2}{(1+\mu)^2})$ owing to Lemma \ref{lem:truncatiom-rk-3rd} and $\mu \left((\bm{A}_0 \bm{Q}) \otimes \bm{H}\right)\bm{\Phi} = O(\frac{\mu}{1+\mu})$ from Lemma \ref{lemma:muH}.

Therefore we finally get 
\begin{equation}\nonumber
	\begin{aligned}
		\int \sum_{i,k=1}^m U^{n+1}_{ie} a_{0ik} U^{n+1}_{ke} \diff x \leq  (1+ C\Delta t)\int \sum_{j, k=1}^m a_{0jk} U^{n}_{je} U^{n}_{ke} \diff x + O(\Delta t^7 + \Delta t^5\frac{\mu^2}{(1+\mu)^2}).
	\end{aligned}
\end{equation}
Using Gronwall's inequality, we get
\begin{equation}\label{err:uim1}
	\begin{aligned}
		&\int \sum_{i,k=1}^m U^{n+1}_{ie} a_{0ik} U^{n+1}_{ke} \diff x \leq C(T)(\Delta t^6 + \Delta t^4\frac{\mu^2}{(1+\mu)^2} + \frac{1}{N^{12}}),
	\end{aligned}
\end{equation}
which implies second-order uniform accuracy and the desired third-order accuracy if $\varepsilon=O(1)$.
Here we use the fact that the error of the initial value is bounded by 
\begin{equation}\nonumber	 
\norm{\bm{U}_{e}^0}^2=\norm{(\bm{U}^0)_{N}-\bm{U}_{in}}^2 \leq  \frac{C}{N^{12}}.
\end{equation}
In the rest of this proof, we assume $\varepsilon\leq 1$, and $\Delta t^4\frac{\mu^2}{(1+\mu)^2}$ is always the worst term in \eqref{err:uim1}.

\subsubsection{Improve to third order uniform accuracy}\label{subsubsec:3rd-order}
Here we improve the error estimate \eqref{err:uim1} to third-order uniform accuracy. The process of proving the third-order uniform accuracy error is similar to that in Subsection \ref{subsubsec:2nd-order}. We simplify this part and only emphasize the similarities and differences.

We start by revisiting \eqref{equ:r-3or-2-1} as follows:
\begin{equation}\label{5equ:estimate-ustarstar-s}
	\begin{aligned}
		&\int \sum_{j, k=m-r+1}^m U^{(s)}_{je\star\star} a_{0jk}  U^{(s)}_{ke\star\star} \diff x \\
    	\leq {}& \int \sum_{j, k=m-r+1}^m a_{0jk} U^{n}_{je} U^{n}_{ke} \diff x - C_\star \sum_{i=m-r+1}^m \norm{\delta U_{ie\star\star}}^2 \\
        &- \Delta t \int \sum_{j, k=1}^m a_{0jk} \bm{U}^T_{je\star\star} \bm{M} \tilde{\bm{H}}\sum_{i=1}^m a_{ki}\partial_x \bm{U}_{ie\star\star} \diff x\\
		&+ \mu \int \sum_{j, k=m-r+1}^m a_{0jk} \bm{U}^T_{je\star\star} \bm{M} \bm{H} \sum_{i=1}^m q_{ki} \bm{U}_{ie\star\star} \diff x - \Delta t^2 \int \sum_{j, k=m-r+1}^m a_{0jk} \bm{U}^T_{je\star} \bm{M} \tilde{\bm{H}}\bm{J}_{k} \diff x.
	\end{aligned}
\end{equation}
The fourth term in \eqref{5equ:estimate-ustarstar-s} can be bounded by
$- \mu C_{\bm{M}\bm{H}} \sum_{i=m-r+1}^m \norm{U_{ie\star\star}^{(s)}}^2$ by Lemma \ref{lemma:M1A-prop}, and this term together with $-C_\star\sum_{i=m-r+1}^m\norm{\delta U_{ie\star\star}}^2$ can estimate the third term in \eqref{5equ:estimate-ustarstar-s} as in \eqref{equ:r-2or-2-1}. The remaining last term can be controlled by $C\Delta t \norm{U_{e}^n}^2 + C\Delta t \norm{U_{e\star\star}^{(i)} - U_{e\star}^{n}}^2 + O(\Delta t^7)$ as we did before.
Therefore we can get the following estimate
\begin{equation}\nonumber
	\begin{aligned}
		&\int \sum_{j, k=m-r+1}^m U^{(s)}_{je\star\star} a_{0jk}  U^{(s)}_{ke\star\star} \diff x \\
        \leq {}& (1+ C\Delta t)\int \sum_{j, k=m-r+1}^m a_{0jk} U^{n}_{je} U^{n}_{ke} \diff x -\frac{C_\star}{2} \sum_{i=m-r+1}^m\norm{\delta U_{ie\star\star}}^2\\
        &-\mu \frac{C_{\bm{M}\bm{H}}}{2} \sum_{j=m-r+1}^m \norm{U_{je\star\star}^{(s)}}^2 + C\frac{\mu}{1+\mu}(\Delta t^6+\frac{1}{N^{12}}).
	\end{aligned}
\end{equation}

Combining the last estimate, \eqref{err:rk-um1star} can be rewritten as
\begin{equation}\nonumber
	\begin{aligned}
		&\int \sum_{i,k=m-r+1}^m U^{n+1}_{ie} a_{0ik} U^{n+1}_{ke} \diff x \\
        \leq {}& (1+ C\Delta t)\int \sum_{j, k=m-r+1}^m a_{0jk} U^{n}_{je} U^{n}_{ke} \diff x - \frac{C_\star}{2}\sum_{i=m-r+1}^m\norm{\delta U_{ie\star\star}}^2 - \mu \frac{C_{\bm{M}\bm{H}} }{2} \sum_{i=m-r+1}^m \norm{U_{ie\star\star}^{(s)}}^2 \\
		&-C\sum_{i=m-r+1}^m \norm{U^{n+1}_{ie}-U^{(s)}_{ie\star\star}}^2 - 2\Delta t \int \sum_{i,k=1}^m U^{n+1}_{ie} a_{0ik}  (\tilde{\bm{b}}- \tilde{\bm{H}}_s) \sum_{l=1}^m A_{kl}\partial_x \bm{U}_{le\star\star} \diff x\\
		&+ 2 \int \sum_{i,k=m-r+1}^m U^{n+1}_{ie} a_{0ik} E^{n+1}_k \diff x - 2\Delta t \int \sum_{i,k=m-r+1}^m U^{n+1}_{ie} a_{0ik} \tilde{\bm{b}} \bm{F}_{k} \diff x\\
		&+ 2 \mu \int \sum_{i,k=m-r+1}^m U^{n+1}_{ie} a_{0ik} \sum_{l=1}^m \bm{H}_s q_{kl}\bm{V}_l \diff x + 2 \mu \int \sum_{i,k=m-r+1}^m U^{n+1}_{ie} a_{0ik} \sum_{l=1}^m \bm{H}_s q_{kl}\bm{V}_{l\star} \diff x \\
       &- 2\Delta t^2 \int \sum_{i,k=m-r+1}^m U^{n+1}_{ie} a_{0ik}(\tilde{\bm{b}}- \tilde{\bm{H}}_s)\bm{J}_{k}  \diff x + C\frac{\mu}{1+\mu}(\Delta t^6+\frac{1}{N^{12}}).
	\end{aligned}
\end{equation}
We can gain a good term $\frac{\mu}{1+\mu}\norm{U^{n+1}_{ie}}^2$, $i=m-r+1,\cdots, m$ out of the good terms $- \mu \frac{C_{\bm{M}\bm{H}} }{2}\sum_{i=m-r+1}^m \norm{U_{ie\star\star}^{(s)}}^2- C \sum_{i=m-r+1}^m \norm{U^{n+1}_{ie}-U^{(s)}_{ie\star\star}}$. 
This helps us improve the worst stiff term estimate as 
\begin{equation}\nonumber
    \begin{aligned}
       \bigg| 2 \mu \int \sum_{i,k=m-r+1}^m U^{n+1}_{ie} a_{0ik} \sum_{l=1}^m \bm{H}_s q_{kl}\bm{V}_l \diff x \bigg| 
        \leq {} C\frac{\mu}{1+\mu}\sum_{i=m-r+1}^m\norm{U^{n+1}_{ie}}^2 + C \frac{\mu}{1+\mu}(\Delta t^6 + \frac{1}{N^{12}}).
    \end{aligned}
\end{equation}
By estimating other terms as we did before, which gives $O(\Delta t^7) $ in total, we obtain
\begin{equation}\nonumber
    \begin{aligned}
       &\int \sum_{i,k=m-r+1}^m U^{n+1}_{ie} a_{0ik} U^{n+1}_{ke} \diff x\\ 
       \leq {}& (1+ C\Delta t)\int \sum_{j, k=m-r+1}^m a_{0jk} U^{n}_{je} U^{n}_{ke} \diff x + \frac{\mu}{1+\mu}(C\Delta t^6 - C\sum_{i=m-r+1}^m \norm{U^{n+1}_{ie}}^2).
    \end{aligned}
\end{equation}
Using the same induction method as in Subsection \ref{subsubsec:2nd-order}, we can get
\begin{equation}\nonumber
    \begin{aligned}
        \norm{U_{ie}^{n+1}}^2 \leq C(T) (\Delta t^6+\frac{1}{N^{12}}), \qquad \forall n\Delta t \leq T, \quad i=m-r+1,\cdots, m.
    \end{aligned}
\end{equation}

The proof of the third-order accuracy for $U_{i}$ for $i = 1, \cdots, m-r$  closely resembles that in \ref{subsubsec:2nd-order}; therefore, we omit the details here.

\section{Numerical tests}\label{sec:numerical-test}
In this section, we numerically verify the accuracy of some IMEX-RK schemes applied to two linearized hyperbolic relaxation systems including the Broadwell model \cite{broadwell1964shock} and the Grad's moment system \cite{grad1949kinetic,cai2014cpam}. 
We consider four IMEX-RK schemes: ARS(2,2,2) \cite{ascher1997}, ARS(2,3,2) \cite{ascher1997}, ARS(4,4,3) \cite{ascher1997}, and BHR(5,5,3)* \cite{boscarino2009siam}. Their double Butcher tableau are also included in Appendix \ref{subsec:some-imex-rk-scheme}.
In all the numerical tests, we adopt the Fourier-Galerkin spectral method for spatial discretization with modes $|k|\leq N$ and fix $N=40$ to ensure that the discretization error in space is much smaller than that in time. The reference solution $\bm{U}_{ref}$ is computed with a much finer time step.

\subsection{Broadwell model}

The Broadwell model is a simplified discrete velocity model for the Boltzmann equation \cite{broadwell1964shock}. It describes a two-dimensional (2D) gas as composed of particles of only four velocities with a binary collision law and spatial variation in only one direction. When looking for one-dimensional solutions of the 2D gas, the evolution equations of the model are given by 
\begin{equation}\nonumber
	\begin{aligned}
		\partial_t f_{+} + \partial_x f_{+} &= -\frac{1}{\varepsilon}(f_{+}f_{-} - f_0^2),\\
		\partial_t f_{-} - \partial_x f_{-} &= -\frac{1}{\varepsilon}(f_{+}f_{-} - f_0^2),\\
		\partial_t f_{0} &= \frac{1}{\varepsilon}(f_{+}f_{-} - f_0^2).
	\end{aligned}
\end{equation}
Here $f_{+}$, $f_{-}$ and $f_0$ denote the particle density function at time $t$, position $x$ with velocity $1$, $-1$ and $0$, respectively, $\varepsilon>0$ is the mean free path. 
Set 
\begin{equation}\nonumber
	\rho = f_{+} + 2f_0 + f_{-}, \quad m = f_{+} - f_{-}, \quad z = f_{+} + f_{-}.
\end{equation}
The Broadwell equations can be rewritten as 
\begin{equation}\nonumber
	\begin{aligned}
		\partial_t \rho + \partial_x m &= 0,\\
		\partial_t m + \partial_x z &= 0,\\
		\partial_t z + \partial_x m &=  \frac{1}{2\varepsilon}( \rho^2 + m^2 - 2\rho z).
	\end{aligned}
\end{equation}

A local Maxwellian is the density function that satisfies $z = \frac{1}{2\rho}(\rho^2 + m^2)$. Considering the linearized version at $\rho_\star = 2,  m_\star=0, z_\star=1 $, we obtain the linearized Broadwell system as follows
\begin{equation}\nonumber
    \begin{aligned}
        \partial_t \bm{U} + \bm{A} \partial_x \bm{U} = \frac{1}{\varepsilon}\bm{Q} \bm{U},
    \end{aligned}
\end{equation}
with 
\begin{equation}\nonumber
    \begin{aligned}
        \bm{U} =  (\rho, ~m, ~z)^T, \qquad \bm{A}=\begin{pmatrix}
            0 & 1 & 0\\
            0 & 0 & 1\\
            0 & 1 & 0\\
        \end{pmatrix}, \qquad \bm{Q}=\begin{pmatrix}
            0 & 0 & 0\\
            0 & 0 & 0\\
            1 & 0 & -2\\
            \end{pmatrix}.
    \end{aligned}
\end{equation}
It has been shown in \cite{yong_singular_1999} that the Broadwell model satisfies the structural stability condition.

In our numerical test, the computational spatial domain is $[-\pi, \pi]$ with periodic boundary conditions and the initial data of $\rho$ and $m$ are given by
\begin{equation}\nonumber
    \begin{aligned}
        \rho(x, 0) = 0.5 \exp(0.3\sin(2x)), \quad  m(x, 0) = \rho(x, 0) \left(0.5 + 0.05\cos(2x) \right), \quad z(x, 0) = \frac{1}{2}\rho(x, 0).
    \end{aligned}
\end{equation}
To avoid the initial layer or prepare the initial data satisfying the conditions of Theorem \ref{thm:regularity-const}, we start the computation from time $T_0=1$ which are computed using BHR(5,5,3)* with a much smaller time step $\Delta t = 10^{-5}$. 
We compute the solution to time $T=2$ and estimate the error of the solutions $\bm{U}_{\Delta t}$ as $\norm{\bm{U}_{\Delta t}-\bm{U}_{ref}}$. 

The numerical results are presented in Figures \ref{fig:Broadwell-RK} and \ref{fig:Broadwell-order}, from which the desired uniform second- and third-order accuracy can be observed for the ARS(2,2,2), ARS(2,3,2) and BHR(5,5,3)* schemes, respectively. Meanwhile, ARS(4,4,3), although being third-order when $\varepsilon=O(1)$ and $\varepsilon\rightarrow 0$, suffers from order reduction in the intermediate regime. The numerical results are in perfect agreement with our theoretical analysis.

\begin{figure}[htbp!]
\centering
\subfigure{\includegraphics[width=0.45\textwidth]{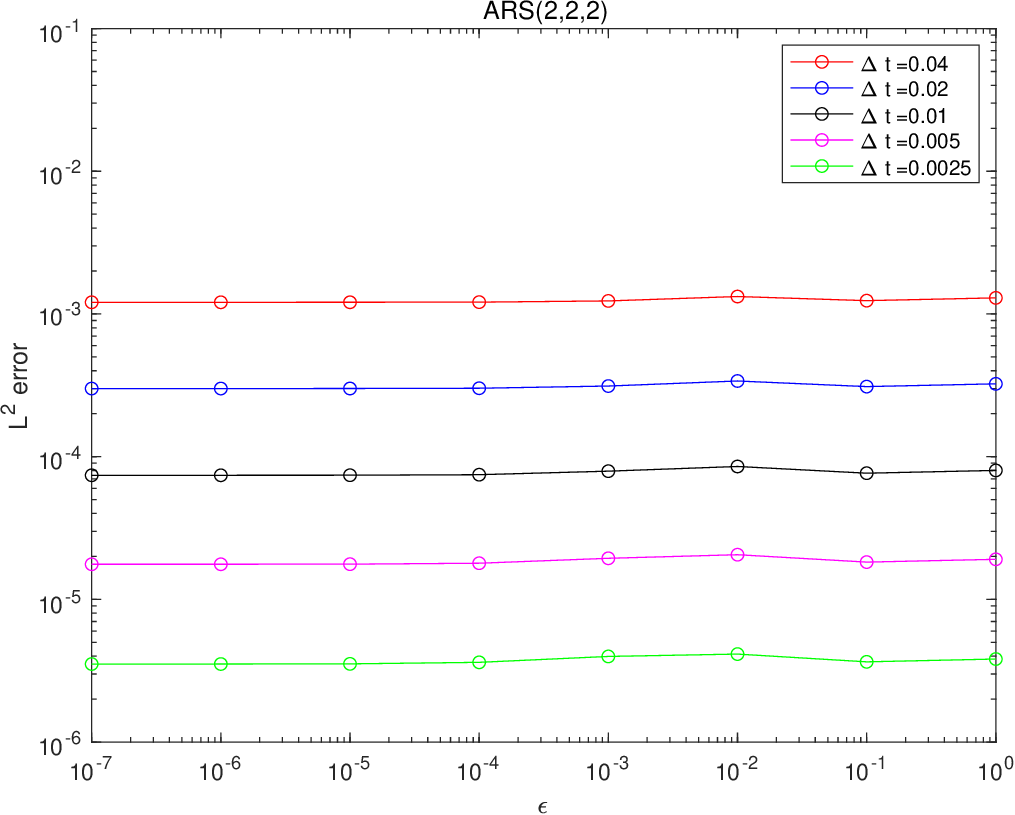}}
\subfigure{\includegraphics[width=0.45\textwidth]{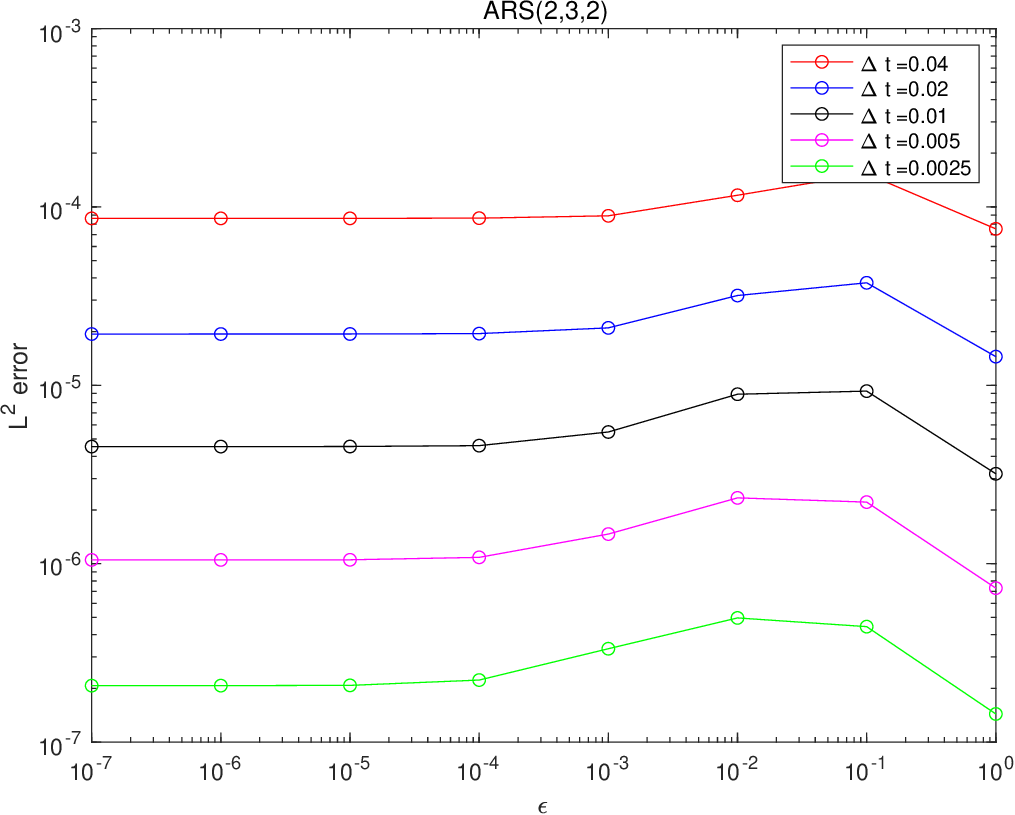}}
\subfigure{\includegraphics[width=0.45\textwidth]{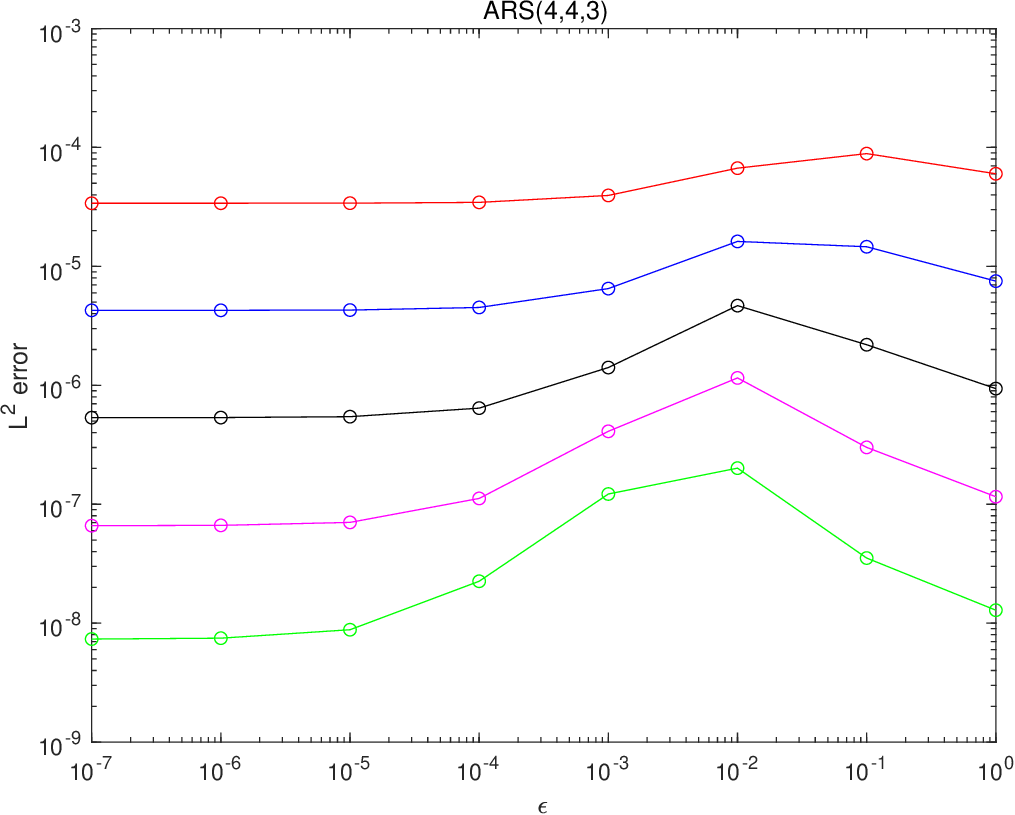}}
\subfigure{\includegraphics[width=0.45\textwidth]{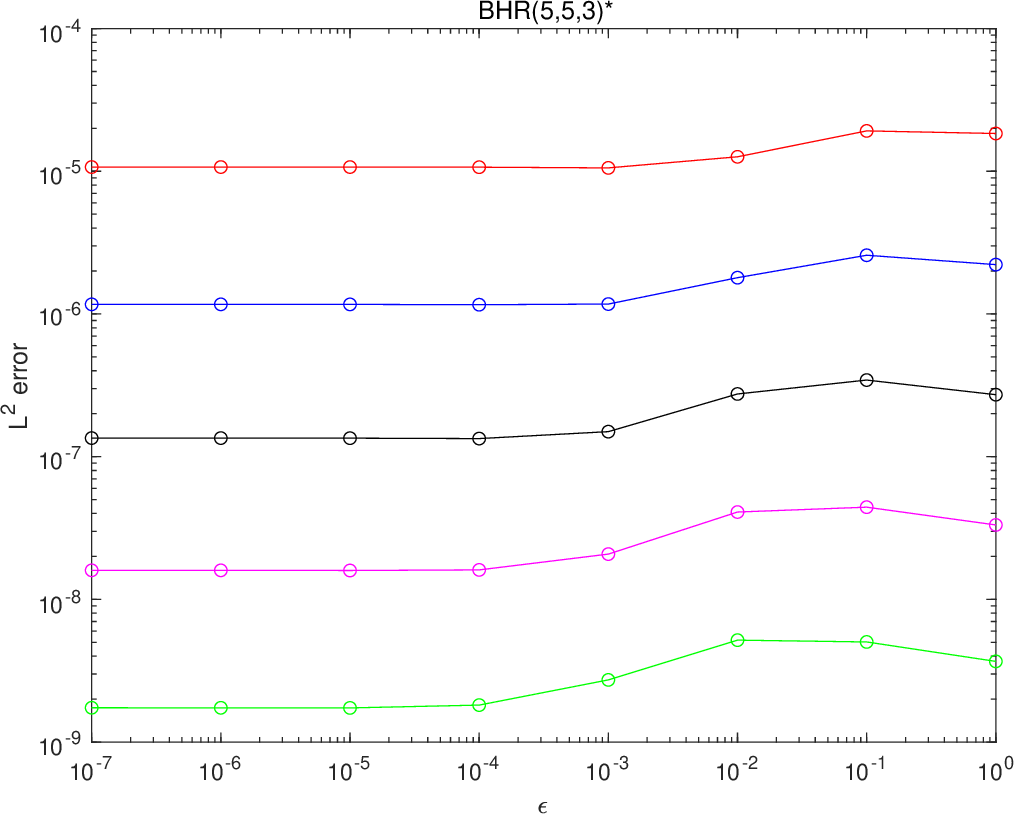}}
\caption{Broadwell system. The $L^2$ error of the solutions computed by IMEX-RK schemes. Top left, top right, bottom left and bottom right figures: ARS(2,2,2), ARS(2,3,2), ARS(4,4,3), BHR(5,5,3)* schemes, respectively. In these four sufigures, horizontal axis ranges from $1e-7$ to $1$, and different curves represent different values of $\Delta t$ as shown in the top left figure. }
\label{fig:Broadwell-RK}
\end{figure}

\begin{figure}[h] 
    \centering 
    \includegraphics[width=0.5\linewidth]{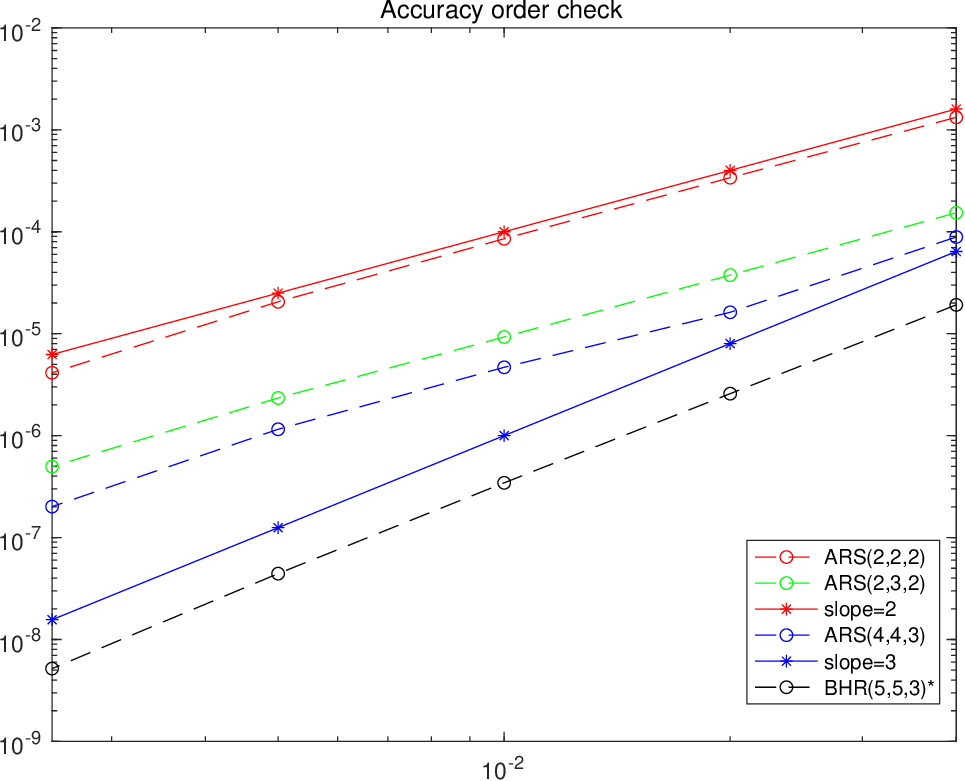} 
    \caption{Broadwell system. The accuracy order for these four IMEX-RK schemes, is obtained as follows: for each scheme, take the maximal $L^2$ error among all values of for a fixed $\Delta t$.} 
    \label{fig:Broadwell-order} 
\end{figure}

\subsection{Linearized Grad's moment system} 
The linearized Grad's moment system in 1D \cite{grad1949kinetic,cai2014cpam,zhao2017stability} reads as
\begin{equation}\label{equ:linear-moment-equ}
	\partial_t \bm{U} + \bm{A}\partial_x \bm{U} = \frac{1}{\varepsilon}\bm{Q}\bm{U}
\end{equation}
with
\begin{equation}\nonumber
	\begin{aligned}
		\bm{U} = \begin{pmatrix}
			\rho \\ w \\ \theta/\sqrt{2} \\ \sqrt{3!}f_3 \\ \vdots \\\sqrt{M!}f_M
		\end{pmatrix},
		\bm{A} = \begin{pmatrix}
			0 & 1 & & & &  \\
			1 & 0 & \sqrt{2} & & & \\
			& \sqrt{2} & 0 & \sqrt{3} & &   \\
			& & \sqrt{3} & 0 & \ddots &   \\
			& & & \ddots & 0 & \sqrt{M} \\
			& & & & \sqrt{M} & 0			
		\end{pmatrix},
		\bm{Q} = -\diag(0, 0, 0, \underbrace{1, \cdots, 1}_{M-2}).
	\end{aligned}
\end{equation}
In the above equation, $\rho$ is the density, $w$ is the macroscopic velocity, $\theta$ is the temperature and $f_3, \cdots, f_M$ with $M\geq 3$ are high order moments. The moment system is obtained by taking moments on both sides of the Bhatnagar-Gross-Krook (BGK) model \cite{Bhatnagar1954511}. 
It was shown in \cite{Di2017nm,zhao2017stability} that the moment system satisfies the structural stability condition.
Here we only consider its linearized version.

The spatial domain is taken as $x\in [-\pi, \pi]$ with periodic boundary conditions.
We solve the linearized Grad’s moment system \eqref{equ:linear-moment-equ} with $M=5$.  
The initial data are prepared by
\begin{equation}\nonumber
	(\rho, ~ w, ~ \theta)(x, 0) = \left( \sin(2x)+1.1, ~ 0, ~ \sqrt{2} \right),\qquad (f_3, f_4, f_5) = (0, ~ 0, ~ 0).
\end{equation}
We compute the solution to time $T=2$ with $\Delta t =\frac{32}{N^2}\cdot 2^{-k}, k=1, 2,\cdots,6$ and estimate the error of the solution $\bm{U}_{\Delta t}$ as $\norm{\bm{U}_{\Delta t}-\bm{U}_{ref}}$. 
The numerical results are in Figure \ref{fig:LinearMoment-RK} and \ref{fig:LinearMoment-order} . We can observe that the numerical results are in perfect agreement with our theoretical analysis. Both the ARS(2,2,2) and ARS(2,3,2) schemes exhibit the uniform second-order accuracy, and the BHR(5,5,3)* scheme achieve the uniform third-order accuracy. The ARS(4,4,3) scheme attains third-order in the regimes $\varepsilon=O(1)$ and $\varepsilon\rightarrow0$, but suffers from order reduction in the intermediate regime.

\begin{figure}[htbp!]
\centering
\subfigure{\includegraphics[width=0.45\textwidth]{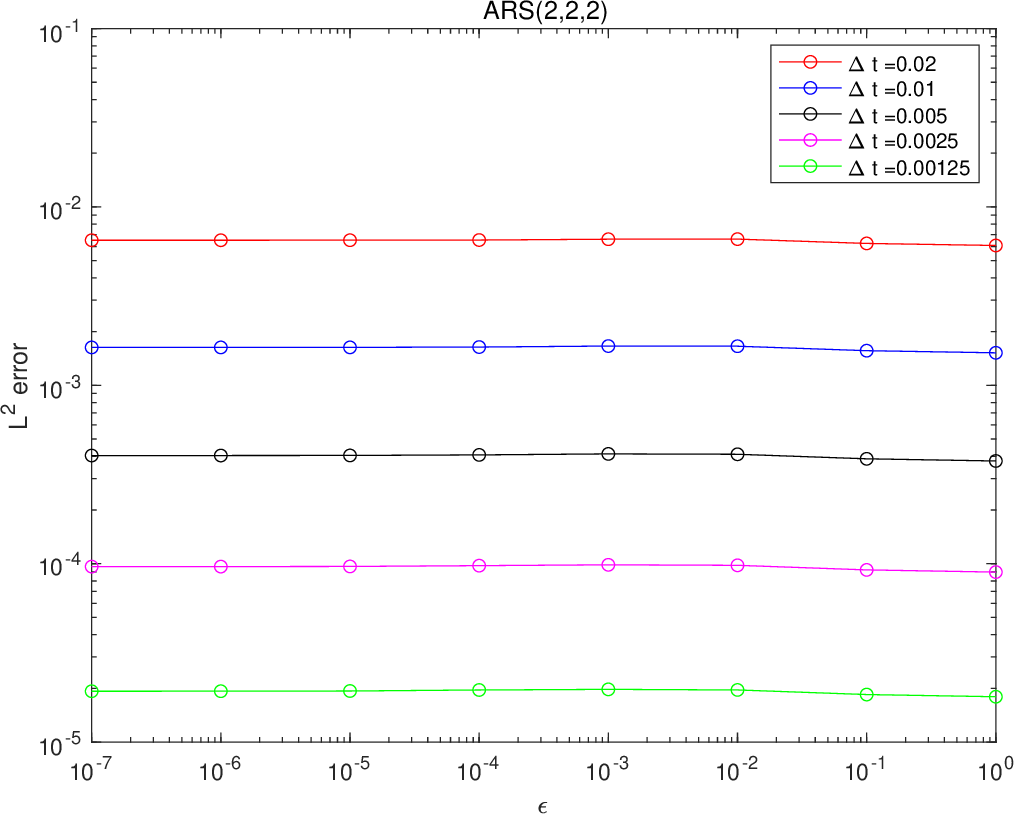}}
\subfigure{\includegraphics[width=0.45\textwidth]{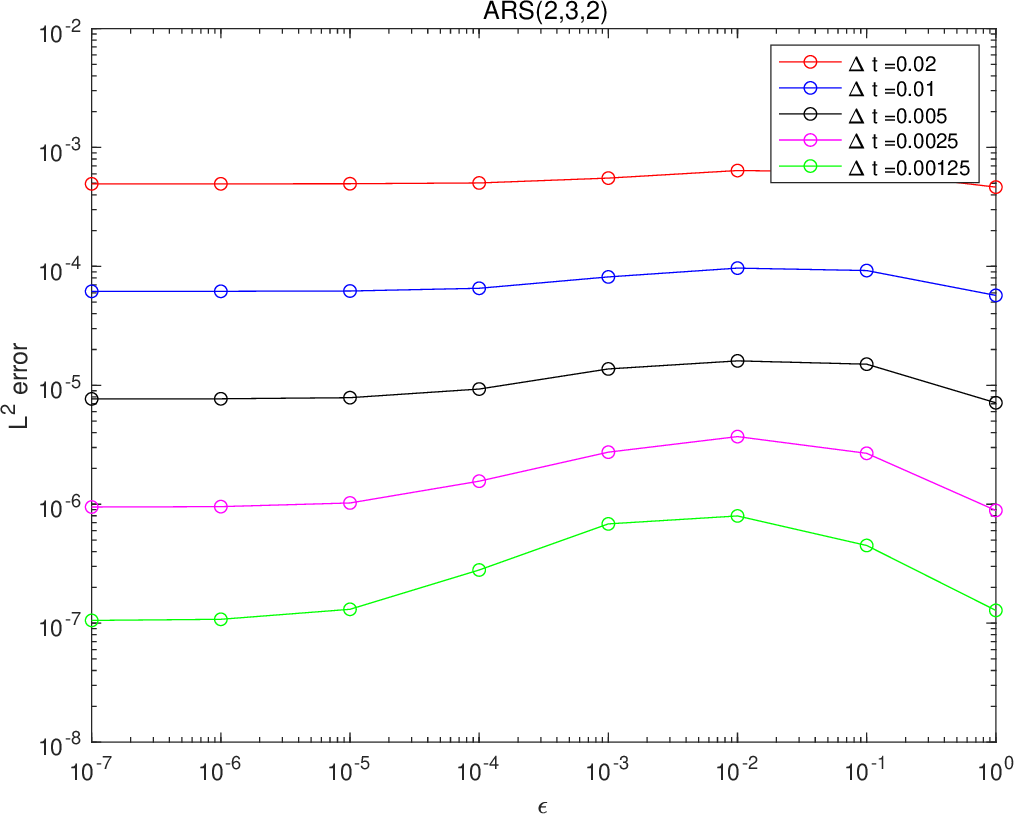}}
\subfigure{\includegraphics[width=0.45\textwidth]{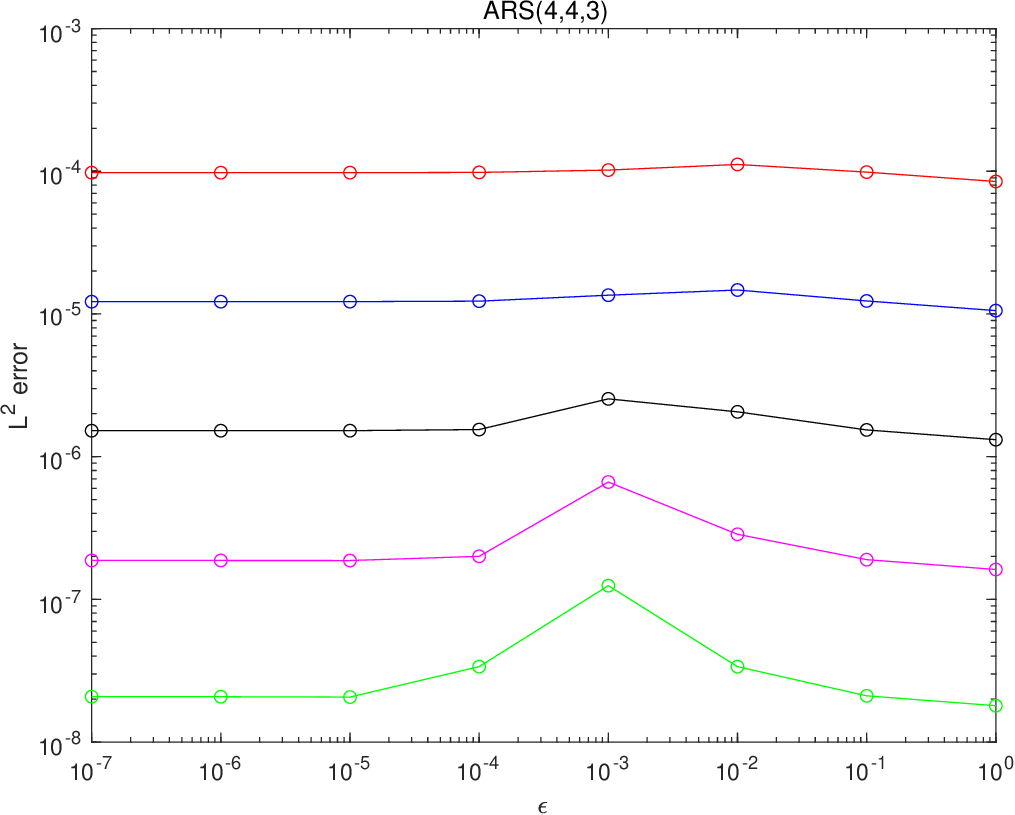}}
\subfigure{\includegraphics[width=0.45\textwidth]{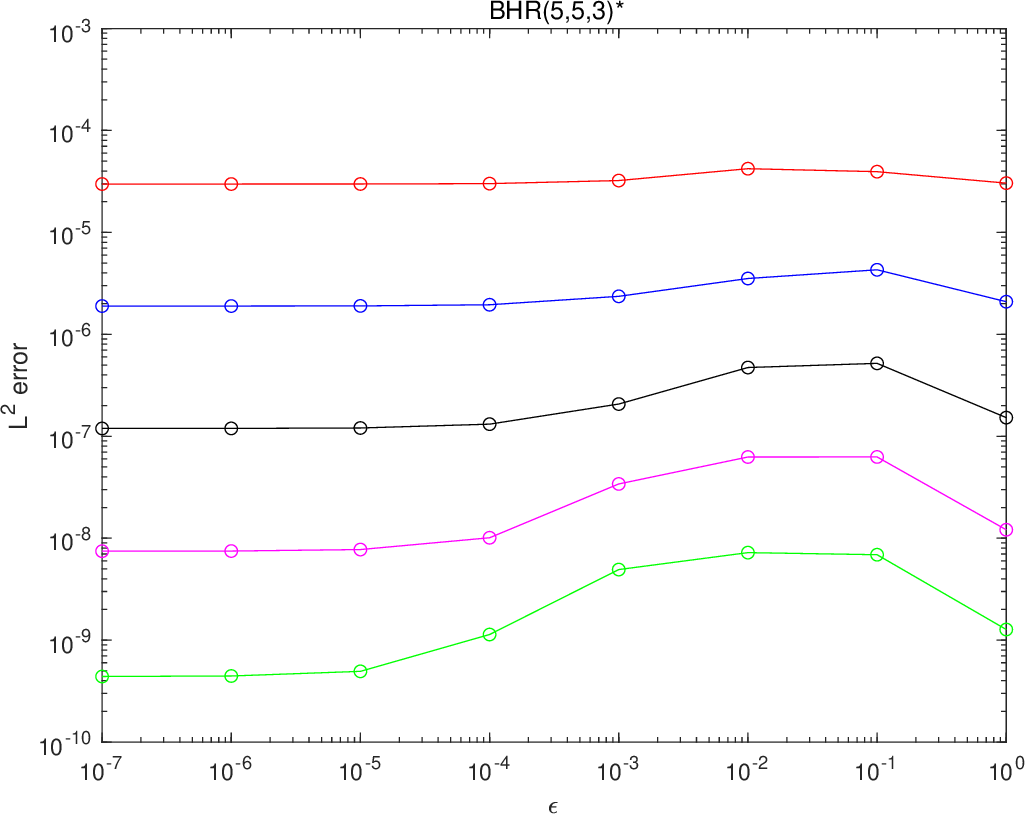}}
\caption{Linearized Grad's moment system. The $L^2$ error of the solutions computed by IMEX-RK schemes. Top left, top right, bottom left and bottom right figures: ARS(2,2,2), ARS(2,3,2), ARS(4,4,3), BHR(5,5,3)* schemes, respectively. In these four sufigures, horizontal axis ranges from $1e-7$ to $1$, and different curves represent different values of $\Delta t$ as shown in the top left figure.}
\label{fig:LinearMoment-RK}
\end{figure}

\begin{figure}[h] 
    \centering 
    \includegraphics[width=0.5\linewidth]{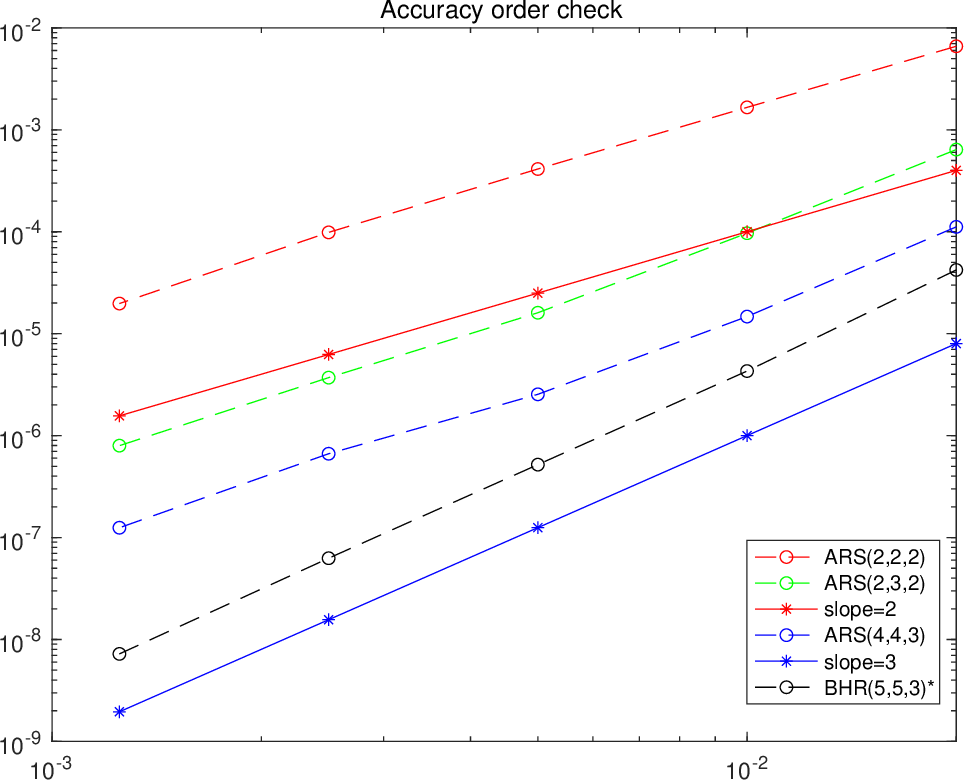} 
    \caption{Linearized Grad's moment system. The accuracy order for these four IMEX-RK schemes, is obtained as follows: for each scheme, take the maximal $L^2$ error among all values of for a fixed $\Delta t$.} 
    \label{fig:LinearMoment-order} 
\end{figure}

\section*{Acknowledgments}

We would like to thank the anonymous reviewers for their invaluable comments and questions which greatly improved the quality of this paper.

\appendix
\section{Appendix}\label{sec:appendix}

\subsection{Lemma \ref{lem:b-phi-2}} \label{appendix:proof_lemma}
For formula \eqref{formula-bv-appendix}, we state the following Lemma.
\begin{lem}\label{lem:b-phi-2}
    The condition \eqref{condition-5.3-vanishing} implies that for all $1\leq l, q \leq m \in \mathbb{N}^+$ and $1\leq j \leq s\in \mathbb{N}^+$, $\sum_{j=1}^s\tilde{b}^{j} \Phi_{(l-1)s+j,~(q-1)s+2} = 0$ and $\sum_{j=1}^sh_{sj} \Phi_{(l-1)s+j,~(q-1)s+2} = 0$ for any $\mu>0$.
\end{lem}
\begin{proof}

Since $\bm{\Phi} = (\bm{G} \otimes \bm{I}_{s})^{-1}\left(\bm{I}_{ms} - \mu (\bm{\Lambda} \otimes \bm{H})\right)^{-1}(\bm{G} \otimes \bm{I}_{s}) \in \mathbb{R}^{ms \times ms}$, we have
\begin{equation}\nonumber
    \begin{aligned}
        \Phi_{(l-1)s+j,~(q-1)s+r} = \sum_{i, k=1}^{ms}(\bm{G} \otimes \bm{I}_{s})^{-1}_{(l-1)s+j, i}\left(\bm{I}_{ms} - \mu (\bm{\Lambda} \otimes \bm{H})\right)^{-1}_{i,k}(\bm{G} \otimes \bm{I}_{s})_{k,(q-1)s+r}.
    \end{aligned}
\end{equation}
Here $\bm{G}\in \mathbb{R}^{m\times m}$ is an invertible matrix and $\bm{\Lambda}\in \mathbb{R}^{m\times m}$ is a diagonal matrix, which satisfy
\begin{equation}\nonumber
    \bm{A}_0 = \bm{G}^T \bm{G}, \qquad \bm{A}_0 \bm{Q} = \bm{G}^T \bm{\Lambda} \bm{G}.
\end{equation}

Therefore, we have
\begin{equation}\nonumber
    \begin{aligned}
        &\sum_{j=1}^s\tilde{b}^{j} \Phi_{(l-1)s+j,~(q-1)s+r} \\
        ={}& \sum_{j=1}^s \sum_{i, k=1}^{ms}\tilde{b}^{j} (\bm{G} \otimes \bm{I}_{s})^{-1}_{(l-1)s+j, i}\left(\bm{I}_{ms} - \mu (\bm{\Lambda} \otimes \bm{H})\right)^{-1}_{i,k}(\bm{G} \otimes \bm{I}_{s})_{k,(q-1)s+r}\\
        ={}& \sum_{j, z, \beta=1}^s\sum_{y,\alpha=1}^m \tilde{b}^{j} (\bm{G} \otimes \bm{I}_{s})^{-1}_{(l-1)s+j, (y-1)s+z}\left(\bm{I}_{ms} - \mu (\bm{\Lambda} \otimes \bm{H})\right)^{-1}_{(y-1)s+z,(\alpha-1)s+\beta}(\bm{G} \otimes \bm{I}_{s})_{(\alpha-1)s+\beta,(q-1)s+r}\\
        ={}& \sum_{j, z, \beta=1}^s\sum_{y,\alpha=1}^m \tilde{b}^{j} G^{-1}_{l,y} \delta_{j}^{z}\left(\bm{I}_{ms} - \mu (\bm{\Lambda} \otimes \bm{H})\right)^{-1}_{(y-1)s+z,(\alpha-1)s+\beta} G_{\alpha,q} \delta_{\beta}^{r}\\
        ={}& \sum_{j=1}^s\sum_{y,\alpha=1}^m \tilde{b}^{j} G^{-1}_{l,y} \left(\bm{I}_{ms} - \mu (\bm{\Lambda} \otimes \bm{H})\right)^{-1}_{(y-1)s+j,(\alpha-1)s+r} G_{\alpha,q} .
    \end{aligned}
\end{equation}
In the fourth equality, we denote $i=(y-1)s+z, k=(\alpha-1)s+\beta$ and use the property \eqref{prop: Kronecker-product-1}
\begin{equation}\nonumber
    (\bm{A}\otimes \bm{B})_{(\alpha-1)s+\beta,(q-1)s+r} = a_{\alpha, q}b_{\beta, r}.
\end{equation}
According to \eqref{equ:phi-ana-form}, we have
 \begin{equation}\nonumber
        \begin{aligned}
            (\bm{I}_{ms}  - \mu \bm{\Lambda}\otimes \bm{H})^{-1} = \brac{\bm{I}_{ms}  + \sum_{i=1}^{ms-1}\brac{\mu\bm{D}^{-1}(\bm{\Lambda}\otimes \bm{K})}^i } \bm{D}^{-1}.
        \end{aligned}
\end{equation}
Here $\bm{D} = \bm{I}_{ms}  - \mu \bm{\Lambda} \otimes \tilde{\bm{D}}$ with $\tilde{\bm{D}} = \diag\{0, h_{22}, \cdots, h_{ss}\}$ and $\bm{K}$ as the diagonal and strictly lower triangular parts of $\bm{H}$. 
Since $\bm{\Lambda}$ and $\tilde{\bm{D}}$ are diagonal matrices, as stated in the proof of Lemma \ref{lemma:muH}, we can obtain
\begin{equation}\nonumber
    \begin{aligned}
           \brac{\bm{D}^{-1}(\bm{\Lambda}\otimes \bm{K})}_{(y-1)s+j,(p-1)s+\beta}= \sum_{q=1}^m\sum_{t=1}^s \frac{\delta^q_y \delta^t_j \Lambda_{p,p}\delta_q^p k_{t, \beta}}{1 - \mu \Lambda_{y,y} \tilde{d}_{j,j}}= \frac{\Lambda_{p,p}\delta_y^p k_{j, \beta}}{1 - \mu \Lambda_{y,y} \tilde{d}_{j,j}}
    \end{aligned}
\end{equation}
and 
\begin{equation}\nonumber
    \begin{aligned}
            (\bm{D}^{-1})_{(p-1)s+\beta,(\alpha-1)s+r}= \frac{\delta_p^\alpha \delta_\beta^r}{1 - \mu \Lambda_{p,p}  \tilde{d}_{r,r}}.
    \end{aligned}
\end{equation}
Therefore, we have
\begin{equation}\nonumber
    \begin{aligned}
        &\sum_{j=1}^s\tilde{b}^{j} \Phi_{(l-1)s+j,~(q-1)s+r}\\
        ={}&\sum_{j=1}^s\sum_{y,\alpha=1}^m \tilde{b}^{j} G^{-1}_{l,y} \left(\bm{I}_{ms} - \mu (\bm{\Lambda} \otimes \bm{H})\right)^{-1}_{(y-1)s+j,(\alpha-1)s+r} G_{\alpha,q}\\
        ={}&\sum_{j=1}^s\sum_{y,\alpha=1}^m G^{-1}_{l,y} \tilde{b}^{j} \brac{\delta^j_r +\sum_{i=1}^{ms-1}\brac{\frac{\Lambda_{p,p} k_{j, r}}{1 - \mu \Lambda_{y,y} \tilde{d}_{j,j}}}^i} \frac{\delta^y_\alpha }{1 - \mu \lambda_{p,p}\tilde{d}_{r,r}} G_{\alpha,q}\\
        ={}&\sum_{j=1}^s\sum_{y=1}^m G^{-1}_{l,y} \tilde{b}^{j} \brac{\delta^j_r +\sum_{i=1}^{ms-1}\brac{\frac{\Lambda_{p,p} k_{j, r}}{1 - \mu \Lambda_{y,y} \tilde{d}_{j,j}}}^i} \frac{G_{y,q}}{1 - \mu \lambda_{p,p}\tilde{d}_{r,r}} \\
        ={}&\sum_{y=1}^m G^{-1}_{l,y} \brac{\tilde{b}^{r} +\sum_{j=1}^s\tilde{b}^{j} \sum_{i=1}^{ms-1}\brac{\frac{\Lambda_{p,p} k_{j, r}}{1 - \mu \Lambda_{y,y} \tilde{d}_{j,j}}}^i} \frac{G_{y,q}}{1 - \mu \lambda_{p,p}\tilde{d}_{r,r}}.
    \end{aligned}
\end{equation}
Due to conditions \eqref{condition-5.3-vanishing}, we know that $\tilde{b}^2=0$ and $k_{j,2}=h_{j,2}=0, j=1,\cdots, s$. When $r=2$, we get
\begin{equation}\nonumber
    \begin{aligned}
        \sum_{j=1}^s\tilde{b}^{j} \Phi_{(l-1)s+j,~(q-1)s+2}
        ={}&\sum_{y=1}^m G^{-1}_{l,y} \brac{\tilde{b}^{2} +\sum_{j=1}^s\tilde{b}^{j} \sum_{i=1}^{ms-1}\brac{\frac{\Lambda_{p,p} k_{j, 2}}{1 - \mu \Lambda_{y,y} \tilde{d}_{j,j}}}^i} \frac{G_{y,q}}{1 - \mu \lambda_{p,p}\tilde{d}_{2,2}} \equiv 0 \\
    \end{aligned}
\end{equation}
Therefore the coefficients in $\tilde{b}^{j} v_{jl}$ which involved $e_{2q}$ are all zero. Replacing $\tilde{b}^{j}$ with $h_{sj}$ yields the same result. Hence we get the conclusion.
\end{proof}

\subsection{Some IMEX-RK schemes}\label{subsec:some-imex-rk-scheme}
In this part, we present some examples of type CK IMEX-RK schemes. They are also numerically tested in Section \ref{sec:numerical-test}.
\begin{itemize}
    \item $\operatorname{ARS}(2,2,2)$ \cite{ascher1997}:
\begin{equation}\nonumber
	\begin{tabular}{l|lll}
0 & 0 & 0 & 0 \\
$\gamma$ & $\gamma$ & 0 & 0 \\
1 & $\delta$ & $1-\delta$ & 0 \\
\hline & $\delta$ & $1-\delta$ & 0
\end{tabular}
\qquad 
\begin{tabular}{l|lll}
0 & 0 & 0 & 0 \\
$\gamma$ & 0 & $\gamma$ & 0 \\
1 & 0 & $1-\gamma$ & $\gamma$ \\
\hline & 0 & $1-\gamma$ & $\gamma$
\end{tabular}
\end{equation}
$\gamma=1-\frac{\sqrt{2}}{2} \approx 0.292893218813452$, \quad $\delta=1-\frac{1}{2 \gamma} \approx-0.707106781186548$.

 \item $\operatorname{ARS}(2,3,2)$ \cite{ascher1997}:
\begin{equation}\nonumber
	\begin{tabular}{l|lll}
0 & 0 & 0 & 0 \\
$\gamma$ & $\gamma$ & 0 & 0 \\
1 & $\delta$ & $1-\delta$ & 0 \\
\hline & $ 0 $ & $ 1-\gamma $ & $\gamma$
\end{tabular}
\qquad 
\begin{tabular}{l|lll}
0 & 0 & 0 & 0 \\
$\gamma$ & 0 & $\gamma$ & 0 \\
1 & 0 & $1-\gamma$ & $\gamma$ \\
\hline & 0 & $1-\gamma$ & $\gamma$
\end{tabular}
\end{equation}
$\gamma=1-\frac{\sqrt{2}}{2} \approx 0.292893218813452$, \quad $\delta=-2\sqrt{2}/3  \approx-0.942809041582063$.

\item $\operatorname{ARS}(4,4,3)$ \cite{ascher1997}:
\begin{equation}\nonumber
\begin{tabular}{l|lllll}
0 & 0 & 0 & 0 & 0 & 0 \\
$1 / 2$ & $1 / 2$ & 0 & 0 & 0 & 0 \\
$2 / 3$ & $11 / 18$ & $1 / 18$ & 0 & 0 & 0 \\
$1 / 2$ & $5 / 6$ & $-5 / 6$ & $1 / 2$ & 0 & 0 \\
1 & $1 / 4$ & $7 / 4$ & $3 / 4$ & $-7 / 4$ & 0 \\
\hline & $1 / 4$ & $7 / 4$ & $3 / 4$ & $-7 / 4$ & 0
\end{tabular}
\qquad 
\begin{tabular}{l|lllll}
0 & 0 & 0 & 0 & 0 & 0 \\
$1 / 2$ & 0 & $1 / 2$ & 0 & 0 & 0 \\
$2 / 3$ & 0 & $1 / 6$ & $1 / 2$ & 0 & 0 \\
$1 / 2$ & 0 & $-1 / 2$ & $1 / 2$ & $1 / 2$ & 0 \\
1 & 0 & $3 / 2$ & $-3 / 2$ & $1 / 2$ & $1 / 2$ \\
\hline & 0 & $3 / 2$ & $-3 / 2$ & $1 / 2$ & $1 / 2$
\end{tabular}
\end{equation}
\item $\operatorname{BHR}(5,5,3)^*$ (a variant of the $\operatorname{BHR}(5,5,3)$ scheme in \cite{boscarino2009siam}):
\begin{equation}\nonumber
\begin{tabular}{c|ccccc}
0 & 0 & 0 & 0 & 0 & 0 \\
$2 \gamma$ & $2 \gamma$ & 0 & 0 & 0 & 0 \\
$2 \gamma$ & $\gamma$ & $\gamma$ & 0 & 0 & 0 \\
$c_4$ & $c_4-\frac{c_4^2}{4 \gamma}$ & 0 & $\frac{c_4^2}{4 \gamma}$ & 0 & 0 \\
1 & $1+b_3-\tilde{a}_{53}-\tilde{a}_{54}$ & $-b_3$ & $\tilde{a}_{53}$ & $\tilde{a}_{54}$ & 0 \\
\hline & $1-b_3-b_4-\gamma$ & 0 & $b_3$ & $b_4$ & $\gamma$
\end{tabular} 
\qquad 
\begin{tabular}{c|ccccc}
0 & 0 & 0 & 0 & 0 & 0 \\
$2 \gamma$ & $\gamma$ & $\gamma$ & 0 & 0 & 0 \\
$2 \gamma$ & $\gamma$ & 0 & $\gamma$ & 0 & 0 \\
$c_4$ & $\frac{3 c_4}{2}-\frac{c_4^2}{4 \gamma}-\gamma$ & 0 & $\frac{c_4^2}{4 \gamma}-\frac{c_4}{2}$ & $\gamma$ & 0 \\
1 & $1-b_3-b_4-\gamma$ & 0 & $b_3$ & $b_4$ & $\gamma$ \\
\hline & $1-b_3-b_4-\gamma$ & 0 & $b_3$ & $b_4$ & $\gamma$
\end{tabular}
\end{equation}
\begin{equation}\nonumber
\begin{aligned}
    &\gamma=0.435866521508460, \quad c_4 =1.5, \quad  b_3=0.362863385578740,\\
    &\tilde{a}_{53}=1.195970114894582, \quad \tilde{a}_{54}=-0.150831109536248, \quad b_4=-0.168124349878957.
\end{aligned}
\end{equation}
\end{itemize}

\section*{Data Availability}

Data will be made available upon reasonable request.

\section*{Conflict of interest}

The authors declare that they have no conflict of interest.

\end{document}